\documentclass{article}
\usepackage{graphicx} 

\usepackage{authblk}

\usepackage{geometry}
\usepackage{amsmath}
\usepackage{comment}
\usepackage{leftidx}
\usepackage{mathtools}
\usepackage{multicol}
\usepackage{amssymb}

\usepackage{stmaryrd}
\usepackage{accents}
\usepackage{mathrsfs}
\usepackage{xfrac}
\usepackage{amsthm}
\usepackage{xcolor}
\usepackage{sectsty}
\usepackage{relsize}
\usepackage{float}

\usepackage{Mypackage}

\usepackage[all]{xy}
\usepackage[T1]{fontenc}
\usepackage[utf8]{inputenc}

\usepackage{indentfirst}
\usepackage{tikz}
\tikzset{shorten <>/.style={shorten >=#1,shorten <=#1}}

\usepackage{tikz-cd}
\newcounter{nodemaker}
\tikzset{Rightarrow/.style={double equal sign distance,>={Implies},->},
triple/.style={-,preaction={draw,Rightarrow}},
quadruple/.style={preaction={draw,Rightarrow,shorten >=0pt},shorten >=1pt,-,double,double
distance=0.2pt}}

\tikzset{%
    symbol/.style={%
        draw=none,
        every to/.append style={%
            edge node={node [sloped, allow upside down, auto=false]{$#1$}}}
    }
}
 \usepackage{quiver}
\usetikzlibrary{nfold}

\usepackage{hyperref}
\hypersetup{
    colorlinks = true,
    linkbordercolor = {red},
   	linkcolor ={teal},
	anchorcolor = {pink},
	citecolor =  {orange},
	filecolor = {teal},
	menucolor = {teal},
	runcolor =  {teal},
	urlcolor = {teal},
}

\usepackage{cleveref}
\newtheorem{theorem}{Theorem}[subsection]
\newtheorem*{theorem*}{Theorem}
\newtheorem{proposition}[theorem]{Proposition}
\newtheorem{corollary}[theorem]{Corollary}
\newtheorem{corollary'}[theorem]{Corollary}

\newtheorem{lemma}[theorem]{Lemma}
\theoremstyle{definition}
\newtheorem{definition}[theorem]{Definition}
\newtheorem{example}[theorem]{Example}
\theoremstyle{definition}
\newtheorem{remark}[theorem]{Remark}
\theoremstyle{definition}
\newtheorem{division}[theorem]{}
\theoremstyle{definition}

\theoremstyle{definition}

\theoremstyle{definition}

\usepackage{imakeidx}
\usepackage{authblk}

\usepackage[backend=biber,
sorting=nty
]{biblatex}
\title{Morphisms and comorphisms of sites I \\
Double-categories of sites}
\author[1]{Olivia Caramello}
\author[1]{Axel Osmond \footnote{Corresponding author. axelosmond@orange.fr}}
\affil[1]{Istituto Grothendieck, Mondovì}
\date{May 2025}

\begin{document}

\maketitle

\begin{abstract}
    We arrange morphisms and comorphisms of sites as the horizontal and vertical cells of a double category of sites; using the formalism of extensions and restrictions of presheaves, we explain how one can define a sheafification double functor from this double category to the quintet double category of Grothendieck topoi. We describe properties of this double functor and recover some classical results of topos theory through a new notion of \emph{locally exact square}, generalizing exact squares in the presence of topologies. We also describe a copointed endo-2-functor on $\Cat$ for which lax morphisms of coalgebras are morphisms of sites and colax morphisms are comorphisms of sites, explaining the arrangement as a double category. This endo-2-functor can be turned into a lax-comonad for which sites are lax-coalgebras, while an indexed version produce a 2-comonad whose strict coalgebras provide a good notion of sites. 
\end{abstract}

\tableofcontents

\newpage

\section*{Introduction}

Geometric morphisms can be induced either from morphisms or comorphisms of sites, respectively in a contravariant and a covariant way; the first are characterized through a cover-preservation property (aside flatness), the second through a cover-lifting property. As both define a relevant notion of 1-cells between sites, two questions:\begin{itemize}
    \item is there a proper way to mix them altogether in a single categorical structure on sites, and if so, does it help understanding the reason for which we have those twin classes of functors rather than a single one?
    \item is it possible to subsume them into a single notion jointly generalizing the cover-preservation and cover-lifting into a single condition?
\end{itemize}

In this first part, we will explain how morphisms and comorphisms, though they do not compose with each other, can be arranged as the horizontal and vertical 1-cells of a double-category of sites. Relying on the formalism of extension-restriction applied to sieves seen as presheaves, we will then show how the sheaf-topos construction defines a horizontally contravariant, vertically covariant double-functor to the quintet double-category of topoi and discuss a few properties of this double-functor, as its relation to tabulators, or a conjoints-to-companions phenomenon.

Double-categories are a good environment to manipulate one the most expressive gadgets of category theory, the so called \emph{exact squares} introduced by \cite{guitart1980carrésexact}: those are lax squares whose corresponding extension-restriction square is invertible; in the context of sites, one can speak more generally of \emph{locally exact squares}, those that are sent to an invertible double cell by the sheafification double-functor. After giving an intrinsic characterization of those locally exact squares in terms of \emph{relative cofinality} à la \cite{caramello2020denseness}, will discuss a variety of examples in and recover several classical results of topos theory as local exactness conditions of some suited squares.

In a second part, we will discuss a reason behind this double-categorical presentation. It is known since \cite{paregrandismultiple} that a 2-(co)monad (and in fact, any mere copointed endo-2-functor) comes with a canonical double-category of strict (co)algebras, together with lax morphisms of (co)algebras as horizontal maps and colax morphisms as vertical maps. Here, we introduce a certain copointed endo-2-functor sending a category to its \emph{cofree site} containing all possible filters of sieves. Then one can exhibit sites as its normal coalgebras and more crucially, cover-preserving functors as lax-morphisms of coalgebras, and cover-lifting functors as colax-morphisms of coalgebras. Modulo a few adjustments to incorporate flatness, those results give a more conceptual reason for the dichotomy between morphisms and comorphisms of sites. Moreover this endo-2-functor can be endowed with a sturcture of lax-comonad, and in fact is \emph{almost} a comonad; but the comultiplication has some defects and sites only are lax-coalgebras. An indexed version of this construction display better behavior and its strict coalgebras provide a plausible generalization of sites.   

A second upcoming instalment of this work will be devoted to the second question, namely the unification of morphisms and comorphisms into a single notion, namely \emph{continuous distributors} between sites. Though both those works will share some ideas and rely on the formalism of extension and restriction, they go in quite orthogonal directions; moreover, though distributors and double categories often come altogether, the double categorical structures involved in this first instalment are \emph{not} of the kind produced by mixing functors and distributors. To emphasize the independence between those two aspects, we chose to split this program accordingly.

\newpage

\section{Extensions, restrictions, morphisms and comorphisms}

This first section contains prerequisites about the main protagonists of our story; after recalling some elementary categorical technology that will ease our manipulations, we recall the definitions of sites, morphisms and comorphisms as well as the way they induce geometric morphisms between sheaf topoi.

\subsection{Extensions and restrictions}


The present work will rely most on the formalism of \emph{left and right extension} and \emph{restrictions}: those are the three adjoint operations induced by a functor at the level of presheaf categories.

\begin{division}[Nerve and Yoneda extensions]\label{Nerve and Yoneda extensions}
Recall that any functor $ f: \mathcal{C} \rightarrow \mathcal{D}$ with $\mathcal{C}$ small and $\mathcal{D}$ locally small induces a \emph{nerve} functor 
\[\begin{tikzcd}
	{\mathcal{C}} & {\mathcal{D}} \\
	{\widehat{\mathcal{C}}}
	\arrow["f", from=1-1, to=1-2]
	\arrow["{\hirayo_\mathcal{C}}"', hook, from=1-1, to=2-1]
	\arrow[""{name=0, anchor=center, inner sep=0}, "{\mathcal{D}(f,1)}", from=1-2, to=2-1]
	\arrow["{\nu_f}", shorten >=2pt, Rightarrow, from=1-1, to=0]
\end{tikzcd}\]
sending $ d$ in $\mathcal{D}$ to the presheaf $ \mathcal{D}[f,d] : \mathcal{C}^{\op} \rightarrow \Set$, which comes equiped together with a 2-cell $ \nu_f$ whose component at $c$ is the transformation with component at $c'$ given as $f_{c',c} : \mathcal{C}[c',c] \rightarrow \mathcal{D}[f(c'), f(c)] $. On the other hand, if $ \mathcal{D}$ is cocomplete, this nerve functor admits a left adjoint given by the left Kan extension along the Yoneda embedding
\[\begin{tikzcd}
	{\mathcal{C}} & {\mathcal{D}} \\
	{\widehat{\mathcal{C}}}
	\arrow["f", from=1-1, to=1-2]
	\arrow["{\hirayo_\mathcal{C}}"', hook, from=1-1, to=2-1]
	\arrow[""{name=0, anchor=center, inner sep=0}, "{\lan_{\hirayo_\mathcal{C}}f}"', from=2-1, to=1-2]
	\arrow["\simeq"{description}, draw=none, from=1-1, to=0]
\end{tikzcd}\]
This nerve can also be exhibited as a left Kan extension through the relation 
\[ \mathcal{D}(f,1) = \lan_{f} \hirayo_{\mathcal{C}} \]    
\end{division}

\begin{division}[Extensions and restrictions along a functor]\label{Extensions and restrictions}
First recall that any functor $ f : \mathcal{C} \rightarrow \mathcal{D}$ induces a triple of adjoints 
\[\begin{tikzcd}
	{\widehat{\mathcal{C}}} && {\widehat{\mathcal{D}}}
	\arrow[""{name=0, anchor=center, inner sep=0}, "{\lext_f}", curve={height=-18pt}, from=1-1, to=1-3]
	\arrow[""{name=1, anchor=center, inner sep=0}, "{\rext_f}"', curve={height=18pt}, from=1-1, to=1-3]
	\arrow[""{name=2, anchor=center, inner sep=0}, "{\rest_f}"{description}, from=1-3, to=1-1]
	\arrow["\dashv"{anchor=center, rotate=-90}, draw=none, from=0, to=2]
	\arrow["\dashv"{anchor=center, rotate=-90}, draw=none, from=2, to=1]
\end{tikzcd}\]
where $ \lext_f $ (resp. $\rext_f$) sends a presheaf $ X : \mathcal{C}^{\op} \rightarrow \Set$ to its left (resp. right) Kan extension along $ f^{\op}$ as depicted below
\[\begin{tikzcd}
	{\mathcal{C}^{\op}} & \Set \\
	{\mathcal{D}^{\op}}
	\arrow["X", from=1-1, to=1-2]
	\arrow["{f^{\op}}"', from=1-1, to=2-1]
	\arrow[""{name=0, anchor=center, inner sep=0}, "{\lan_{f^{\op}}X}"', from=2-1, to=1-2]
	\arrow["{\zeta_{f^{\op}}}", shorten >=2pt, Rightarrow, from=1-1, to=0]
\end{tikzcd} \hskip1cm \begin{tikzcd}
	{\mathcal{C}^{\op}} & \Set \\
	{\mathcal{D}^{\op}}
	\arrow["X", from=1-1, to=1-2]
	\arrow["{f^{\op}}"', from=1-1, to=2-1]
	\arrow[""{name=0, anchor=center, inner sep=0}, "{\ran_{f^{\op}}X}"', from=2-1, to=1-2]
	\arrow["{\xi_{f^{\op}}}", shorten >=2pt, Rightarrow, from=0, to=1-1]
\end{tikzcd}\]
while the restriction functor $ \rest_f$ sends a presheaf $Y : \mathcal{D}^{\op} \rightarrow \Set$ to the precomposite
\[\begin{tikzcd}
	{\mathcal{C}^{\op}} \\
	{\mathcal{D}^{\op}} & \Set
	\arrow["{f^{\op}}"', from=1-1, to=2-1]
	\arrow["{\rest_f Y}", from=1-1, to=2-2]
	\arrow["Y"', from=2-1, to=2-2]
\end{tikzcd}\]
so that in particular for each $c$ in $\mathcal{C}$ one has 
\[ \rest_f Y(c) = Y(f(c)) \]

In fact both the extensions and restriction functors can also be constructed formally as follows. An $ f : \mathcal{C} \rightarrow \mathcal{D}$ defines two possible nerve functors, the left and right nerve constructed respectively as the composite and the extension
\[\begin{tikzcd}
	{\mathcal{C}} & {\mathcal{D}} \\
	& {\widehat{\mathcal{D}}}
	\arrow["f", from=1-1, to=1-2]
	\arrow["{\mathcal{D}(1,f)}"', from=1-1, to=2-2]
	\arrow["{\hirayo_\mathcal{D}}", from=1-2, to=2-2]
\end{tikzcd} \hskip1cm
\begin{tikzcd}
	{\mathcal{C}} & {\mathcal{D}} \\
	{\widehat{\mathcal{C}}}
	\arrow["f", from=1-1, to=1-2]
	\arrow["{\hirayo_\mathcal{C}}"', hook, from=1-1, to=2-1]
	\arrow[""{name=0, anchor=center, inner sep=0}, "{\mathcal{D}(f,1)}", from=1-2, to=2-1]
	\arrow["{n_f}", shorten >=2pt, Rightarrow, from=1-1, to=0]
\end{tikzcd}\]
where $ \mathcal{D}(1,f)$ sends $c$ to the presheaf $ \mathcal{D}(1, f(c)) = \hirayo_{f(c)}$, while $\mathcal{D}(f,1)$ sends $ d$ in $\mathcal{D}$ to the presheaf $ \mathcal{D}(f, d) : \mathcal{C}^{\op} \rightarrow \Set$, which is actually the left Kan extension $ \mathcal{D}(f,1) = \lan_f \hirayo_{\mathcal{D}}$. But now one can compute either the left Kan extension $ \lan_{\hirayo_\mathcal{C}} D(1,f)$ on one side, while one can compute the left Kan extension of $\mathcal{D}(f,1)$ one the other side. 

\begin{proposition}\label{Formulas for restrictions and lextensions}
    For any functor $ f:  \mathcal{C} \rightarrow \mathcal{D}$ the nerve satisfies the following identities 
    \begin{align*}
    \rest_f &= \lan_{\hirayo_{\mathcal{D}}} \mathcal{D}(f,1) \\
     &=  \widehat{\mathcal{D}}(\hirayo_{\mathcal{D}}f, 1) \\
     &= \lan_{\mathcal{D}(1,f)} \hirayo_{\mathcal{C}}
    \end{align*}
while the left and right extensions can be computed as the extensions
    \[ \lext_f = \lan_{\hirayo_\mathcal{C}} \mathcal{D}(1,f) \hskip1cm \rext_f = \ran_{\hirayo_\mathcal{C}} \mathcal{D}(1,f) \]
 \end{proposition}

\begin{proof}
The first two expressions of the restriction functor are standards; for the third expression, observe that $ \mathcal{C}$ is cocomplete, so using that $ \hirayo_D f = \mathcal{D}(1,f)$ and applying the nerve expression of \cref{Nerve and Yoneda extensions} as an extension gives:
\[ \widehat{\mathcal{D}}(\mathcal{D}(1,f), 1) = \lan_{\mathcal{D}(1,f)} \hirayo_{\mathcal{C}} \]

For the extensions, observe that $\lan_{\mathcal{D}(1,f)} \hirayo_\mathcal{C}$ is right adjoint to the left extension $ \lan_{\hirayo_\mathcal{C}} \mathcal{D}(1,f)$; hence by uniqueness of the left adjoints we have the expression of the left extension. 
\end{proof}

    
\end{division}

\begin{division}[Functoriality]

Extensions and restrictions are respectively pseudo and strictly functorial relative to composition: 
\[ \lext_g\lext_f \simeq \lext_{gf} \hskip1cm \rest_{gf} = \rest_f \rest_g \hskip1cm \rext_g\rext_f \simeq \rext_{gf}\]
where the canonical isomorphisms between extension are provided by their respective universal property. Regarding the 2-dimensional structure, if one has a globular 2-cell as below
\[\begin{tikzcd}
	{\mathcal{C}} && {\mathcal{D}}
	\arrow[""{name=0, anchor=center, inner sep=0}, "f", curve={height=-12pt}, from=1-1, to=1-3]
	\arrow[""{name=1, anchor=center, inner sep=0}, "g"', curve={height=12pt}, from=1-1, to=1-3]
	\arrow["\phi", shorten <=3pt, shorten >=3pt, Rightarrow, from=0, to=1]
\end{tikzcd}\]
then the corresponding restriction functors are related as follows: $ \phi$ induces in a contravariant way a 2-cell between the nerves $ \mathcal{D}(\phi,1) : \mathcal{D}(g,1) \Rightarrow \mathcal{D}(f,1)$ which in turn induces a 2-cell $ \rest_g \Rightarrow \rest_f$ whose component at a presheaf $ Y : \mathcal{D}^{\op} \rightarrow \Set$ is the whiskering $ X* \phi^{\op} : Xg^{\op} \Rightarrow Xf^{\op}$. 

This 2-cell has a mate $ \lext_\phi : \lext_f \Rightarrow \lext_g$ (resp. $ \rext_\phi : \rext_f \Rightarrow \rext_g$), whose component at a presheaf $X$ is the universal 2-cell $ \lan_{f^{\op}} X \Rightarrow  \lan_{g^{\op}} X$ (resp. $ \ran_{f^{\op}} X \Rightarrow  \ran_{g^{\op}} X$) obtained from the universal property of the Kan extension from the composite 2-cell
\[\begin{tikzcd}[row sep=large]
	{\mathcal{C}^{\op}} && \Set \\
	{\mathcal{D}^{\op}}
	\arrow[""{name=0, anchor=center, inner sep=0}, "X", from=1-1, to=1-3]
	\arrow[""{name=1, anchor=center, inner sep=0}, "{f^{\op}}"', curve={height=12pt}, from=1-1, to=2-1]
	\arrow[""{name=2, anchor=center, inner sep=0}, "{g^{\op}}"{description}, curve={height=-12pt}, from=1-1, to=2-1]
	\arrow[""{name=3, anchor=center, inner sep=0}, "{\lan_{g^{\op}}}"', from=2-1, to=1-3]
	\arrow["{\zeta_g}", shorten <=2pt, shorten >=2pt, Rightarrow, from=0, to=3]
	\arrow["{\phi^{\op}}"', shorten <=5pt, shorten >=5pt, Rightarrow, from=2, to=1]
\end{tikzcd} \hskip1cm 
\begin{tikzcd}[row sep=large]
	{\mathcal{C}^{\op}} && \Set \\
	{\mathcal{D}^{\op}}
	\arrow[""{name=0, anchor=center, inner sep=0}, "X", from=1-1, to=1-3]
	\arrow[""{name=1, anchor=center, inner sep=0}, "{g^{\op}}"', curve={height=12pt}, from=1-1, to=2-1]
	\arrow[""{name=2, anchor=center, inner sep=0}, "{f^{\op}}"{description}, curve={height=-12pt}, from=1-1, to=2-1]
	\arrow[""{name=3, anchor=center, inner sep=0}, "{\ran_{f^{\op}}}"', from=2-1, to=1-3]
	\arrow["{\phi^{\op}}", shorten <=5pt, shorten >=5pt, Rightarrow, from=1, to=2]
	\arrow["{\xi_f}"', shorten <=2pt, shorten >=2pt, Rightarrow, from=3, to=0]
\end{tikzcd}\]
  
\end{division}

\begin{division}[The twofold treatment of squares]\label{Adjoint squares}

Now suppose one has a 2-cell of the following form: 
\[\begin{tikzcd}
	{\mathcal{A}} & {\mathcal{B}} \\
	{\mathcal{C}} & {\mathcal{D}}
	\arrow["f", from=1-1, to=1-2]
	\arrow["g"', from=1-1, to=2-1]
	\arrow["\phi"', shorten <=6pt, shorten >=6pt, Rightarrow, from=1-2, to=2-1]
	\arrow["k", from=1-2, to=2-2]
	\arrow["h"', from=2-1, to=2-2]
\end{tikzcd}\]
This 2-cell induces the following squares at the level of the extensions and at the level of the restrictions respectively:
\[\begin{tikzcd}
	{\widehat{\mathcal{A}}} & {\widehat{\mathcal{B}}} \\
	{\widehat{\mathcal{C}}} & {\widehat{\mathcal{D}}}
	\arrow["{\lext_{f}}", from=1-1, to=1-2]
	\arrow["{\lext_{g}}"', from=1-1, to=2-1]
	\arrow["{\lext_{\phi}}"', shorten <=6pt, shorten >=6pt, Rightarrow, from=1-2, to=2-1]
	\arrow["{\lext_{k}}", from=1-2, to=2-2]
	\arrow["{\lext_{h}}"', from=2-1, to=2-2]
\end{tikzcd} \hskip1cm 
\begin{tikzcd}
	{\widehat{\mathcal{A}}} & {\widehat{\mathcal{B}}} \\
	{\widehat{\mathcal{C}}} & {\widehat{\mathcal{D}}}
	\arrow["{\rest_f}"', from=1-2, to=1-1]
	\arrow["{\rest_g}", from=2-1, to=1-1]
	\arrow["{\rest_{\phi}}", shorten <=6pt, shorten >=6pt, Rightarrow, from=2-1, to=1-2]
	\arrow["{\rest_k}"', from=2-2, to=1-2]
	\arrow["{\rest_h}", from=2-2, to=2-1]
\end{tikzcd} \hskip1cm \begin{tikzcd}
	{\widehat{\mathcal{A}}} & {\widehat{\mathcal{B}}} \\
	{\widehat{\mathcal{C}}} & {\widehat{\mathcal{D}}}
	\arrow["{\rext_{f}}", from=1-1, to=1-2]
	\arrow["{\rext_{g}}"', from=1-1, to=2-1]
	\arrow["{\rext_{\phi}}"', shorten <=6pt, shorten >=6pt, Rightarrow, from=1-2, to=2-1]
	\arrow["{\rext_{k}}", from=1-2, to=2-2]
	\arrow["{\rext_{h}}"', from=2-1, to=2-2]
\end{tikzcd}\]

In the following, we are going to consider canonical cross-adjoint squares, as morphisms and comorphisms induce geometric morphisms through extension and restriction functor respectively. There are two possible such cross squares, the one where one takes restriction along $f,h$ and extension along $g,k$, and the one where one takes extension along $f,h$ and restriction along $g,k$ for a same orientation of the 2-cell.

For the first choice, one can construct a canonical 2-cell as the composite on the right
\[\begin{tikzcd}
	{\widehat{\mathcal{A}}} & {\widehat{\mathcal{B}}} \\
	{\widehat{\mathcal{C}}} & {\widehat{\mathcal{D}}}
	\arrow["{\lext_f}", from=1-1, to=1-2]
	\arrow["{\overline{\phi}^\flat}", shorten <=6pt, shorten >=6pt, Rightarrow, from=1-1, to=2-2]
	\arrow["{\rest_g}", from=2-1, to=1-1]
	\arrow["{\lext_h}"', from=2-1, to=2-2]
	\arrow["{\rest_k}"', from=2-2, to=1-2]
\end{tikzcd}\hskip1cm
\begin{tikzcd}
	{\lext_f\rest_g} & {\lext_f\rest_g\rest_h \lext_h} \\
	{\rest_k \lext_h} & {\lext_f\rest_f\rest_k \lext_h}
	\arrow["{\eta_h}", Rightarrow, from=1-1, to=1-2]
	\arrow["{\overline{\phi}^\flat}"', Rightarrow, from=1-1, to=2-1]
	\arrow["{\rest_{\phi}}", Rightarrow, from=1-2, to=2-2]
	\arrow["{\epsilon_f}", Rightarrow, from=2-2, to=2-1]
\end{tikzcd}\]

For the other choice, we can construct a canonical 2-cell as the composite on the right:
\[\begin{tikzcd}
	{\widehat{\mathcal{A}}} & {\widehat{\mathcal{B}}} \\
	{\widehat{\mathcal{C}}} & {\widehat{\mathcal{D}}}
	\arrow["{\rext_g}"', from=1-1, to=2-1]
	\arrow["{\rest_f}"', from=1-2, to=1-1]
	\arrow["{\rext_k}", from=1-2, to=2-2]
	\arrow["{\overline{\phi}^\sharp}"', shorten <=6pt, shorten >=6pt, Rightarrow, from=2-2, to=1-1]
	\arrow["{\rest_h}", from=2-2, to=2-1]
\end{tikzcd} \hskip1cm
\begin{tikzcd}
	{\rext_g\rest_f} & {\rext_g\rest_f\rest_k\rext_k} \\
	{\rest_h\rext_k} & {\rext_g\rest_g\rest_h\rext_k}
	\arrow["\epsilon_k"', Rightarrow, from=1-2, to=1-1]
	\arrow["{\overline{\phi}^\sharp}", Rightarrow, from=2-1, to=1-1]
	\arrow["{\eta_g}"', Rightarrow, from=2-1, to=2-2]
	\arrow["{\rest_\phi}"', Rightarrow, from=2-2, to=1-2]
\end{tikzcd}\]   
\end{division}

\begin{remark}\label{exact squares explicitly}
    Concretely, the induced 2-cell $ \overline{\phi}^\flat : \lext_f\rest_g \Rightarrow \rest_k\lext_h $ has as component at $ X : \mathcal{C}^{\op} \rightarrow \Set$ the 2-cell induced by the universal property of the left Kan extension at the pasting
\[\begin{tikzcd}
	{\mathcal{A}^\op} && {\mathcal{B}^\op} \\
	{\mathcal{C}^\op} && {\mathcal{D}^\op} \\
	& \Set
	\arrow["{f^\op}", from=1-1, to=1-3]
	\arrow[""{name=0, anchor=center, inner sep=0}, "{g^\op}"', from=1-1, to=2-1]
	\arrow[""{name=1, anchor=center, inner sep=0}, "{k^\op}", from=1-3, to=2-3]
	\arrow["{h^\op}"{description}, from=2-1, to=2-3]
	\arrow[""{name=2, anchor=center, inner sep=0}, "X"', from=2-1, to=3-2]
	\arrow[""{name=3, anchor=center, inner sep=0}, "{\lan_{h^\op}X}", from=2-3, to=3-2]
	\arrow["{\phi^\op}", shorten <=13pt, shorten >=13pt, Rightarrow, from=0, to=1]
	\arrow["{\zeta_X}", shorten <=6pt, shorten >=6pt, Rightarrow, from=2, to=3]
\end{tikzcd} = 
\begin{tikzcd}
	{\mathcal{A}^\op} && {\mathcal{B}^\op} \\
	{\mathcal{C}^\op} && {\mathcal{D}^{\op}} \\
	& \Set
	\arrow["{f^\op}", from=1-1, to=1-3]
	\arrow["{g^\op}"', from=1-1, to=2-1]
	\arrow["{k^{\op}}", from=1-3, to=2-3]
	\arrow[""{name=0, anchor=center, inner sep=0}, "{ \lan_{f^\op} (X g^{\op})}"{description}, curve={height=18pt}, from=1-3, to=3-2]
	\arrow["X"', from=2-1, to=3-2]
	\arrow[""{name=1, anchor=center, inner sep=0}, "{\lan_{h^{\op}}X}", shift left, from=2-3, to=3-2]
	\arrow["{\zeta_{X  g^{\op}}}", shorten <=7pt, shorten >=10pt, Rightarrow, from=1-1, to=0]
	\arrow["{\overline{\phi}^\flat_X}"', shorten <=5pt, shorten >=5pt, Rightarrow, from=0, to=1]
\end{tikzcd}\]
while the cell $\overline{\phi^\sharp} $ is dually obtained from the property of the right Kan extension. 
\end{remark}

The construction of the mate is moreover natural in the sense that it is coherent relative to pasting and identity squares, as detailled in the following two lemmas:

\begin{lemma}\label{naturality of mate for pasting}
    Suppose one has two horizontally contiguous squares as below
\[\begin{tikzcd}
	{\cal A} & {\cal B} & {\cal E} \\
	{\cal C} & {\cal D} & {\cal F}
	\arrow["f", from=1-1, to=1-2]
	\arrow[""{name=0, anchor=center, inner sep=0}, "g"', from=1-1, to=2-1]
	\arrow["m", from=1-2, to=1-3]
	\arrow[""{name=1, anchor=center, inner sep=0}, "k"{description}, from=1-2, to=2-2]
	\arrow[""{name=2, anchor=center, inner sep=0}, "l", from=1-3, to=2-3]
	\arrow["h"', from=2-1, to=2-2]
	\arrow["n"', from=2-2, to=2-3]
	\arrow["\phi"', between={0.2}{0.8}, Rightarrow, from=1, to=0]
	\arrow["\psi"', between={0.2}{0.8}, Rightarrow, from=2, to=1]
\end{tikzcd} \]
    then the mate induced by their pasting is related to the mates induced separately as follows:
    \[ \rest_l*\alpha_{nh}^{-1} \overline{(n*\phi)(\psi*f)}^\flat = (\overline{\psi}^\flat*\lext_h)(\lext_m*\overline{\phi}^\flat)\alpha_{mf} \]
    (where $ \alpha_{mf}: \lext_m\lext_f \Rightarrow \lext_{mf}$ and $ \alpha_{nh} : \lext_n\lext_h \Rightarrow \lext_{nh} $ are the comparison isomorphisms between left Kan extensions).
    
    Similarly if one has a two vertically contiguous squares as below
\[\begin{tikzcd}
	{\cal A} & {\cal B} \\
	{\cal C} & {\cal D} \\
	{\cal F} & {\cal E}
	\arrow["f", from=1-1, to=1-2]
	\arrow[""{name=0, anchor=center, inner sep=0}, "g"', from=1-1, to=2-1]
	\arrow[""{name=1, anchor=center, inner sep=0}, "k", from=1-2, to=2-2]
	\arrow["h"{description}, from=2-1, to=2-2]
	\arrow[""{name=2, anchor=center, inner sep=0}, "n"', from=2-1, to=3-1]
	\arrow[""{name=3, anchor=center, inner sep=0}, "m", from=2-2, to=3-2]
	\arrow["l"', from=3-1, to=3-2]
	\arrow["\phi"', between={0.2}{0.8}, Rightarrow, from=1, to=0]
	\arrow["\psi"', between={0.2}{0.8}, Rightarrow, from=3, to=2]
\end{tikzcd}\]
    then the mate induced by their pasting is related to the mates induced separately as follows:
    \[\overline{(\psi*g)(m*\phi)}^\flat = (\rest_k*\overline{\psi}^\flat)( \overline{\phi}^\flat*\rest_n) \]
\end{lemma}

\begin{proof}
     Though this can be proven through a sequence of formal identities, we prefer here to process concretely through diagrams manipulation to understand the concrete meaning of those relations. Take some presheaf $ X : \cal C^\op \rightarrow \Set $; for the horizontal assortment, one can apply the equality above to the resulting pasting to extract first the unique universal 2-cell from the property of left Kan extensions:
\[\begin{tikzcd}
	{\cal A^{\op}} && {\cal B^{\op}} && {\cal E^{\op}} \\
	{\cal C^{\op}} && {\cal D^{\op}} && {\cal F^{\op}} \\
	\Set &&&& {\cal F^{\op}}
	\arrow["{f^{\op}}", from=1-1, to=1-3]
	\arrow[""{name=0, anchor=center, inner sep=0}, "{g^{\op}}"', from=1-1, to=2-1]
	\arrow["{m^{\op}}", from=1-3, to=1-5]
	\arrow[""{name=1, anchor=center, inner sep=0}, "{l^{\op}}", from=1-5, to=2-5]
	\arrow["{h^{\op}}", from=2-1, to=2-3]
	\arrow["X"', from=2-1, to=3-1]
	\arrow["{n^{\op}}", from=2-3, to=2-5]
	\arrow[""{name=2, anchor=center, inner sep=0}, "{\lan_{(nh)^{\op}}X}"{description}, from=2-5, to=3-1]
	\arrow[equals, from=2-5, to=3-5]
	\arrow["{\lan_{n^{\op}}\lan_{h^{\op}}X}"{description}, from=3-5, to=3-1]
	\arrow["{(n*\phi\psi*f)^{\op}}", between={0.2}{0.8}, Rightarrow, from=0, to=1]
	\arrow["{\zeta_ {nh}}"'{pos=0.3}, between={0.2}{0.6}, Rightarrow, from=2-1, to=2]
	\arrow["{\alpha_{nh}^{-1} \atop \simeq}"{pos=0.6}, draw=none, from=2, to=3-5]
\end{tikzcd} = 
\begin{tikzcd}
	{\cal A^{\op}} && {\cal B^{\op}} && {\cal E^{\op}} \\
	{\cal C^{\op}} &&&& {\cal F^{\op}} \\
	\Set &&&& {\cal F^{\op}}
	\arrow["{f^{\op}}", from=1-1, to=1-3]
	\arrow["{g^{\op}}"', from=1-1, to=2-1]
	\arrow["{m^{\op}}", from=1-3, to=1-5]
	\arrow["{l^{\op}}", from=1-5, to=2-5]
	\arrow[""{name=0, anchor=center, inner sep=0}, "{\small{\lan_{(mf)^{\op}}X \circ g^{\op}}}"{description}, curve={height=12pt}, from=1-5, to=3-1]
	\arrow["X"', from=2-1, to=3-1]
	\arrow[""{name=1, anchor=center, inner sep=0}, "{\lan_{(nh)^{\op}}X}"{description}, from=2-5, to=3-1]
	\arrow[equals, from=2-5, to=3-5]
	\arrow["{\lan_{n^{\op}}\lan_{h^{\op}}X}", from=3-5, to=3-1]
	\arrow["{\zeta_{mf}}"'{pos=0.4}, between={0.2}{0.8}, Rightarrow, from=1-1, to=0]
	\arrow["{\small{\overline{(n*\phi\psi*f)}^\flat}}", shift right, between={0.2}{0.8}, Rightarrow, from=0, to=1]
	\arrow["{\alpha_{nh}^{-1} \atop \simeq}"{pos=0.6}, draw=none, from=1, to=3-5]
\end{tikzcd}\]
    But one can also apply successively the same equality to each cell, and then use functoriality of extensions relative to 2-cells:
\[\begin{tikzcd}
	{\cal A^{\op}} && {\cal B^{\op}} && {\cal E^{\op}} \\
	{\cal C^{\op}} && {\cal D^{\op}} && {\cal F^{\op}} \\
	\Set &&&& {\cal F^{\op}}
	\arrow["{f^{\op}}", from=1-1, to=1-3]
	\arrow[""{name=0, anchor=center, inner sep=0}, "{g^{\op}}"', from=1-1, to=2-1]
	\arrow["{m^{\op}}", from=1-3, to=1-5]
	\arrow[""{name=1, anchor=center, inner sep=0}, "{k^{\op}}"{description}, from=1-3, to=2-3]
	\arrow[""{name=2, anchor=center, inner sep=0}, "{l^{\op}}", from=1-5, to=2-5]
	\arrow["{h^{\op}}", from=2-1, to=2-3]
	\arrow[""{name=3, anchor=center, inner sep=0}, "X"', from=2-1, to=3-1]
	\arrow["{n^{\op}}", from=2-3, to=2-5]
	\arrow[""{name=4, anchor=center, inner sep=0}, "{\lan_{h^{\op}}X}"{description, pos=0.2}, from=2-3, to=3-1]
	\arrow[""{name=5, anchor=center, inner sep=0}, equals, from=2-5, to=3-5]
	\arrow["{\lan_{n^{\op}}\lan_{h^{\op}}X}", from=3-5, to=3-1]
	\arrow["{\phi^{\op}}", between={0.2}{0.8}, Rightarrow, from=0, to=1]
	\arrow["{\psi^{\op}}", between={0.2}{0.8}, Rightarrow, from=1, to=2]
	\arrow["{\zeta_h}", between={0.2}{0.8}, Rightarrow, from=3, to=4]
	\arrow["{\zeta_n}", between={0.2}{0.8}, Rightarrow, from=4, to=5]
\end{tikzcd} = 
\begin{tikzcd}
	{\cal A^{\op}} && {\cal B^{\op}} && {\cal E^{\op}} \\
	{\cal C^{\op}} && {\cal D^{\op}} && {\cal F^{\op}} \\
	\Set &&&& {\cal F^{\op}}
	\arrow["{f^{\op}}", from=1-1, to=1-3]
	\arrow["{g^{\op}}"', from=1-1, to=2-1]
	\arrow["{m^{\op}}", from=1-3, to=1-5]
	\arrow[""{name=0, anchor=center, inner sep=0}, "{k^{\op}}"{description}, from=1-3, to=2-3]
	\arrow[""{name=1, anchor=center, inner sep=0}, "{\lan_{f^{\op}}Xg^{\op}}"{description}, curve={height=12pt}, from=1-3, to=3-1]
	\arrow[""{name=2, anchor=center, inner sep=0}, "{l^{\op}}", from=1-5, to=2-5]
	\arrow["X"', from=2-1, to=3-1]
	\arrow["{n^{\op}}"{description}, from=2-3, to=2-5]
	\arrow[""{name=3, anchor=center, inner sep=0}, "{\lan_{h^{\op}}X}"{description}, from=2-3, to=3-1]
	\arrow[""{name=4, anchor=center, inner sep=0}, equals, from=2-5, to=3-5]
	\arrow["{\lan_{n^{\op}}\lan_{h^{\op}}X}", from=3-5, to=3-1]
	\arrow["{\zeta_{f}}", between={0}{0.8}, Rightarrow, from=1-1, to=1]
	\arrow["{\psi^{\op}}", between={0.2}{0.8}, Rightarrow, from=0, to=2]
	\arrow["{\overline{\phi}^\flat_X}"{pos=0.75}, between={0.2}{0.8}, Rightarrow, from=1, to=3]
	\arrow["{\zeta_{n}}", between={0.2}{0.8}, Rightarrow, from=3, to=4]
\end{tikzcd}\]
\[  =
\begin{tikzcd}
	{\cal A^{\op}} && {\cal B^{\op}} && {\cal E^{\op}} \\
	{\cal C^{\op}} \\
	\Set &&&& {\cal F^{\op}}
	\arrow["{f^{\op}}", from=1-1, to=1-3]
	\arrow["{g^{\op}}"', from=1-1, to=2-1]
	\arrow["{m^{\op}}", from=1-3, to=1-5]
	\arrow[""{name=0, anchor=center, inner sep=0}, curve={height=12pt}, from=1-3, to=3-1]
	\arrow[""{name=1, anchor=center, inner sep=0}, "{(\lan_{h^{\op}}X) \circ k^{\op}}"{description, pos=0.2}, curve={height=-12pt}, from=1-3, to=3-1]
	\arrow[""{name=2, anchor=center, inner sep=0}, "{\small \lan_{m^{\op}}((\lan_{h^{\op}}X) \circ k^{\op})}"{description, pos=0.3}, curve={height=-12pt}, from=1-5, to=3-1]
	\arrow["{l^{\op}}", from=1-5, to=3-5]
	\arrow["X"', from=2-1, to=3-1]
	\arrow["{\lan_{n^{\op}}\lan_{h^{\op}}X}", from=3-5, to=3-1]
	\arrow["{\zeta_{f}}", between={0}{0.8}, Rightarrow, from=1-1, to=0]
	\arrow["{\overline{\phi}^\flat_X}"', between={0.2}{0.8}, Rightarrow, from=0, to=1]
	\arrow["{\zeta_m}", between={0.2}{0.8}, Rightarrow, from=1, to=2]
	\arrow["{\overline{\psi}_{\lan_{h^{\op}}X}^\flat}", between={0.2}{0.8}, Rightarrow, from=2, to=3-5]
\end{tikzcd} = 
\begin{tikzcd}
	{\cal A^{\op}} && {\cal B^{\op}} && {\cal E^{\op}} \\
	{\cal C^{\op}} \\
	\Set &&&& {\cal F^{\op}}
	\arrow["{f^{\op}}", from=1-1, to=1-3]
	\arrow["{g^{\op}}"', from=1-1, to=2-1]
	\arrow["{m^{\op}}", from=1-3, to=1-5]
	\arrow[""{name=0, anchor=center, inner sep=0}, "{\small\lan_{(mf)^{\op}}(X\circ g^{\op})}"{description, pos=0.3}, shift left, curve={height=18pt}, from=1-5, to=3-1]
	\arrow[""{name=1, anchor=center, inner sep=0}, "{\small{\lan_{m^{\op}}((\lan_{h^{\op}}X) \circ k^{\op})}}"{description, pos=0.6}, shift right, curve={height=-18pt}, from=1-5, to=3-1]
	\arrow["{l^{\op}}", from=1-5, to=3-5]
	\arrow["X"', from=2-1, to=3-1]
	\arrow["{\lan_{n^{\op}}\lan_{h^{\op}}X}", from=3-5, to=3-1]
	\arrow["{\zeta_{mf}}", between={0.2}{0.8}, Rightarrow, from=1-1, to=0]
	\arrow["{\small{\lan_{m^\op}\overline{\phi}_X^\flat \alpha_{mf}} }", shift right, between={0.2}{0.8}, Rightarrow, from=0, to=1]
	\arrow["{\overline{\psi}_{\lan_{h^{\op}}X}^\flat}"{pos=0.7}, shift left=3, between={0.3}{0.8}, Rightarrow, from=1, to=3-5]
\end{tikzcd}\]
    Then uniqueness of the universal 2-cell provided by the left Kan extensions ensures the desired equality 
    \[ \alpha_{nh}^{-1}*l^{\op} \overline{(n*\phi)(\psi*f)}^\flat_X = (\lan_{m^\op}\overline{\phi}_X^\flat) \overline{\psi}_{\lan_{h^{\op}}X}^\flat \alpha_{mf}  \]
    Similarly, for the vertical case, one extracts first the universal 2-cell induced from the pasting:

\[\begin{tikzcd}
	{\cal A^{\op}} && {\cal B^{\op}} \\
	{\cal C^{\op}} && {\cal D^{\op}} \\
	{\cal F^{\op}} && {\cal E^{\op}} \\
	\Set
	\arrow["{f^{\op}}", from=1-1, to=1-3]
	\arrow["{g^{\op}}"', from=1-1, to=2-1]
	\arrow["{k^{\op}}", from=1-3, to=2-3]
	\arrow["{((\psi*g)(m*\phi))^{\op}}", between={0.2}{0.8}, Rightarrow, from=2-1, to=2-3]
	\arrow["{n^{\op}}"', from=2-1, to=3-1]
	\arrow["{m^{\op}}", from=2-3, to=3-3]
	\arrow["{l^{\op}}"{description}, from=3-1, to=3-3]
	\arrow[""{name=0, anchor=center, inner sep=0}, "X"', from=3-1, to=4-1]
	\arrow[""{name=1, anchor=center, inner sep=0}, "{\lan_{l^\op}X}", from=3-3, to=4-1]
	\arrow["{\zeta_l}", between={0.2}{0.8}, Rightarrow, from=0, to=1]
\end{tikzcd} = 
\begin{tikzcd}
	{\cal A^{\op}} && {\cal B^{\op}} \\
	{\cal C^{\op}} && {\cal D^{\op}} \\
	{\cal F^{\op}} \\
	\Set && {\cal E^{\op}}
	\arrow["{f^{\op}}", from=1-1, to=1-3]
	\arrow["{g^{\op}}"', from=1-1, to=2-1]
	\arrow["{k^{\op}}", from=1-3, to=2-3]
	\arrow[""{name=0, anchor=center, inner sep=0}, "{\small \lan_{f^\op} X(ng)^{\op}}"{description}, from=1-3, to=4-1]
	\arrow["{n^{\op}}"', from=2-1, to=3-1]
	\arrow["{m^{\op}}", from=2-3, to=4-3]
	\arrow["X"', from=3-1, to=4-1]
	\arrow[""{name=1, anchor=center, inner sep=0}, "{\lan_{l^\op}X}", from=4-3, to=4-1]
	\arrow["{\zeta_f}", between={0.2}{0.8}, Rightarrow, from=1-1, to=0]
	\arrow["{\small \overline{(\psi*g)(m*\phi)}^\flat_X}"{description, pos=0.6}, between={0.2}{0.9}, Rightarrow, from=0, to=1] 
\end{tikzcd}\]
    On the other hand, applying separately the universal property for each cell gives the equality
\[\begin{tikzcd}
	{\cal A^{\op}} && {\cal B^{\op}} \\
	{\cal C^{\op}} && {\cal D^{\op}} \\
	{\cal F^{\op}} && {\cal E^{\op}} \\
	\Set
	\arrow["{f^{\op}}", from=1-1, to=1-3]
	\arrow[""{name=0, anchor=center, inner sep=0}, "{g^{\op}}"', from=1-1, to=2-1]
	\arrow[""{name=1, anchor=center, inner sep=0}, "{k^{\op}}", from=1-3, to=2-3]
	\arrow["{h^{\op}}"{description}, from=2-1, to=2-3]
	\arrow[""{name=2, anchor=center, inner sep=0}, "{n^{\op}}"', from=2-1, to=3-1]
	\arrow[""{name=3, anchor=center, inner sep=0}, "{m^{\op}}", from=2-3, to=3-3]
	\arrow["{l^{\op}}"{description}, from=3-1, to=3-3]
	\arrow[""{name=4, anchor=center, inner sep=0}, "X"', from=3-1, to=4-1]
	\arrow[""{name=5, anchor=center, inner sep=0}, "{\lan_{l^\op}X}", from=3-3, to=4-1]
	\arrow["{\phi^{\op}}", between={0.2}{0.8}, Rightarrow, from=0, to=1]
	\arrow["{\psi^{\op}}", between={0.2}{0.8}, Rightarrow, from=2, to=3]
	\arrow["{\zeta_l}", between={0.2}{0.8}, Rightarrow, from=4, to=5]
\end{tikzcd} =
\begin{tikzcd}
	{\cal A^{\op}} && {\cal B^{\op}} \\
	{\cal C^{\op}} && {\cal D^{\op}} \\
	{\cal F^{\op}} \\
	\Set && {\cal E^{\op}}
	\arrow["{f^{\op}}", from=1-1, to=1-3]
	\arrow[""{name=0, anchor=center, inner sep=0}, "{g^{\op}}"', from=1-1, to=2-1]
	\arrow[""{name=1, anchor=center, inner sep=0}, "{k^{\op}}", from=1-3, to=2-3]
	\arrow["{h^{\op}}"{description}, from=2-1, to=2-3]
	\arrow["{n^{\op}}"', from=2-1, to=3-1]
	\arrow[""{name=2, anchor=center, inner sep=0}, "{\lan_{h^{\op}}Xn^{\op}}"{description, pos=0.5}, from=2-3, to=4-1]
	\arrow["{m^{\op}}", from=2-3, to=4-3]
	\arrow["X"', from=3-1, to=4-1]
	\arrow["{\lan_{l^\op}X}", from=4-3, to=4-1]
	\arrow["{\phi^{\op}}", between={0.2}{0.8}, Rightarrow, from=0, to=1]
	\arrow["{\zeta_h}"{pos=0.3}, shift right, between={0}{0.8}, Rightarrow, from=2-1, to=2]
	\arrow["{\overline{\psi}^\flat}", between={0.2}{0.8}, Rightarrow, from=2, to=4-3]
\end{tikzcd} =
\begin{tikzcd}
	{\cal A^{\op}} && {\cal B^{\op}} \\
	{\cal C^{\op}} && {\cal D^{\op}} \\
	{\cal F^{\op}} \\
	\Set && {\cal E^{\op}}
	\arrow["{f^{\op}}", from=1-1, to=1-3]
	\arrow["{g^{\op}}"', from=1-1, to=2-1]
	\arrow["{k^{\op}}", from=1-3, to=2-3]
	\arrow[""{name=0, anchor=center, inner sep=0}, "{\small{\lan_{f{\op}}X(ng)^{\op}}}"{description}, curve={height=12pt}, from=1-3, to=4-1]
	\arrow["{n^{\op}}"', from=2-1, to=3-1]
	\arrow[""{name=1, anchor=center, inner sep=0}, "{\lan_{h^{\op}}Xn^{\op}}"{description, pos=0.7}, curve={height=-12pt}, from=2-3, to=4-1]
	\arrow["{m^{\op}}", from=2-3, to=4-3]
	\arrow["X"', from=3-1, to=4-1]
	\arrow["{\lan_{l^\op}X}", from=4-3, to=4-1]
	\arrow["{\zeta_f}"{pos=0.5}, between={0.2}{0.8}, Rightarrow, from=1-1, to=0]
	\arrow["{\overline{\phi}^\flat_{Xn^{\op}}}"{pos=0.6}, between={0.2}{0.8}, Rightarrow, from=0, to=1]
	\arrow["{\overline{\psi}^\flat_X}", shift left, between={0.2}{0.8}, Rightarrow, from=1, to=4-3]
\end{tikzcd}\]
    which again ensures by uniqueness of the solution to the universal property of Kan extensions the desired equality
    \[  \overline{(\psi*g)(m*\phi)}^\flat_X = (\overline{\psi}^\flat_X * k^{\op})\overline{\phi}^\flat_{Xn^{\op}}  \]
\end{proof}

\begin{lemma}\label{naturality of mate for units}
    For any functor $f: \mathcal{A} \rightarrow \mathcal{B}$, the corresponding horizontal and vertical identity squares
\[\begin{tikzcd}
	{\mathcal{A}} & {\mathcal{B}} \\
	{\mathcal{A}} & {\mathcal{B}}
	\arrow["f", from=1-1, to=1-2]
	\arrow[equals, from=1-1, to=2-1]
	\arrow[equals, from=1-2, to=2-2]
	\arrow["f"', from=2-1, to=2-2]
\end{tikzcd} \hspace{1cm} \textup{resp.} \hspace{1cm} 
\begin{tikzcd}
	{\mathcal{A}} & {\mathcal{A}} \\
	{\mathcal{B}} & {\mathcal{B}}
	\arrow["f"', from=1-1, to=2-1]
	\arrow[equals, from=1-2, to=1-1]
	\arrow["f", from=1-2, to=2-2]
	\arrow[equals, from=2-2, to=2-1]
\end{tikzcd}\]
    induce respectively the identity of $\lext_f$ and $\rest_f$.
\end{lemma}

\begin{proof}
    This is immediate: both restricting or extending along the identity functor gives identity for presheaves.
\end{proof}

\subsection{Sieves and sites}

The cross-adjoint cells we saw at the end of the previous section will play a crucial role when establishing the double functoriality of the sheaf construction in the next section: we will see that, if one consider \emph{coverages} on the categories and require a correct assortment of cover preservation and lifting for the functors constituting our squares, then the cross-adjoint cells will restrict between the associated sheaf topoi and produce geometric transformations between geometric morphisms.

We first recall generalities about coverages and sites. Those notions can be expressed in several ways, either through \emph{covering families}, or also through \emph{sieves}. In this work we will mostly use the sieve formulation, for its affinity with the extension-restriction formalism.

\begin{definition}
    A \emph{sieve} on an object $c$ in a category $ \mathcal{C}$ is a subobject $ S \rightarrowtail \hirayo_c$, or equivalently, a family $S$ of arrows with codomain $c$ absorbing under precomposition: if $ u : c' \rightarrow c$ is in $S$, then for any $ u': c'' \rightarrow c'$ the composite $ uu': c'' \rightarrow c$ is in $S$. 
\end{definition}

For a family of arrows $ S$ with common codomain $c$, the sieve \emph{generated} by $ S$ is the set of maps that factorizes through a map in $S$
\[ \overline{S} = \Bigg{\{} v : d \rightarrow c \mid \exists u : c' \rightarrow c \in S \textup{ such that } 
\begin{tikzcd}
	d & c \\
	{c'}
	\arrow["v", from=1-1, to=1-2]
	\arrow["\exists"', dashed, from=1-1, to=2-1]
	\arrow["u"', from=2-1, to=1-2]
\end{tikzcd}\Bigg{\}} \] 

For two sieves $S$, $R$ on $c$, one has $ R \leq S$ if any arrow in $R$ factorizes through an arrow in $S$.

\begin{definition}
  A \emph{site} is a category $\mathcal{C}$ together with, for each $c$, a set $J(c)$ of sieves on $c$ such that for each $u : d \rightarrow c$ and any $S \in J(c)$, there is $R \in J(d)$ such that for all $v \in R$, one has $ uv \in S$. The data of all the sets $J(c)$ will be called a \emph{coverage}.   
\end{definition}

\begin{division}\label{axiom for topo}
This is the most general definition possible; however one often consider the following additional axioms, which facilitate inferences about coverages:\begin{itemize}
        \item \textit{maximality}: for each $c$ the maximal sieve $ \max(c) = \hirayo_c$ generated from $1_c$ is in $J(c)$
        \item \textit{up-closure}: for a sieve $S \in J(c)$, if $ S \leq T$, then $T \in J(c)$
        \item \textit{stability}: for $S \in J(c)$ and $u:d \rightarrow c$, the pullback sieve below is in $J(d)$: \[ u^*S = \Bigg{\{} v : d' \rightarrow d  \mid \exists w : c' \rightarrow c \in S \textup{ and a factorization } 
\begin{tikzcd}[sep=small]
	{d'} & {c'} \\
	d & c
	\arrow["\exists", dashed, from=1-1, to=1-2]
	\arrow["v"', from=1-1, to=2-1]
	\arrow["{w \in S}", from=1-2, to=2-2]
	\arrow["u"', from=2-1, to=2-2]
\end{tikzcd} \Bigg{\}} \]
        \item \textit{filteredness}: for $ S,R \in J(c)$, then the intersection sieve $S\wedge R$ is in $ J(c)$
        \item \textit{weak locality}: if $ S \in J(c)$ and $T \leq S$ satisfies that for each $u : b \rightarrow c \in S$, $u^*T \in J(b)$, then $T $ is in $J(c)$
        \item \textit{locality}: for any sieve $ S$ over $ c$, if the following \emph{$J$-localizing sieve} of $S$ 
        \[ J/S = \{ v: d\rightarrow c \mid v^*S \in J(d) \} \]
        is in $J(c)$, then $S$ itself has to be in $J(c)$.
        \item \textit{transitivity}: if $S \in J(c)$ and for each $u :b \rightarrow c $ in $S$ one takes $T_u \in J(b)$, then the multicomposite $ \bigvee_{u \in S} u \circ T_u $ is in $J(c)$.
    \end{itemize}

\end{division}

\begin{remark}
      Upclosure and filteredness follow from the conjunction of maximality, stability and locality; from \cite{maclane&moerdijk} it is known that reciprocally, maximality, stability and weak locality together with either upclosure or filteredness ensures locality; finally it is well known that in the presence of maximality and stability, locality and transitivity are equivalent. 
\end{remark}

The term \emph{coverage} was introduced in \cite{elephant}; in the following we will slightly adapt it, enforcing the following assortment of axioms for the simplifications they will allows in our context:

\begin{definition}
    We will call a \emph{Grothendieck coverage} on $\mathcal{C}$ a coverage $J$ that satisfies the maximality, up-closure and stability axioms. A \emph{Grothendieck topology} is a Grothendieck coverage that satisfies moreover the locality axioms, hence all the possible axioms of the list above.
\end{definition}

   \begin{remark}
        This notion corresponds actually to what is called \emph{sifted coverage} in \cite{elephant}[C2.1]. In practice, even if a coverage does not satisfies one of those axioms, it is possible to close it into a Grothendieck coverage or a Grothendieck topology by closing it under pullback, maximality and transitivity: this process does not change the notion of sheaves, not the notion of continuity we are going to consider below. In the following we may say simply "coverage'' rather than ``Grothendieck coverage'' for concision.
   \end{remark}

\begin{remark}
There is an alternative presentation of topologies in terms of covers: a Grothendieck coverage is the data for each $c$ of a set $J(c)$ of families of arrows $(v_i : d_i \rightarrow c)_{i \in I}$ such that:\begin{itemize}
    \item for any isomorphism $ u : c' \simeq c$, the singleton $\{ u \}$ is in $J(c)$
    \item for any $v : d \rightarrow c$ and any $J$-cover $ (u_i : c_i \rightarrow c)_{i \in I}$, there is a $J$-cover $ (v_j : d_j \rightarrow d)_{j \in J}$ such that for any $j \in J$ the composite $ vv_j$ factorizes through $u_i$ for some $i$;
    \item for any $(v_i : d_i \rightarrow c)_{i \in I}$ in $J(c)$, if one has for each $i\in I$ a sieve $ (w_{ij} : d_{ij} \rightarrow d_i)_{j \in J_i}$ in $J(d_i)$, then the family of the composites $ (v_iw_{ij} : d_{ij} \rightarrow c)_{(i,j) \in \coprod_{i \in I} J_i}$ is in $J(c)$.
\end{itemize}
\end{remark}

\begin{remark}
    If $\mathcal{C}$ has pullbacks, then the second condition can be rephrased as follows: for any $v : d \rightarrow c$, the family of pullbacks $ (v^*v_i : v^*d_i \rightarrow d)_{i \in I}$ exists and is in $J(d)$. In practice, for we often assume $ \mathcal{C}$ to have finite limits, this is the form we are going to use. 
\end{remark}

\begin{remark}
    Sites in this work will be considered as \emph{small generated}, that is, generated from a small dense subsite, even when they are large. 

\end{remark}

Extension and restrictions have a concrete signification for sieves: they correspond to the operations taking images and preimages of covering families. The following computations will play a crucial role in the remaining of this work as well as its sequel.

\begin{division}[Extensions of sieves]
Consider a functor $ f : \mathcal{C} \rightarrow \mathcal{D}$ and a sieve $ S \rightarrowtail \hirayo_c $ on an object $c$ of $\mathcal{C}$. Observing that $ \lext_f(\hirayo_c) \simeq \hirayo_{f(c)}$, one gets a canonical arrow $ \lext_f(S) \rightarrow \hirayo_{f(c)}$, where $ \lext_f(S)$ returns at any object $d$ the coend 
\[ 
    \lext_f(S)(d) = \int^{c' \in \mathcal{C}} D(d,f(c')) \times S(c')
 \] 
 Though in full generality, this arrow may not be a monomorphism, one can still take the image factorization in $\widehat{\mathcal{D}}$   
\[\begin{tikzcd}
	{\lext_f(S)} && {\hirayo_{f(c)}} \\
	& {f[S]}
	\arrow["{\lext_f(m)}", from=1-1, to=1-3]
	\arrow[two heads, from=1-1, to=2-2]
	\arrow["{f[m]}"', tail, from=2-2, to=1-3]
\end{tikzcd}\]

whose evaluation at $d$ is exactly the set of arrows factorizing through the image of $S$
\[ f[S](d) =  \Bigg{\{} v : d \rightarrow f(c) \mid  
\begin{tikzcd}
	d & {f(c)} \\
	{f(c')}
	\arrow["v", from=1-1, to=1-2]
	\arrow["\exists"', dashed, from=1-1, to=2-1]
	\arrow["{f(u)}"', from=2-1, to=1-2]
\end{tikzcd} \textrm{ for some } u : c' \rightarrow c \in S(c')  \Bigg{\}}
\]
\end{division}

\begin{remark}
    In the case where $ f$ is a \emph{flat} functor, as we are going to discus below, then $ \lext_f$ becomes lex and hence preserves monomorphisms, so the distinction between $ \lext_f(S)$ and $ f[S]$ vanishes.
\end{remark}

\begin{division}[Restriction of sieves]
Dually, for a functor $ f: \mathcal{C} \rightarrow \mathcal{D}$ and a sieve $ R \rightarrowtail \hirayo_d$, the restriction functor, which is by \cref{Formulas for restrictions and lextensions} the left extension $ \rest_f = \lan_{\hirayo_\mathcal{D}} \mathcal{D}(f,1)$, returns at $R$ a presheaf $ \rest_f R$ computed through the coend formula
\[ \rest_f R = \int^{d' \in \mathcal{D}} R(d') \times \mathcal{D}(f,d') \]
which, at $c'$ in $\mathcal{C}$, returns the set of maps 
\[ \rest_f R(c') = \Bigg{\{} v: f(c') \rightarrow d \mid 
\begin{tikzcd}
	{f(c')} & d \\
	{d'}
	\arrow["v", from=1-1, to=1-2]
	\arrow["\exists"', dashed, from=1-1, to=2-1]
	\arrow["v'"', from=2-1, to=1-2]
\end{tikzcd} \textrm{ for some } v' : d' \rightarrow d \in R(d) \Bigg{\}}  \]
In particular, in the case where $ d$ is of the form $f(c)$, $ \rest_f R$ is a subobject of $ \rest_f \hirayo_{f(c)}$ for $ \rest_f$, as a right adjoint, preserves monomorphisms. However, beware that $ \rest_f \hirayo_{f(c)} \simeq \mathcal{D}(f,f(c)) $, which is not a representable on $\mathcal{C}$: hence $ \rest_fR$ is not itself yet a sieve on $\mathcal{C}$. However, it locally is. In particular, there is a canonical element $ \nu_c : \hirayo_c \rightarrow \rest_f \hirayo{f(c)}$ given by the canonical 2-cell $\nu : \hirayo_\mathcal{C} \Rightarrow \mathcal{D}(f,1) f$ associated with the nerve of $f$, which is also the name of the identity $ 1_{f(c)}$ through the Yoneda isomorphism $ \widehat{\mathcal{C}}(\hirayo_c, \rest_f \hirayo_{f(c)}) \simeq \mathcal{D}(f(c), f(c))$: then one can consider the pullback presheaf 
\[\begin{tikzcd}
	{f^{-1}(R)} & {\rest_f R} \\
	{\hirayo_c} & {\mathcal{D}(f,f(c))}
	\arrow[from=1-1, to=1-2]
	\arrow[tail, from=1-1, to=2-1]
	\arrow["\lrcorner"{anchor=center, pos=0.125}, draw=none, from=1-1, to=2-2]
	\arrow[tail, from=1-2, to=2-2]
	\arrow["{\nu_c}"', from=2-1, to=2-2]
\end{tikzcd}\]
which is now a subobject of $ \hirayo_c$ and corresponds to the sieve on $c$ given by the set
\[ \Bigg{\{} u: c' \rightarrow c \mid 
\begin{tikzcd}
	{f(c')} & f(c) \\
	{d'}
	\arrow["f(u)", from=1-1, to=1-2]
	\arrow["\exists"', dashed, from=1-1, to=2-1]
	\arrow["v'"', from=2-1, to=1-2]
\end{tikzcd} \textrm{ for some } v' : d' \rightarrow f(c) \in R(d) \Bigg{\}}\]

\end{division}

\begin{remark}
    The pair $ f^{-1}$, $f[-]$ form at each $c$ a posetal adjunction mirroring the adjunction $ \lext_f\dashv \rest_f$:
\[\begin{tikzcd}
	{\Sub_\mathcal{D}(f(c))} && {\Sub_\mathcal{C}(c)}
	\arrow[""{name=0, anchor=center, inner sep=0}, "{f^{-1}}"', curve={height=12pt}, from=1-1, to=1-3]
	\arrow[""{name=1, anchor=center, inner sep=0}, "{f[-]}"', curve={height=12pt}, from=1-3, to=1-1]
	\arrow["\dashv"{anchor=center, rotate=-90}, draw=none, from=1, to=0]
\end{tikzcd}\]
    Indeed, for $S \rightarrowtail \hirayo_c$ and $R \rightarrowtail \hirayo_{f(c)}$ saying that $ S \leq f^{-1}(R)$ amounts to saying that for any $u$ in $S$, $f(u)$ factorizes through some $v$ in $R$ which means that $f[S] \leq R$.
\end{remark}

Those operations on sieves will provide a synthetic way to work with the two main protagonists of our story, the dual classes of functors one can consider between sites, namely morphisms and comorphisms of sites.

\subsection{Morphisms of sites}


\begin{definition}
A functor between sites $ f : (\mathcal{C}, J) \rightarrow (\mathcal{D},K)$ is \emph{cover-preserving} if for any $J$-covering sieve $ S \rightarrowtail \hirayo_c$ in $J(c)$, the image sieve $ f[S]$ is in $K(f(c))$. 
\end{definition}

In general cover-preservation is combined with flatness to produce a convenient notion of functor between sites, as sole cover preservation does not ensure convenient restriction to categories of sheaves. In the context where one consider sites both as domain and codomain, one must use the most general notion of flatness, as introduced in \cite{karazeris2004notions} and \cite{kock1989postulated} (see however \cite{shulman2012exact}[lemma 4.2] for the explicit formulation we use):

\begin{definition}
    A functor between sites $ f : (\mathcal{C},J) \rightarrow (\mathcal{D},K)$ is said to be \emph{covering-flat} if for any finite diagram $(c_i)_{i \in I}$ in $\mathcal{C}$ and any cone $ (v_i : d \rightarrow f(c_i))_{i \in I}$ in $\mathcal{D}$, the set of arrows $ w: d' \rightarrow d$ such that there is a cone $ (a_i :c \rightarrow c_i)_{i \in I}$ in $\mathcal{C}$ together with a factorization as below is $K$-covering.
\[\begin{tikzcd}
	{d'} & d \\
	{f(c)} & {f(c_i)}
	\arrow["w", from=1-1, to=1-2]
	\arrow["{\exists u}"', from=1-1, to=2-1]
	\arrow["{v_i}", from=1-2, to=2-2]
	\arrow["{f(a_i)}"', from=2-1, to=2-2]
\end{tikzcd}\]

\end{definition}

\begin{remark}
    Observe in fact this definition only involve the topology on the codomain category.
\end{remark}

\begin{definition}
A \emph{morphism of sites} $ f : (\mathcal{C},J) \rightarrow (\mathcal{D},K)$ is a functor that is both covering-flat and cover-preserving.

\end{definition}

We will denote as $ \Site^{\flat}$ the 2-category of (small generated) sites, morphisms of sites and arbitrary natural transformations between them. We emphasize here that no restrictions are to be enforced on 2-cells. Observe that there is an inclusion $ \Top^{\op} \hookrightarrow \Site^\flat$ sending any topos \(\mathcal{E} \) on the canonical site $ (\mathcal{E}, J_{can})$, any geometric morphism $f$ to its inverse image $ f^*$, and any geometric $ \phi$ transformation to its inverse image part $ \phi^*$.

\begin{definition}
    If $ \mathcal{E}$ is a topos, a functor $(\mathcal{C}, J) \rightarrow \mathcal{E}$ is said to be \emph{$J$-continuous} if it is covering-flat and defines a morphism of sites into $ (\mathcal{E},J_{can})$ where $J_{can}$ is the canonical topology on $\mathcal{E}$.
\end{definition}

\begin{division}[Geometric morphism from morphism of sites]

A morphism of sites $f : (\mathcal{C},J) \rightarrow (\mathcal{D},K) $ induces a geometric morphism $ \Sh(f)$ whose inverse image is related to $f$ through an invertible 2-cell 
\[\begin{tikzcd}
	{\mathcal{C}} & {\mathcal{D}} \\
	{\widehat{\mathcal{C}}_J} & {\widehat{\mathcal{D}}_K}
	\arrow[""{name=0, anchor=center, inner sep=0}, "f", from=1-1, to=1-2]
	\arrow["{\mathfrak{y}_{(\mathcal{C},J)}}"', from=1-1, to=2-1]
	\arrow["{\mathfrak{y}_{(\mathcal{D},k)}}", from=1-2, to=2-2]
	\arrow[""{name=1, anchor=center, inner sep=0}, "{\Sh(f)^*}"', from=2-1, to=2-2]
	\arrow["{\zeta_f \atop\simeq }"{description}, draw=none, from=0, to=1]
\end{tikzcd}\]

Moreover the inverse image part is also related to the left adjoint $ \lext_f$ as follows
\[\begin{tikzcd}
	{\widehat{C}} & {\widehat{D}} \\
	{\widehat{\mathcal{C}}_J} & {\widehat{\mathcal{D}}_K}
	\arrow["{\lext_f}", from=1-1, to=1-2]
	\arrow["{\mathfrak{a}_K}", from=1-2, to=2-2]
	\arrow["{\mathfrak{i}_J}", hook, from=2-1, to=1-1]
	\arrow["{\Sh(f)^*}"', from=2-1, to=2-2]
\end{tikzcd}\]

Moreover, the restriction functors $\rest_f$ restrict to sheaves as the direct image part
\[\begin{tikzcd}
	{\widehat{C}} & {\widehat{D}} \\
	{\widehat{\mathcal{C}}_J} & {\widehat{\mathcal{D}}_K}
	\arrow["{\rest_f}"', from=1-2, to=1-1]
	\arrow["{\mathfrak{i}_J}", hook, from=2-1, to=1-1]
	\arrow["{\mathfrak{i}_K}"', hook, from=2-2, to=1-2]
	\arrow["{\Sh(f)_*}", dashed, from=2-2, to=2-1]
\end{tikzcd}\]

Finally by adjunction this produces also an invertible 2-cell relating the inverse image parts
\[\begin{tikzcd}
	{\widehat{C}} & {\widehat{D}} \\
	{\widehat{\mathcal{C}}_J} & {\widehat{\mathcal{D}}_K}
	\arrow[""{name=0, anchor=center, inner sep=0}, "{\lext_f}", from=1-1, to=1-2]
	\arrow["{\mathfrak{a}_J}"', from=1-1, to=2-1]
	\arrow["{\mathfrak{a}_K}", from=1-2, to=2-2]
	\arrow[""{name=1, anchor=center, inner sep=0}, "{\Sh(f)^*}"', from=2-1, to=2-2]
	\arrow["\simeq"{description}, draw=none, from=0, to=1]
\end{tikzcd}\]
and at the level of the 2-category of Grothendieck topoi a canonical invertible 2-cell
\[\begin{tikzcd}
	{\widehat{\mathcal{C}}} & {\widehat{\mathcal{D}}} \\
	{\widehat{\mathcal{C}}_J} & {\widehat{\mathcal{D}}_K}
	\arrow["{\widehat{f}}"', from=1-2, to=1-1]
	\arrow["{\iota_J}", hook', from=2-1, to=1-1]
	\arrow["{\iota_{K}}"', hook', from=2-2, to=1-2]
	\arrow["{\Sh(f)}", from=2-2, to=2-1]
    \arrow["\simeq"{description}, draw=none, from=0, to=1]
\end{tikzcd}\]

      This defines a pseudofunctor, contravariant in 1-cells and covariant on 2-cells:
\[\begin{tikzcd}
	{(\Site^{\flat})^{\op}} & \Top
	\arrow["\Sh", from=1-1, to=1-2]
\end{tikzcd}\]
\end{division}

\subsection{Comorphisms of sites}

Recall that a functor between sites $f: (\mathcal{D}, K) \rightarrow (\mathcal{C},J)$ is usually said to be {cover-lifting}, or also to be a {comorphism of sites}, if for any $d$ in $\mathcal{D}$ and any $J$-covering sieve $S$ on $f(d)$, there is a $K$-covering sieve $R$ such that $ f(v)$ is in $S$ for all $u \in R$. Again, this can be rephrased, this time with the restriction functor:

\begin{definition}
    A functor between sites $f: (\mathcal{D}, K) \rightarrow (\mathcal{C},J)$ is \emph{cover-lifting}
if for any $d$ in $\mathcal{D}$ and any $J$-covering sieve $S$ on $f(d)$, the restricted sieve $ f^{-1}(S)$ on $d$ is $K$-covering. 
\end{definition}

\begin{remark}
  Equivalently, this conditions amounts to $S$ containing a sieve of the form $ f[R]$ for some $K$-covering sieve $ R$ on $d$.  
\end{remark}

We will denote as $ \Site^{\sharp}$ the 2-category of sites, comorphisms of sites and transformations between them. Again no restrictions are required on the 2-cells.

\begin{division}[Geometric morphisms from comorphisms of sites]\label{geometric morphism from comorphism of sites}

Here is the process from \cite{caramello2020denseness}[Section 3.3] from which a comorphism of sites induces a geometric morphism. For a comorphism of sites $ F : (\mathcal{D},K) \rightarrow (\mathcal{C},J)$, denote as $ A_F : (\mathcal{C},J) \rightarrow \widehat{\mathcal{D}}_K$ the composite $ \mathfrak{a}_K \mathcal{C}(F,1)$. Then $A_F$ is a $J$-continuous flat functor, and it induces a geometric morphism $ \Sh(A_F) = C_F : \widehat{\mathcal{D}}_K \rightarrow \widehat{\mathcal{C}}_J$ with $ C_F^* = \Sh(A_F)^* = \lan_{\mathfrak{y}_{(\mathcal{C},J)}}A_F$. In fact, following \cite{elephant}[C2.3.18], the inverse image part is also obtained as the composite
\[\begin{tikzcd}
	{\widehat{\mathcal{C}}} & {\widehat{\mathcal{D}}} \\
	{\widehat{\mathcal{C}}_J} & {\widehat{\mathcal{D}}_K}
	\arrow["{\rest_F}", from=1-1, to=1-2]
	\arrow["{\mathfrak{a}_K}", from=1-2, to=2-2]
	\arrow["{\mathfrak{i}_J}", hook, from=2-1, to=1-1]
	\arrow["{C_F^*}"', dashed, from=2-1, to=2-2]
\end{tikzcd}\]

More generally a functor $F: \mathcal{D} \rightarrow \mathcal{C}$ defines a comorphism of sites $ (\mathcal{D},K) \rightarrow (\mathcal{C},J) $ if and only if the restriction functor fits into a diagram
\[\begin{tikzcd}
	{\widehat{C}} & {\widehat{D}} \\
	{\widehat{\mathcal{C}}_J} & {\widehat{\mathcal{D}}_K}
	\arrow["{\rest_F}", from=1-1, to=1-2]
	\arrow["{\mathfrak{a}_J}"', from=1-1, to=2-1]
	\arrow["{\mathfrak{a}_K}", from=1-2, to=2-2]
	\arrow["{C_F^*}"', dashed, from=2-1, to=2-2]
\end{tikzcd}\]

Such a 2-cell has a mate relating the direct image part with the right extension functor:
\[\begin{tikzcd}
	{\widehat{C}} & {\widehat{D}} \\
	{\widehat{\mathcal{C}}_J} & {\widehat{\mathcal{D}}_K}
	\arrow["{\rext_F}"', from=1-2, to=1-1]
	\arrow["{\mathfrak{i}_J}", hook, from=2-1, to=1-1]
	\arrow["{\mathfrak{i}_K}"', hook, from=2-2, to=1-2]
	\arrow["{C_{F*}}", from=2-2, to=2-1]
\end{tikzcd}\]
  
   This defines a pseudofunctor, covariant in 1-cells and contravariant in 2-cells:
\[\begin{tikzcd}
	{(\Site^{\sharp})^{\co}} & \Top
	\arrow["C", from=1-1, to=1-2]
\end{tikzcd}\]
\end{division}

\section{Double category of sites}

As announced above, our main goal in this work will be to reconciliate morphisms and comorphisms into a same two-dimensional categorical structure on sites. A hindrance to fitting them into a single 2-category is that morphisms and comorphisms do not compose with each other: of course, one could compose the underlying functors, but then cover-preservation and cover-lifting may be lost. While in the sequel of this paper we will see that a common refinement of those two properties will provide a suitable, composition-stable notion of 1-cell for a 2-category, we here rather \emph{avoid} the question of composing morphisms and comorphisms altogether: this is the purpose of arranging them as the vertical and horizontal cells of a \emph{double category} of sites.

\subsection{Prerequisites on double categories}

Double categories, which were initially introduced in \cite{ehresmann1963categories}, are characterized by a dichotomy between two classes of 1-cells that are not be required to compose with each other, being rather related through square-like 2-cells intertwining the two possible assortments of either classes of 1-cells. We recall here the definition of (strict) double categories as well as double (pseudo) functors.

\begin{definition}
    A (strict) \emph{double category} $\mathbb{D}$ is the data of \begin{itemize}
        \item a class of \emph{objects} $ C,D...$
        \item a class of \emph{horizontal 1-cells} $ f: C \rightarrow D$
        \item a class of \emph{vertical 1-cells} $ F : C \distrightarrow D$
        \item a class of \emph{double cells} of the form
\[\begin{tikzcd}
	A & B \\
	C & D
	\arrow["f", from=1-1, to=1-2]
	\arrow["F"', "\shortmid"{marking}, from=1-1, to=2-1]
	\arrow["\phi"{description}, draw=none, from=1-1, to=2-2]
	\arrow["G", "\shortmid"{marking}, from=1-2, to=2-2]
	\arrow["g"', from=2-1, to=2-2]
\end{tikzcd}\]
    \end{itemize}

subject to the following axioms:\begin{itemize}
    \item each object $C$ admits both a horizontal and vertical identity $ \id_C$ and $ 1_C$
    \item horizontal arrows from a category with identities given by horizontal identities
    \item vertical arrows from a category with identities given by vertical identities
    \item double cells paste horizontally and vertically
     \[ 
\begin{tikzcd}
	A & {A'} & {A''} \\
	B & {B'} & {B''}
	\arrow["f", from=1-1, to=1-2]
	\arrow[""{name=0, anchor=center, inner sep=0}, "F"', "\shortmid"{marking}, from=1-1, to=2-1]
	\arrow["{f'}", from=1-2, to=1-3]
	\arrow[""{name=1, anchor=center, inner sep=0}, "{F'}"{description}, from=1-2, to=2-2]
	\arrow[""{name=2, anchor=center, inner sep=0}, "{F''}", "\shortmid"{marking},  from=1-3, to=2-3]
	\arrow["g"', from=2-1, to=2-2]
	\arrow["{g'}"', from=2-2, to=2-3]
	\arrow["\phi"{description}, draw=none, from=1, to=0]
	\arrow["{\phi'}"{description}, draw=none, from=2, to=1]
\end{tikzcd}  = \begin{tikzcd}[column sep=small]
	A && A'' \\
	{B} && {B''}
	\arrow["{f'f}", from=1-1, to=1-3]
	\arrow[""{name=0, anchor=center, inner sep=0}, "F"', "\shortmid"{marking}, from=1-1, to=2-1]
	\arrow[""{name=1, anchor=center, inner sep=0}, "G", "\shortmid"{marking}, from=1-3, to=2-3]
	\arrow["{g'g}"', from=2-1, to=2-3]
	\arrow["{\phi'\circ\phi}"{description}, draw=none, from=0, to=1]
\end{tikzcd} \]
\[\begin{tikzcd}
	A & B \\
	{A'} & {B'} \\
	{A''} & {B''}
	\arrow["f", from=1-1, to=1-2]
	\arrow[""{name=0, anchor=center, inner sep=0}, "F"', "\shortmid"{marking}, from=1-1, to=2-1]
	\arrow[""{name=1, anchor=center, inner sep=0}, "G", "\shortmid"{marking}, from=1-2, to=2-2]
	\arrow["{f'}"{description}, from=2-1, to=2-2]
	\arrow[""{name=2, anchor=center, inner sep=0}, "{F'}"', "\shortmid"{marking}, from=2-1, to=3-1]
	\arrow[""{name=3, anchor=center, inner sep=0}, "{G'}", "\shortmid"{marking}, from=2-2, to=3-2]
	\arrow["{f''}"', from=3-1, to=3-2]
	\arrow["\phi"{description}, draw=none, from=1, to=0]
	\arrow["{\phi'}"{description}, draw=none, from=3, to=2]
\end{tikzcd}= \begin{tikzcd}
	A & B \\
	\\
	{A''} & {B''}
	\arrow["f", from=1-1, to=1-2]
	\arrow[""{name=0, anchor=center, inner sep=0}, "{F'F}"', "\shortmid"{marking}, from=1-1, to=3-1]
	\arrow[""{name=1, anchor=center, inner sep=0}, "{G'G}", "\shortmid"{marking}, from=1-2, to=3-2]
	\arrow["{f''}"', from=3-1, to=3-2]
	\arrow["{\phi'\bullet\phi}"{description}, shorten <=6pt, shorten >=6pt, Rightarrow, draw=none, from=0, to=1]
\end{tikzcd}  \]
    \item horizontal and vertical identities define horizontal and vertical identity squares which are units for pasting
\[\begin{tikzcd}
	A & A \\
	B & B
	\arrow["{\id_A}", equals, from=1-1, to=1-2]
	\arrow[""{name=0, anchor=center, inner sep=0}, "F"'{inner sep=.8ex}, "\shortmid"{marking}, from=1-1, to=2-1]
	\arrow[""{name=1, anchor=center, inner sep=0}, "F"{inner sep=.8ex}, "\shortmid"{marking}, from=1-2, to=2-2]
	\arrow["{\id_B}"', equals, from=2-1, to=2-2]
	\arrow["{\id_F}"{description}, draw=none, from=0, to=1]
\end{tikzcd} \hspace{1cm} 
\begin{tikzcd}
	A & B \\
	A & B
	\arrow[""{name=0, anchor=center, inner sep=0}, "f", from=1-1, to=1-2]
	\arrow["{1_A}"'{inner sep=.8ex}, "\shortmid"{marking}, equals, from=1-1, to=2-1]
	\arrow["{1_B}"{inner sep=.8ex}, "\shortmid"{marking}, equals, from=1-2, to=2-2]
	\arrow[""{name=1, anchor=center, inner sep=0}, "f"', from=2-1, to=2-2]
	\arrow["{1_f}"{description}, draw=none, from=0, to=1]
\end{tikzcd}\]
    \item horizontal and vertical pasting are subject to strict associativity as well as the following interchange rule:
\[\begin{tikzcd}
	A & {A'} & {A''} \\
	B & {B'} & {B''} \\
	C & {C'} & {C''}
	\arrow["f", from=1-1, to=1-2]
	\arrow[""{name=0, anchor=center, inner sep=0}, "F"', "\shortmid"{marking}, from=1-1, to=2-1]
	\arrow["{f'}", from=1-2, to=1-3]
	\arrow[""{name=1, anchor=center, inner sep=0}, "{F'}"{description}, from=1-2, to=2-2]
	\arrow[""{name=2, anchor=center, inner sep=0}, "{F''}", "\shortmid"{marking}, from=1-3, to=2-3]
	\arrow["{g}"{description}, from=2-1, to=2-2]
	\arrow[""{name=3, anchor=center, inner sep=0}, "G"', "\shortmid"{marking}, from=2-1, to=3-1]
	\arrow["{g'}"{description}, from=2-2, to=2-3]
	\arrow["h"{name=4, anchor=center, inner sep=0}, "{G'}"{description}, from=2-2, to=3-2]
	\arrow[""{name=5, anchor=center, inner sep=0}, "{G''}", "\shortmid"{marking}, from=2-3, to=3-3]
	\arrow["{h}"', from=3-1, to=3-2]
	\arrow["{h'}"',from=3-2, to=3-3]
	\arrow["\phi"{description}, draw=none, from=1, to=0]
	\arrow["{\phi'}"{description}, draw=none, from=2, to=1]
	\arrow["\psi"{description}, draw=none, from=3, to=4]
	\arrow["{\psi'}"{description}, draw=none, from=5, to=4]
\end{tikzcd} = 
\begin{tikzcd}[row sep=large]
	A & {A'} & {A''} \\
	\\
	C & {C'} & {C''}
	\arrow[""{name=0, anchor=center, inner sep=0}, "f", from=1-1, to=1-2]
	\arrow["GF"', "\shortmid"{marking}, from=1-1, to=3-1]
	\arrow[""{name=1, anchor=center, inner sep=0}, "{f'}", from=1-2, to=1-3]
	\arrow["{G'F'}"{description}, from=1-2, to=3-2]
	\arrow["{G''F''}", "\shortmid"{marking}, from=1-3, to=3-3]
	\arrow[""{name=2, anchor=center, inner sep=0}, "h"', from=3-1, to=3-2]
	\arrow[""{name=3, anchor=center, inner sep=0}, "{h'}"', from=3-2, to=3-3]
	\arrow["{\psi \bullet \phi}"{description}, draw=none, from=0, to=2]
	\arrow["{\psi' \bullet \phi'}"{description}, draw=none, from=1, to=3]
\end{tikzcd}\]
\[ \rotatebox[origin=c]{90}{=} \hskip4.5cm \rotatebox[origin=c]{90}{=} \]
\[\begin{tikzcd}[column sep=large]
	A && {A''} \\
	B && {B''} \\
	C && {C''}
	\arrow[""{name=0, anchor=center, inner sep=0}, "{f'f}", from=1-1, to=1-3]
	\arrow["F"', "\shortmid"{marking}, from=1-1, to=2-1]
	\arrow["{F''}", "\shortmid"{marking}, from=1-3, to=2-3]
	\arrow[""{name=1, anchor=center, inner sep=0}, "{g'g}"{description}, from=2-1, to=2-3]
	\arrow["G"', "\shortmid"{marking}, from=2-1, to=3-1]
	\arrow["{G''}", "\shortmid"{marking}, from=2-3, to=3-3]
	\arrow[""{name=2, anchor=center, inner sep=0}, "{h'h}"', from=3-1, to=3-3]
	\arrow["{\phi'\circ\phi}"{description}, draw=none, from=0, to=1]
	\arrow["{\psi'\circ\psi}"{description}, draw=none, from=1, to=2]
\end{tikzcd} = 
\begin{tikzcd}[sep=large]
	A && {A''} \\
	\\
	C && {C''}
	\arrow["{f'f}", from=1-1, to=1-3]
	\arrow[""{name=0, anchor=center, inner sep=0}, "GF"', "\shortmid"{marking}, from=1-1, to=3-1]
	\arrow[""{name=1, anchor=center, inner sep=0}, "{G''F''}", from=1-3, to=3-3]
	\arrow["{h'h}"', from=3-1, to=3-3]
	\arrow["\begin{array}{c} (\psi'\circ\psi)\bullet(\phi'\circ\phi) \\ = (\psi\bullet\phi)\circ(\psi'\bullet\phi') \end{array}"{description}, draw=none, from=0, to=1]
\end{tikzcd}\]
\end{itemize}
\end{definition}

\begin{remark}
   We should emphasize that the double categories we are going to construct are strict; this departs in particular from a very well-known example people may have in mind when broadly speaking of double categories, namely the double category whose objects are categories, horizontal maps are functors, vertical maps are \emph{distributors} (a.k.a. \emph{profunctors}) and double cells are natural transformations; this actually forms a \emph{pseudo} double category as its vertical component only is a bicategory (since composition of distributors is weak), and whose coherences are relaxed accordingly. The examples we are going to consider actually do not share this relational flavour; in particular they wont be instance of \emph{equipments}, an umbrella term describing the kind of relational double categories involving distributors-like notions as horizontal cells. In such (pseudo) double categories, there is a certain dissymmetry between the two classes of 1-cells, one being more general than the other, the other embedding in the second one in the same way as functors canonically define adjoint pairs of distributors. A contrario, in our case, we will consider a more symmetric and stricter kind of double category where the horizontal and vertical cells are more like two dual classes of functors, without one being more "relational" than the other. In the last section, we will describe a peculiar class of double categories our example of interest will fit in. 
\end{remark}

\begin{example}
    For any 2-category $ \mathcal{K}$, there is a \emph{lax quintet} (resp. oplax quintet) double category $ \mathcal{K}^{\Box}_\lax$ (resp. $ \mathcal{K}^{\Box}_\oplax$)  whose objects are those of $\mathcal{K}$, vertical and horizontal morphisms are any morphisms in $\mathcal{K}$, and double cells are lax (resp. oplax) squares. In fact $ \mathcal{K}^{\Box}_\lax$ and $ \mathcal{K}^{\Box}_\oplax$ are the same thing up to a transposition duality.
\end{example}

There are several double categorical analogues to functors, depending on the level of strictness required on the horizontal and vertical components. It is more frequent to consider either strict or \emph{pseudo double functors}, which are strict horizontally and pseudo vertically, but in our case, the relevant notion is the one introduced in \cite{shulman2007comparing}[Definition 6.1] we recall below:

\begin{definition}
    A \emph{double pseudofunctor} between double categories $ \mathbb{F} : \mathbb{C} \rightarrow \mathbb{D}$ is the data of 
    \begin{itemize}
        \item a function $ \mathbb{F} : \textup{Ob}_\mathbb{C} \rightarrow \textup{Ob}_\mathbb{D} $ at the level of objects 
        \item a horizontal component $ h\mathbb{F} : h\mathbb{C} \rightarrow h\mathbb{D}$
        \item a vertical component $ v\mathbb{F} : v\mathbb{C} \rightarrow v\mathbb{D}$
        \item and for every double cell in $ \mathbb{C}$ a double cell in $\mathbb{D}$
        \[\begin{tikzcd}
	A & B \\
	C & D
	\arrow["f", from=1-1, to=1-2]
	\arrow["H"', from=1-1, "\shortmid"{marking}, to=2-1]
	\arrow["\phi"{description}, draw=none, from=1-1, to=2-2]
	\arrow["K", "\shortmid"{marking}, from=1-2, to=2-2]
	\arrow["g"', from=2-1, to=2-2]
\end{tikzcd} 
\hskip1cm \begin{tikzcd}
	\mathbb{F}A & \mathbb{F}B \\
	\mathbb{F}C & \mathbb{F}D
	\arrow["{h\mathbb{F}f}", from=1-1, to=1-2]
	\arrow["{v\mathbb{F}H}"', "\shortmid"{marking}, from=1-1, to=2-1]
	\arrow["{\mathbb{F}\phi}"{description}, draw=none, from=1-1, to=2-2]
	\arrow["{v\mathbb{F}K}", "\shortmid"{marking}, from=1-2, to=2-2]
	\arrow["{h\mathbb{F}g}"', from=2-1, to=2-2]
\end{tikzcd}\]
    \item at each object $ A$, an horizontal and vertical invertible unitor cells
\[\begin{tikzcd}
	\mathbb{F}A & \mathbb{F}A \\
	\mathbb{F}A & \mathbb{F}A
	\arrow["{\mathbb{F}\id_A}", from=1-1, to=1-2]
	\arrow["{1_A}"'{inner sep=.8ex}, "\shortmid"{marking}, equals, from=1-1, to=2-1]
	\arrow["{h_A}"{description}, draw=none, from=1-1, to=2-2]
	\arrow["{1_A}"{inner sep=.8ex}, "\shortmid"{marking}, equals, from=1-2, to=2-2]
	\arrow["{\id_{\mathbb{F}A}}"', equals, from=2-1, to=2-2]
\end{tikzcd} \hspace{1cm} 
\begin{tikzcd}
	\mathbb{F}A & \mathbb{F}A \\
	\mathbb{F}A & \mathbb{F}A
	\arrow["{\id_{\mathbb{F}A}}", equals, from=1-1, to=1-2]
	\arrow["{1_{\mathbb{F}A}}"'{inner sep=.8ex}, "\shortmid"{marking}, equals, from=1-1, to=2-1]
	\arrow["{v_A}"{description}, draw=none, from=1-1, to=2-2]
	\arrow["{\mathbb{F}1_A}"{inner sep=.8ex}, "\shortmid"{marking}, from=1-2, to=2-2]
	\arrow["{\id_{\mathbb{F}A}}"', equals, from=2-1, to=2-2]
\end{tikzcd}\]
    \item  at each composable pair of horizontal and vertical 1-cells, a corresponding horizontal and vertical invertible compositor cells:
\[\begin{tikzcd}
	\mathbb{F}A && \mathbb{F}C \\
	\mathbb{F}A & \mathbb{F}B & \mathbb{F}C
	\arrow[""{name=0, anchor=center, inner sep=0}, "{\mathbb{F}hf}", from=1-1, to=1-3]
	\arrow["\shortmid"{marking}, equals, from=1-1, to=2-1]
	\arrow["\shortmid"{marking}, equals, from=1-3, to=2-3]
	\arrow["\mathbb{F}f"', from=2-1, to=2-2]
	\arrow["\mathbb{F}h"', from=2-2, to=2-3]
	\arrow["{h_{f,h}}"{description}, draw=none, from=0, to=2-2]
\end{tikzcd} \hspace{1cm}
\begin{tikzcd}
	\mathbb{F}A & \mathbb{F}A \\
	\mathbb{F}B \\
	\mathbb{F}C &\mathbb{F}C
	\arrow[equals, from=1-1, to=1-2]
	\arrow["\mathbb{F}G"'{inner sep=.8ex}, "\shortmid"{marking}, from=1-1, to=2-1]
	\arrow[""{name=0, anchor=center, inner sep=0}, "\mathbb{F}KG", from=1-2, to=3-2]
	\arrow["\mathbb{F}K"'{inner sep=.8ex}, "\shortmid"{marking}, from=2-1, to=3-1]
	\arrow[equals, from=3-1, to=3-2]
	\arrow["{v_{G,K}}"{description}, draw=none, from=0, to=2-1]
\end{tikzcd}\]
    
    \end{itemize}
which are moreover subject to the following conditions:\begin{itemize}
    \item the vertical and horizontal associators and unitors make both $h\mathbb{F}$ and $v\mathbb{F}$ pseudofunctors
    \item those four coherences expressing naturality:
\[\begin{tikzcd}
	{\mathbb{F}A} && {\mathbb{F}C} \\
	{\mathbb{F}A} & {\mathbb{F}B} & {\mathbb{F}C} \\
	{\mathbb{F}A'} & {\mathbb{F}B'} & {\mathbb{F}C'}
	\arrow["{\mathbb{F}gf}", from=1-1, to=1-3]
	\arrow[""{name=0, anchor=center, inner sep=0}, "\shortmid"{marking}, equals, from=1-1, to=2-1]
	\arrow[""{name=1, anchor=center, inner sep=0}, "\shortmid"{marking}, equals, from=1-3, to=2-3]
	\arrow["{\mathbb{F}f}"{description}, from=2-1, to=2-2]
	\arrow[""{name=2, anchor=center, inner sep=0}, "{\mathbb{F}H}"'{inner sep=.8ex}, "\shortmid"{marking}, from=2-1, to=3-1]
	\arrow["{\mathbb{F}g}"{description}, from=2-2, to=2-3]
	\arrow[""{name=3, anchor=center, inner sep=0}, "{\mathbb{F}G}"{description}, from=2-2, to=3-2]
	\arrow[""{name=4, anchor=center, inner sep=0}, "{\mathbb{F}K}"{inner sep=.8ex}, "\shortmid"{marking}, from=2-3, to=3-3]
	\arrow["{\mathbb{F}k}"', from=3-1, to=3-2]
	\arrow["{\mathbb{F}l}"', from=3-2, to=3-3]
	\arrow["{h_{gf}}"{description}, draw=none, from=1, to=0]
	\arrow["{\mathbb{F}\phi}"{description}, draw=none, from=3, to=2]
	\arrow["{\mathbb{F}\psi}"{description}, draw=none, from=4, to=3]
\end{tikzcd} = 
\begin{tikzcd}
	{\mathbb{F}A} && {\mathbb{F}C} \\
	{\mathbb{F}A'} && {\mathbb{F}C'} \\
	{\mathbb{F}A'} & {\mathbb{F}B'} & {\mathbb{F}C'}
	\arrow["{\mathbb{F}gf}", from=1-1, to=1-3]
	\arrow[""{name=0, anchor=center, inner sep=0}, "{\mathbb{F}H}"'{inner sep=.8ex}, "\shortmid"{marking}, from=1-1, to=2-1]
	\arrow[""{name=1, anchor=center, inner sep=0}, "{\mathbb{F}K}"{inner sep=.8ex}, "\shortmid"{marking}, from=1-3, to=2-3]
	\arrow["{\mathbb{F}kl}"{description}, from=2-1, to=2-3]
	\arrow[""{name=2, anchor=center, inner sep=0}, "\shortmid"{marking}, equals, from=2-1, to=3-1]
	\arrow[""{name=3, anchor=center, inner sep=0}, "\shortmid"{marking}, equals, from=2-3, to=3-3]
	\arrow["{\mathbb{F}k}"', from=3-1, to=3-2]
	\arrow["{\mathbb{F}l}"', from=3-2, to=3-3]
	\arrow["{\mathbb{F}\psi \circ \phi}"{description}, draw=none, from=1, to=0]
	\arrow["{h_{kl}}"{description}, draw=none, from=3, to=2]
\end{tikzcd}\]
\[\begin{tikzcd}
	{\mathbb{F}A} & {\mathbb{F}A'} & {\mathbb{F}A'} \\
	{\mathbb{F}B} & {\mathbb{F}B'} \\
	{\mathbb{F}C} & {\mathbb{F}C'} & {\mathbb{F}C'}
	\arrow[""{name=0, anchor=center, inner sep=0}, "{\mathbb{F}H}", from=1-1, to=1-2]
	\arrow["{\mathbb{F}f}"'{inner sep=.8ex}, "\shortmid"{marking}, from=1-1, to=2-1]
	\arrow["{\mathbb{F}k}"{description}, from=1-2, to=2-2]
	\arrow[""{name=1, anchor=center, inner sep=0}, equals, from=1-3, to=1-2]
	\arrow["{\mathbb{F}lk}"{inner sep=.8ex}, "\shortmid"{marking}, from=1-3, to=3-3]
	\arrow[""{name=2, anchor=center, inner sep=0}, "{\mathbb{F}G}"{description}, from=2-1, to=2-2]
	\arrow["{\mathbb{F}g}"'{inner sep=.8ex}, "\shortmid"{marking}, from=2-1, to=3-1]
	\arrow["{\mathbb{F}l}"{description},from=2-2, to=3-2]
	\arrow[""{name=3, anchor=center, inner sep=0}, "{\mathbb{F}K}"', from=3-1, to=3-2]
	\arrow[""{name=4, anchor=center, inner sep=0}, equals, from=3-3, to=3-2]
	\arrow["{\mathbb{F}\phi}"{description}, draw=none, from=2, to=0]
	\arrow["{\mathbb{F}\psi}"{description}, draw=none, from=3, to=2]
	\arrow["{h_{gf}}"{description}, draw=none, from=4, to=1]
\end{tikzcd} = 
\begin{tikzcd}
	{\mathbb{F}A} & {\mathbb{F}A} & {\mathbb{F}A'} \\
	{\mathbb{F}B} \\
	{\mathbb{F}C} & {\mathbb{F}C} & {\mathbb{F}C'}
	\arrow[equals, from=1-1, to=1-2]
	\arrow["{\mathbb{F}f}"'{inner sep=.8ex}, "\shortmid"{marking}, from=1-1, to=2-1]
	\arrow["{\mathbb{F}H}", from=1-2, to=1-3]
	\arrow[""{name=0, anchor=center, inner sep=0}, "{\mathbb{F}fg}"{description}, from=1-2, to=3-2]
	\arrow[""{name=1, anchor=center, inner sep=0}, "{\mathbb{F}lk}"{inner sep=.8ex}, "\shortmid"{marking}, from=1-3, to=3-3]
	\arrow["{\mathbb{F}g}"'{inner sep=.8ex}, "\shortmid"{marking}, from=2-1, to=3-1]
	\arrow[equals, from=3-2, to=3-1]
	\arrow["{\mathbb{F}K}"', from=3-2, to=3-3]
	\arrow["{v_{fg}}"{description}, draw=none, from=0, to=2-1]
	\arrow["{\mathbb{F}\psi \bullet\phi}"{description}, draw=none, from=1, to=0]
\end{tikzcd}\]

\[\begin{tikzcd}
	{\mathbb{F}A} & {\mathbb{F}A} & {\mathbb{F}B} & {\mathbb{F}B} \\
	{\mathbb{F}A} & {\mathbb{F}A} & {\mathbb{F}B} & {\mathbb{F}B}
	\arrow[""{name=0, anchor=center, inner sep=0}, equals, from=1-1, to=1-2]
	\arrow[""{name=1, anchor=center, inner sep=0}, "{\mathbb{F}f}", from=1-2, to=1-3]
	\arrow["{\mathbb{F}1_A}"{description}, from=1-2, to=2-2]
	\arrow[""{name=2, anchor=center, inner sep=0}, equals, from=1-3, to=1-4]
	\arrow["{\mathbb{F}1_B}"{description}, from=1-3, to=2-3]
	\arrow["\shortmid"{marking}, equals, from=2-1, to=1-1]
	\arrow[""{name=3, anchor=center, inner sep=0}, equals, from=2-1, to=2-2]
	\arrow[""{name=4, anchor=center, inner sep=0}, "{\mathbb{F}f}"', from=2-2, to=2-3]
	\arrow[""{name=5, anchor=center, inner sep=0}, equals, from=2-3, to=2-4]
	\arrow["\shortmid"{marking}, equals, from=2-4, to=1-4]
	\arrow["{h_A^{-1}}"{description}, draw=none, from=0, to=3]
	\arrow["{\mathbb{F}1_f}"{description}, draw=none, from=1, to=4]
	\arrow["{h_B}"{description}, draw=none, from=2, to=5]
\end{tikzcd} =  
\begin{tikzcd}
	{\mathbb{F}A} & {\mathbb{F}B} \\
	{\mathbb{F}A} & {\mathbb{F}B}
	\arrow[""{name=0, anchor=center, inner sep=0}, "{\mathbb{F}f}", from=1-1, to=1-2]
	\arrow["\shortmid"{marking}, equals, from=1-1, to=2-1]
	\arrow["\shortmid"{marking}, equals, from=1-2, to=2-2]
	\arrow[""{name=1, anchor=center, inner sep=0}, "{\mathbb{F}f}"', from=2-1, to=2-2]
	\arrow["{1_{\mathbb{F}f}}"{description}, draw=none, from=0, to=1]
\end{tikzcd}\]
\[\begin{tikzcd}
	{\mathbb{F}A} & {\mathbb{F}A} \\
	{\mathbb{F}A} & {\mathbb{F}A} \\
	{\mathbb{F}B} & {\mathbb{F}B} \\
	{\mathbb{F}B} & {\mathbb{F}B}
	\arrow[equals, from=1-1, to=1-2]
	\arrow[""{name=0, anchor=center, inner sep=0}, "\shortmid"{marking}, equals, from=1-1, to=2-1]
	\arrow[""{name=1, anchor=center, inner sep=0}, "\shortmid"{marking}, equals, from=1-2, to=2-2]
	\arrow[""{name=2, anchor=center, inner sep=0}, "{\mathbb{F}f}"'{inner sep=.8ex}, "\shortmid"{marking}, from=2-1, to=3-1]
	\arrow["{\mathbb{F}\id_A}"{description}, from=2-2, to=2-1]
	\arrow[""{name=3, anchor=center, inner sep=0}, "{\mathbb{F}f}"{inner sep=.8ex}, "\shortmid"{marking}, from=2-2, to=3-2]
	\arrow[""{name=4, anchor=center, inner sep=0}, "\shortmid"{marking}, equals, from=3-1, to=4-1]
	\arrow["{\mathbb{F}\id_B}"{description}, from=3-2, to=3-1]
	\arrow[""{name=5, anchor=center, inner sep=0}, "\shortmid"{marking}, equals, from=3-2, to=4-2]
	\arrow[equals, from=4-1, to=4-2]
	\arrow["{h_A^{-1}}"{description}, draw=none, from=1, to=0]
	\arrow["{\mathbb{F}\id_f}"{description}, draw=none, from=3, to=2]
	\arrow["{h_B}"{description}, draw=none, from=5, to=4]
\end{tikzcd} = 
\begin{tikzcd}
	{\mathbb{F}A} & {\mathbb{F}A} \\
	{\mathbb{F}B} & {\mathbb{F}B}
	\arrow[""{name=0, anchor=center, inner sep=0}, "{\mathbb{F}f}"'{inner sep=.8ex}, "\shortmid"{marking}, from=1-1, to=2-1]
	\arrow[equals, from=1-2, to=1-1]
	\arrow[""{name=1, anchor=center, inner sep=0}, "{\mathbb{F}f}"{inner sep=.8ex}, "\shortmid"{marking}, from=1-2, to=2-2]
	\arrow[equals, from=2-2, to=2-1]
	\arrow["{\id_{\mathbb{F}f}}"{description}, draw=none, from=1, to=0]
\end{tikzcd}\]
\end{itemize}

\end{definition}

\begin{remark}
    A more complete account of the first set of coherences expressing separate pseudofunctoriality of the vertical and horizontal components can be found at \cite{shulman2007comparing}[Definition 6.1 (v) and (vi)] but we will not need to refer explicitly to them in the following, only to the naturality coherences.
\end{remark}

In those axioms, we see that composites of vertical and horizontal maps, in any of the two possible order, exist only formally without any operation returning a composite 1-cell: only squares intertwining such formal composites are considered, as well as their different pasting.

\subsection{Double categories of sites, sheafification as a double pseudofunctor}

We saw there was two possible 2-categories of sites $ \Site^\flat$ and $ \Site^\sharp$ depending whether we considered morphisms or comorphisms as 1-cells. Here we exhibit those two 2-categories respectively as the horizontal and vertical components of a double category, on which sheafification will behave as a double pseudofunctor to the quintet double category of topoi.

\begin{definition}
   Define the double category $ \Site^\natural$ having as objects (small generated) sites, as horizontal arrows morphisms of sites, as vertical arrows comorphisms of sites, and as double cells
\[\begin{tikzcd}
	{(\mathcal{A},M)} & {(\mathcal{B},L)} \\
	{(\mathcal{C},J)} & {(\mathcal{D},K)}
	\arrow["f", from=1-1, to=1-2]
	\arrow["G"', "\shortmid"{marking}, from=1-1, to=2-1]
	\arrow["\phi"{description}, draw=none, from=1-2, to=2-1]
	\arrow["K", "\shortmid"{marking}, from=1-2, to=2-2]
	\arrow["h"', from=2-1, to=2-2]
\end{tikzcd}\]
lax-squares of the following form
    \[\begin{tikzcd}
	{(\mathcal{A},M)} & {(\mathcal{B},L)} \\
	{(\mathcal{C},J)} & {(\mathcal{D},K)}
	\arrow["f", from=1-1, to=1-2]
	\arrow["G"', from=1-1, to=2-1]
	\arrow["\phi"', shorten <=7pt, shorten >=7pt, Rightarrow, from=1-2, to=2-1]
	\arrow["K", from=1-2, to=2-2]
	\arrow["h"', from=2-1, to=2-2]
\end{tikzcd}\]

\end{definition}

Now recall that (Grothendieck) topoi form a bicategory, so we can consider the quintet double category:

\begin{definition}
    In the following we will denote as $\Top^\square$ the quintet double category of topoi, where both horizontal and vertical cells are geometric morphisms, and whose double cells are lax squares filled by geometric transformations.
\end{definition}

\begin{remark}[{\cite{niefield2011glueing}[Example 3.4]}]
   There is another double category of topoi $ \Top^\Lex$ where the horizontal cells are geometric morphisms, vertical cells are lex functors, but this is not the same double categorical structure. We will only focus on the quintet double category in this work. 
\end{remark}

We can then enounce properly the central result of this section:

\begin{theorem}
There is a horizontal contravariant, vertically covariant double pseudofunctor from the double category sites to the quintet double category of topoi, with the pseudofunctors $\Sh$ and $C$ as the horizontal and vertical components respectively:
\[\begin{tikzcd}
	{(\Site^{\flat})^{\op}} & {\Site^{\natural}} & {(\Site^{\sharp})^{\co}} \\
	& {\Top^{\square}}
	\arrow["h", from=1-1, to=1-2]
	\arrow["\Sh"', from=1-1, to=2-2]
	\arrow[from=1-2, to=2-2]
	\arrow["v"', from=1-3, to=1-2]
	\arrow["C", from=1-3, to=2-2]
\end{tikzcd}\]   
 \end{theorem}  

\begin{proof}
    We want to construct a sheafification double pseudofunctor whose horizontal component is the pseudofunctor $ \Sh$ form morphisms and whose vertical component is the pseudofunctor $ C$ for comorphisms. This requires to understand how to manage double cells. A double cell in $\Site^\natural$ is a lax square as below
\[\begin{tikzcd}
	{(\mathcal{A},M)} & {(\mathcal{B},L)} \\
	{(\mathcal{C},J)} & {(\mathcal{D},K)}
	\arrow["f", from=1-1, to=1-2]
	\arrow["G"', from=1-1, to=2-1]
	\arrow["\phi"', shorten <=7pt, shorten >=7pt, Rightarrow, from=1-2, to=2-1]
	\arrow["K", from=1-2, to=2-2]
	\arrow["h"', from=2-1, to=2-2]
\end{tikzcd}\]
where $ f,h$ are morphisms of sites and $ G,K$ comorphisms of sites respectively. Then from \cref{Adjoint squares} we have a canonical cross 2-cell $ \overline{\phi} : \lext_f \rest_g \Rightarrow \rest_k \lext_h $ between the presheaf categories. By commutation of the inverse image, both $ \lext_f$ and $ \lext_h$ lift through the sheafifications functors to $\Sh(f)^*$ and $ \Sh(h)^*$ respectively, and the same is true for the associated $ \rest_G$ and $\rest_K$ which lift to the inverse image parts $ C_G^*$ and $C_H^*$ respectively. Hence the sheafification along $ \mathfrak{a}_L : \widehat{\mathcal{B}} \rightarrow \widehat{\mathcal{B}}_L$ produces a 2-cell 
\[\begin{tikzcd}
	{\widehat{\mathcal{A}}_M} & {\widehat{\mathcal{B}}_L} \\
	{\widehat{\mathcal{C}}_J} & {\widehat{\mathcal{D}}_K}
	\arrow["{\Sh(f)^*}", from=1-1, to=1-2]
	\arrow["{\widetilde{\phi}^\flat}"', shorten <=8pt, shorten >=8pt, Rightarrow, from=1-1, to=2-2]
	\arrow["{C_G^*}", from=2-1, to=1-1]
	\arrow["{\Sh(h)^*}"', from=2-1, to=2-2]
	\arrow["{C_K^*}"', from=2-2, to=1-2]
\end{tikzcd}\]

This is a 2-cell between geometric morphisms
\[\begin{tikzcd}
	{\widehat{\mathcal{A}}_M} & {\widehat{\mathcal{B}}_L} \\
	{\widehat{\mathcal{C}}_J} & {\widehat{\mathcal{D}}_K}
	\arrow["{C_G}"', from=1-1, to=2-1]
	\arrow["{\widetilde{\phi}}"', shorten <=8pt, shorten >=8pt, Rightarrow, from=1-1, to=2-2]
	\arrow["{\Sh(f)}"', from=1-2, to=1-1]
	\arrow["{C_K}", from=1-2, to=2-2]
	\arrow["{\Sh(h)}", from=2-2, to=2-1]
\end{tikzcd}\]

    This defines the action of $F$ on double-cells. We must now provide unitors and compositors and prove the coherences consitutive of a double pseudofunctor, but as seen below, they just come from the separate pseudofunctoriality of $\Sh$ and $ C$. For the sake of explicitness we at least provide those unitors and compositors. At a site $ (\mathcal{C},J)$, the underlying identity functor $1_\mathcal{C} : \mathcal{C} \rightarrow \mathcal{C}$ defines both a morphism and a comorphism of sites, and plays both the roles of the horizontal and vertical identities of $(\mathcal{C},J)$. The horizontal unitor is induced as follows: the denseness of the Yoneda embedding provides at each presheaf $X$ in $\widehat{\mathcal{C}}$ a canonical isomorphism $ \alpha_X : \colim \; \hirayo_\mathcal{C} \downarrow X \simeq X$, which is, by the colimit expression of left Kan extensions, a canonical natural isomorphism $ \alpha : \lext_{1_\mathcal{C}} \simeq 1_{\widehat{\mathcal{C}}} $; in particular when $X$ is a $J$-sheaf, the colimit decomposition above gives a canonical isomorphism $ \widetilde{\alpha}_X : \colim \; \mathfrak{a}_J\hirayo \downarrow X \simeq X$, which is the component of the inverse image part of the unitor 
\[\begin{tikzcd}
	{\widehat{\mathcal{C}}_J} & {\widehat{\mathcal{C}}_J} \\
	{\widehat{\mathcal{C}}_J} & {\widehat{\mathcal{C}}_J}
	\arrow["{\Sh(\id_{(\mathcal{C},J)})}", from=1-1, to=1-2]
	\arrow[""{name=0, anchor=center, inner sep=0}, "\shortmid"{marking}, equals, from=1-1, to=2-1]
	\arrow[""{name=1, anchor=center, inner sep=0}, "\shortmid"{marking}, equals, from=1-2, to=2-2]
	\arrow["{\id_{\widehat{C}_J}}"', from=2-1, to=2-2]
	\arrow["{h_{(\mathcal{C},J)}}"{description}, draw=none, from=1, to=0]
\end{tikzcd}\]
    On the other hand the vertical unitor is constructed as follows: recall that the inverse image part of the geometric morphism induced by $1_\mathcal{C}$ is $ C_{1_{(\mathcal{C},J)}}^* = \mathfrak{a}_J \rest_{1_\mathcal{C}} \mathfrak{i}_J$; but the restriction functor $ \rest_{1_\mathcal{C}}$ is actually equal to the identity functor of $\widehat{\mathcal{C}}$, and then inversibility of the counit $ \epsilon : \mathfrak{a}_J \mathfrak{i}_J \simeq 1_{\widehat{\mathcal{C}}_J}$ provides the inverse image part of the unitor:
\[\begin{tikzcd}
	{\widehat{\mathcal{C}}_J} & {\widehat{\mathcal{C}}_J} \\
	{\widehat{\mathcal{C}}_J} & {\widehat{\mathcal{C}}_J}
	\arrow[equals, from=1-1, to=1-2]
	\arrow[""{name=0, anchor=center, inner sep=0}, "{1_{\widehat{\mathcal{C}}_J}}"'{inner sep=.8ex}, "\shortmid"{marking}, equals, from=1-1, to=2-1]
	\arrow[""{name=1, anchor=center, inner sep=0}, "{C_{1_{(\mathcal{C},J)}}}"{inner sep=.8ex}, "\shortmid"{marking}, from=1-2, to=2-2]
	\arrow[equals, from=2-1, to=2-2]
	\arrow["{v_{(\mathcal{C},J)}}"{description}, draw=none, from=1, to=0]
\end{tikzcd}\]
    The compositors are similarly induced from pseudofunctoriality of $\Sh$ and $C$: for $ f,h$ a pair of composable morphisms of sites, the universal property of the left Kan extensions provides with a natural isomorphism $ \lext_{hf} \simeq \lext_h \lext_f$, whose sheafification provides the inverse image part of the horizontal compositor; for a pair of composable comorphisms of sites $ G,K$, one has an equality $ \rest_{KG} = \rest_{G}\rest_K$, and then, for $K$ is a comorphism of sites, one has $ \mathfrak{a}_J \rest_G \simeq C^*_{G} \mathfrak{a}_M$ as explained at \cref{geometric morphism from comorphism of sites}, for which one has an isomorphism:
    \begin{align*}
        C^*_{G}C^*_K
        &= C^*_G \mathfrak{a}_M \rest_K \mathfrak{i}_L \\
        &\simeq \mathfrak{a}_J \rest_G \rest_K \mathfrak{i}_L \\
        &= \mathfrak{a}_J \rest_{KG} \mathfrak{i}_L \\
        &= C^*_{KG}
    \end{align*}
    which we can use as the inverse image part of the vertical compositor. 
    
    We must now prove those data to satisfy the coherences constitutive of a double pseudofunctor. Those ensuring pseudofunctoriality of the horizontal and vertical component are immediate, by the very pseudofunctoriality of $ \Sh$ and $C$ respectively. On the other hand, the two naturality conditions are corollary of the naturality of the mate construction established in our first section: the coherences relative to pasting established at \cref{naturality of mate for pasting} at the level of presheaves are sent to the desired coherences at the level of sheaves, while the coherences relative to horizontal and vertical units are induced by the two identities established at \cref{naturality of mate for units}.
\end{proof}
     

\begin{division}[The problem with oplax cells]
    Dually we could define another double category having the same data except that a double cell as above corresponds to a 2-cell with the converse orientation
   \[\begin{tikzcd}
	{(\mathcal{A},M)} & {(\mathcal{B},L)} \\
	{(\mathcal{C},J)} & {(\mathcal{D},K)}
	\arrow["f", from=1-1, to=1-2]
	\arrow["G"', from=1-1, to=2-1]
	\arrow["\phi"', shorten <=7pt, shorten >=7pt, Rightarrow, from=2-1, to=1-2]
	\arrow["K", from=1-2, to=2-2]
	\arrow["h"', from=2-1, to=2-2]
\end{tikzcd}\]
    However, whereas were able to manage lax squares as double cells, there is an dissymmetry in topos theory, which is that for a morphism of sites, only $\lext_f$ and $\rest_f$ are visible in the induced geometric morphism, while for a comorphism, we retain only $\rest_G$ and $\rext_G$: but in the first case, $\rext_f$ is not part of the induced geometric morphism $\Sh(f)$ (that is, does not restrict between the categories of sheaves, as in general $ \Sh(f)_*$ will not have a further right adjoint) while in the second case, $ \lext_G$ is not part of $C_G$, as the inverse image part $ C_G^*$ may not have a further left adjoint. As a consequence, we cannot expect to construct in the same way a sheafification double with oplax squares rather than lax squares as double cells: indeed, an oplax square as below
\[\begin{tikzcd}
	{(\mathcal{A},M)} & {(\mathcal{B},L)} \\
	{(\mathcal{C},J)} & {(\mathcal{D},K)}
	\arrow["f", from=1-1, to=1-2]
	\arrow[""{name=0, anchor=center, inner sep=0}, "G"', from=1-1, to=2-1]
	\arrow[""{name=1, anchor=center, inner sep=0}, "K", from=1-2, to=2-2]
	\arrow["h"', from=2-1, to=2-2]
	\arrow["\phi", shorten <=8pt, shorten >=8pt, Rightarrow, from=0, to=1]
\end{tikzcd}\]
induces the following Beck-Chevalley 2-cells
\[\begin{tikzcd}
	{\widehat{\mathcal{A}}} & {\widehat{\mathcal{B}}} \\
	{\widehat{\mathcal{C}}} & {\widehat{\mathcal{D}}}
	\arrow[""{name=0, anchor=center, inner sep=0}, "{\lext_G}"', from=1-1, to=2-1]
	\arrow["{\rest_f}"', from=1-2, to=1-1]
	\arrow[""{name=1, anchor=center, inner sep=0}, "{\lext_K}", from=1-2, to=2-2]
	\arrow["{\rest_h}", from=2-2, to=2-1]
	\arrow["{\overline{\phi}^\flat}", shorten <=6pt, shorten >=6pt, Rightarrow, from=0, to=1]
\end{tikzcd} \hskip1cm 
\begin{tikzcd}
	{\widehat{\mathcal{A}}} & {\widehat{\mathcal{B}}} \\
	{\widehat{\mathcal{C}}} & {\widehat{\mathcal{D}}}
	\arrow["{\rext_f}", from=1-1, to=1-2]
	\arrow[""{name=0, anchor=center, inner sep=0}, "{\rest_G}", from=2-1, to=1-1]
	\arrow["{\rext_h}"', from=2-1, to=2-2]
	\arrow[""{name=1, anchor=center, inner sep=0}, "{\rest_K}"', from=2-2, to=1-2]
	\arrow["{\overline{\phi}^\sharp}"', shorten <=6pt, shorten >=6pt, Rightarrow, from=1, to=0]
\end{tikzcd}\]
and in both case either the vertical or horizontal components are not part of an induced geometric morphism in general.
\end{division}

\subsection{Locally exact squares}

One could ask what are the lax (resp. oplax) squares that get inverted by the sheafification double functors: this would provide a generalization of the notion of \emph{exact squares}, which we recall below.

\begin{definition}
    Recall that a lax square
\[\begin{tikzcd}
	{\mathcal{A}} & {\mathcal{B}} \\
	{\mathcal{C}} & {\mathcal{D}}
	\arrow["f", from=1-1, to=1-2]
	\arrow["g"', from=1-1, to=2-1]
	\arrow["\phi"', shorten <=6pt, shorten >=6pt, Rightarrow, from=1-2, to=2-1]
	\arrow["k", from=1-2, to=2-2]
	\arrow["h"', from=2-1, to=2-2]
\end{tikzcd}\]
is said to be \emph{exact} (or satisfies \emph{Beck-Chevalley condition}) if the associated 2-cell below is invertible
\[\begin{tikzcd}
	{\widehat{\mathcal{A}}} & {\widehat{\mathcal{B}}} \\
	{\widehat{\mathcal{C}}} & {\widehat{\mathcal{D}}}
	\arrow["{\lext_f}", from=1-1, to=1-2]
	\arrow["{\overline{\phi} \atop \simeq}"{description}, shorten <=6pt, shorten >=6pt, draw =none, from=1-1, to=2-2]
	\arrow["{\rest_g}", from=2-1, to=1-1]
	\arrow["{\lext_h}"', from=2-1, to=2-2]
	\arrow["{\rest_k}"', from=2-2, to=1-2]
\end{tikzcd}
\]
\end{definition}

\begin{division}
    There are several equivalent formulations of this condition. First, using the description of $ \overline{\phi} $ given at \cref{exact squares explicitly}, this means that for each presheaf $ X : \mathcal{C}^{\op} \rightarrow \Set$ the precomposite along $k^{\op}$ of the left Kan extension of $X$ along $g^{\op}$ is exhibited as the left Kan extension of $X \circ g^{\op}$ along $f^{\op}$:
    \[ (\lan_{h^{\op}} X ) \circ g^{\op} \simeq \lan_{f^{\op}} X \circ g^{\op} \]
    In other words, that pasting the opposite square of $\phi$ above the extension of a presheaf returns the extension of the restricted presheaf. 

    From \cite{guitart1980carrésexact}, we know that the condition of being exact also amounts to the condition that the following induced relation between the representable distributors
    \[ \mathcal{B}(1,f) \otimes \mathcal{C}(g,1) \rightarrow \mathcal{D}(k,h)  \]
    is invertible; more explicitly the component of this transformation at $ (b,c)$ and $a$ is given by
    \[  \mathcal{B}(b,f(a)) \otimes \mathcal{C}(g(a),c) \rightarrow \mathcal{D}(k(b),h(c)) \]
    which sends a pair $ (u : b \rightarrow f(a), v : g(a) \rightarrow c)$ to the composite 
\[\begin{tikzcd}
	{k(b)} & {kf(a)} & {hg(a)} & {h(c)}
	\arrow["{k(u)}", from=1-1, to=1-2]
	\arrow["{\phi_a}", from=1-2, to=1-3]
	\arrow["{h(v)}", from=1-3, to=1-4]
\end{tikzcd}\]
and an arrow $ ((s,t): u \rightarrow u', (z,w) : v \rightarrow v')$ to the composite square
\[\begin{tikzcd}
	{k(b)} & {kf(a)} & {hg(a)} & {h(c)} \\
	{k(b')} & {kf(a')} & {hg(a')} & {h(c')}
	\arrow["{k(u)}", from=1-1, to=1-2]
	\arrow["{k(s)}"', from=1-1, to=2-1]
	\arrow["{\phi_a}", from=1-2, to=1-3]
	\arrow["{k(t)}"{description}, from=1-2, to=2-2]
	\arrow["{h(v)}", from=1-3, to=1-4]
	\arrow["{h(z)}"{description}, from=1-3, to=2-3]
	\arrow["{h(w)}", from=1-4, to=2-4]
	\arrow["{k(u')}"', from=2-1, to=2-2]
	\arrow["{\phi_{a'}}"', from=2-2, to=2-3]
	\arrow["{h(v')}"', from=2-3, to=2-4]
\end{tikzcd}\]
By the coend formula of composite of distributors, this means that at any $(b,c)$ we have an isomorphism 
\[ \int^{a \in \mathcal{A}} \mathcal{B}[b, f(a)] \times \mathcal{C}[g(a),c] \simeq \mathcal{B}[k(b), h(c)] \]
but we recognize in this coend formula the expression of the composite $ \lext_f \rest_g $
while on the other hand, the distributor $ \mathcal{D}[k,h]$ decomposes as
\[ \mathcal{D}[k,h] \simeq \mathcal{D}[k,1] \otimes \mathcal{D}[1,h] \]
in which again we recognize through its coend formulation the composite $\rest_k \lext_h $.

From this observation, as explained in \cite{guitart1980carrésexact}, exactness of $\phi$ amounts to asking that: \begin{itemize}
    \item for any $u : k(b) \rightarrow h(c)$ there is a pair of arrows $s : k(b) \rightarrow kf(a)$, $ t : hg(a) \rightarrow h(c)$ such that one has a twisted decomposition
\[\begin{tikzcd}
	{k(b)} & {h(c)} \\
	{kf(a)} & {hg(a)}
	\arrow["u", from=1-1, to=1-2]
	\arrow["{k(s)}"', from=1-1, to=2-1]
	\arrow["{\phi_a}"', from=2-1, to=2-2]
	\arrow["{h(t)}"', from=2-2, to=1-2]
\end{tikzcd}\]
    \item and any such two decompositions are related by a zigzag: if $ (s,a,t)$ and $(s',a',t')$ are two twisted decompositions of a same arrow $u$, then there is a zigzag relating $ a$ and $a'$ in $\mathcal{A}$ with intermediate decomposition
\[\begin{tikzcd}
	& {kf(a')} & {hg(a')} \\
	{k(b)} & \cdots & \cdots & {h(c)} \\
	& {kf(a)} & {hg(a)}
	\arrow["{\phi_{a'}}", from=1-2, to=1-3]
	\arrow["{h(t')}", from=1-3, to=2-4]
	\arrow["{k(s')}", from=2-1, to=1-2]
	\arrow[from=2-1, to=2-2]
	\arrow["{k(s)}"', from=2-1, to=3-2]
	\arrow[from=2-2, to=1-2]
	\arrow[from=2-2, to=2-3]
	\arrow[from=2-2, to=3-2]
	\arrow[from=2-3, to=1-3]
	\arrow[from=2-3, to=2-4]
	\arrow[from=2-3, to=3-3]
	\arrow["{\phi_a}"', from=3-2, to=3-3]
	\arrow["{h(t)}"', from=3-3, to=2-4]
\end{tikzcd}\]
\end{itemize}

Equivalently, by applying the graph construction of the distributors above, one can see that $ \phi$ is exact if and only if the comparison functor below is final
\[ g \downarrow c \rightarrow k \downarrow hc  \]
\end{division}

    If we now consider a square relating morphisms and comorphisms between sites, it may happen that the natural transformation $\widetilde{\phi}$ is not yet invertible in the presheaf category, but becomes invertible after sheafification: such will be the double cells of $\Site^\natural$ that are inverted by the double functor $ \Sh$.

\begin{definition}
    A lax square as below between sites (underlying a double cell of $\Site^\natural$)
    \[\begin{tikzcd}
	{(\mathcal{A},M)} & {(\mathcal{B},L)} \\
	{(\mathcal{C},J)} & {(\mathcal{D},K)}
	\arrow["f", from=1-1, to=1-2]
	\arrow["G"', from=1-1, to=2-1]
	\arrow["\phi"', shorten <=7pt, shorten >=7pt, Rightarrow, from=1-2, to=2-1]
	\arrow["K", from=1-2, to=2-2]
	\arrow["h"', from=2-1, to=2-2]
\end{tikzcd}\]
will be said \emph{locally exact} if the corresponding transformation below is invertible
\[\begin{tikzcd}
	{\widehat{\mathcal{A}}_M} & {\widehat{\mathcal{B}}_L} \\
	{\widehat{\mathcal{C}}_J} & {\widehat{\mathcal{D}}_K}
	\arrow["{\Sh(f)^*}", from=1-1, to=1-2]
	\arrow["{\widetilde{\phi}^\flat}"', shorten <=8pt, shorten >=8pt, Rightarrow, from=1-1, to=2-2]
	\arrow["{C_G^*}", from=2-1, to=1-1]
	\arrow["{\Sh(h)^*}"', from=2-1, to=2-2]
	\arrow["{C_K^*}"', from=2-2, to=1-2]
\end{tikzcd}\]
\end{definition}

To characterize such cells, we are going to refer to the following result from \cite{caramello2020denseness} which characterizes thoses morphisms of diagrams into a site which induce an isomorphism between the corresponding colimits in the sheaf topos:

\begin{proposition}[{\cite{caramello2020denseness}[Lemma 2.21]}]\label{Olivia characterization of relative cofinality}
    Let $(\mathcal{B},L)$ be a site together with a 2-cell as below
\[\begin{tikzcd}
	{\mathcal{X}} && {\mathcal{Y}} \\
	& {\mathcal{B}}
	\arrow["G", from=1-1, to=1-3]
	\arrow[""{name=0, anchor=center, inner sep=0}, "F"', from=1-1, to=2-2]
	\arrow[""{name=1, anchor=center, inner sep=0}, "{H}", from=1-3, to=2-2]
	\arrow["\xi", shorten <=6pt, shorten >=6pt, Rightarrow, from=0, to=1]
\end{tikzcd}\]
then the comparison map in $\widehat{\mathcal{B}}_L$
\[\begin{tikzcd}
	{\underset{\mathcal{X}}{\colim} \; \mathfrak{a}_L\hirayo F} & {\underset{\mathcal{Y}}{\colim} \; \mathfrak{a}_L\hirayo H}
	\arrow["{\widetilde{\xi}}", from=1-1, to=1-2]
\end{tikzcd}\]
is an isomorphism if and only if we have the following two conditions:\begin{itemize}
    \item for any $ u : b \rightarrow H(y)$ there is a $L$-cover $ (s_i : b_i \rightarrow b)_{i \in I}$ together with a family of arrows $ (t_i : b_i \rightarrow F(x_i))_{i \in I}$ such that for each $i \in I$ the composites below
    \[\begin{tikzcd}
	& {b_i} \\
	{F(x_i)} && b \\
	{HG(x_i)} && {H(y)}
	\arrow["{t_i}"', from=1-2, to=2-1]
	\arrow["{s_i}", from=1-2, to=2-3]
	\arrow["{\xi_{x_i}}"', from=2-1, to=3-1]
	\arrow["u", from=2-3, to=3-3]
    \end{tikzcd}\]
    are in the same component of the comma $ b_i \downarrow H$
    \item for any span as below 
\[\begin{tikzcd}
	& b \\
	{F(x)} && {F(x')}
	\arrow["v"', from=1-2, to=2-1]
	\arrow["{v'}", from=1-2, to=2-3]
\end{tikzcd}\]
such that the composites $ \xi_x v$ and $ \xi_{x'}v'$ are the the same connected component of $ b \downarrow H$, there is a $L$-covering family $ (s_i : b_i \rightarrow b)_{i\in I}$ such that at any $i$ the composites $ us_i$, $u's_i$ are in the same connected component of $ b_i \downarrow F$.  
\end{itemize}
\end{proposition}

\begin{definition}
    We will say that a morphism of diagram $ (G,\xi)$ satisfying the condition of \cref{Olivia characterization of relative cofinality} is \emph{relatively $L$-cofinal}
\end{definition}

\begin{proposition}
    A lax square $ \phi$ between sites as above is locally exact if and only if, for each $ c$ in $ \mathcal{C}$, the morphism of diagrams below is relatively $L$-cofinal (where $ L$ is the topology on $\mathcal{B}$)
\[\begin{tikzcd}
	{G\downarrow c} && {K\downarrow h(c)} \\
	& {\mathcal{B}}
	\arrow["{\phi \downarrow c}", from=1-1, to=1-3]
	\arrow[""{name=0, anchor=center, inner sep=0}, "{f \pi^G_0}"', from=1-1, to=2-2]
	\arrow[""{name=1, anchor=center, inner sep=0}, "{\pi_0^K}", from=1-3, to=2-2]
	\arrow["{=}", draw=none, from=0, to=1]
\end{tikzcd}\]
\end{proposition}

\begin{proof}
To show that the comparison map $ \mathfrak{a}_L \overline{\phi}^\flat : \mathfrak{a}_L\lext_f \rest_G \Rightarrow \mathfrak{a}_L \rest_K \lext_h $ is a natural isomorphism of sheaves, it suffices by denseness of $\mathcal{C}$ in $\widehat{\mathcal{C}}_J$ to prove that at any $c$ in $\mathcal{C}$, the comparison map in $\widehat{\mathcal{B}}$ at the representable $ \overline{\phi}^{\flat}_{\hirayo_c} : \lext_f \rest_G (\hirayo_c) \Rightarrow  \rest_K \lext_h(\hirayo_c) $ is sent to an isomorphism after sheafification. But by computation of left Kan extension, the presheaf $ \lext_f \rest_G(\hirayo_c)$ is exhibited as the colimit in $\widehat{\mathcal{B}}$ of the functor 
\[\begin{tikzcd}
	{G\downarrow c} & {\mathcal{A}} & {\mathcal{B}} & {\widehat{\mathcal{B}}}
	\arrow["{ \pi^G_0}", from=1-1, to=1-2]
	\arrow["f", from=1-2, to=1-3]
	\arrow[hook, from=1-3, to=1-4]
\end{tikzcd}\]
On the other hand, the Kan extension of a representable is again representable 
\[ \lext_h(\hirayo_c) \simeq \mathcal{D}[1, h(c)] \]
so we recognize $ \rest_K \lext_h(\hirayo_c) = \mathcal{D}[K,h(c)]$, which is the colimit in $ \widehat{B}$ of the diagram
\[\begin{tikzcd}
	{K \downarrow h(c)} & {\mathcal{B}} & {\widehat{\mathcal{B}}}
	\arrow["{\pi_{0}^K}", from=1-1, to=1-2]
	\arrow[hook, from=1-2, to=1-3]
\end{tikzcd}\]
Finally observe that the comparison map $ \overline{\phi}^{\flat}_{\hirayo_c}$ is exhibited as the comparison map between the colimits in $\widehat{\mathcal{B}}$ induced by the morphism of diagram given by $\phi\downarrow c : G\downarrow c \rightarrow K\downarrow h(c) $ sending $ u: G(a) \rightarrow c$ to the composite
\[\begin{tikzcd}
	{Kf(a)} & {hG(a)} & {h(c)}
	\arrow["{\phi_a}", from=1-1, to=1-2]
	\arrow["{h(u)}", from=1-2, to=1-3]
\end{tikzcd}\]
so that by \cref{Olivia characterization of relative cofinality}, it is sent in $\widehat{\mathcal{B}}_L$ to an isomorphism if and only if $ \phi \downarrow c$ is relatively $L$-cofinal.
\end{proof}

In the following we discuss a few examples of locally exact squares, recovering in some cases classical results of topos theory as a local exactness condition for some suited double cell in $\Site^\natural$.


\subsection{Comma and cocomma squares}

An archetypical example of exact square is provided by comma squares, whose exactness encapsulates the pointwiseness of Kan extensions, hence its crucial role in formal category theory; it is dually true that cocomma squares are exact. This lead us to ask first what are the correct notions of comma and cocomma in the double category $ \Site^\natural$, which implies to equip them with convenient topologies: they will provide trivial examples of locally exact squares, which are exact even before localization. 

\begin{lemma}[{\cite{guitart1980carrésexact}[example 1.14 (2)]}]
 For any $ f, G$, the comma square below is exact 
 \[\begin{tikzcd}
	{G \downarrow f} & {\mathcal{D}} \\
	{\mathcal{C}} & {\mathcal{B}}
	\arrow["{\pi_0}", from=1-1, to=1-2]
	\arrow[""{name=0, anchor=center, inner sep=0}, "{\pi_1}"', from=1-1, to=2-1]
	\arrow[""{name=1, anchor=center, inner sep=0}, "G", from=1-2, to=2-2]
	\arrow["f"', from=2-1, to=2-2]
	\arrow["{\lambda_{G,F}}"', shorten <=7pt, shorten >=7pt, Rightarrow, from=1, to=0]
\end{tikzcd}\]
\end{lemma}

Indeed, for left Kan extensions are pointwise, we have a natural isomorphism 
\[ \rest_G \lext_f \simeq \lext_{\pi_0} \rest_{\pi_1} \]
Then in particular, in the case where $ f$ is flat, (we will see below that is implies that $ \pi_0$ is also flat), seeing then $ f$, $ \pi_0$ and $G$, $\pi_1$ respectively as morphisms and comorphisms for the trivial topologies, this means that we have an invertible 2-cell
\[\begin{tikzcd}
	{\widehat{G \downarrow f}} & {\widehat{\mathcal{B}}} \\
	{\widehat{\mathcal{C}}} & {\widehat{\mathcal{D}}}
	\arrow[""{name=0, anchor=center, inner sep=0}, "{C_{\pi_1}}"', from=1-1, to=2-1]
	\arrow["{\Sh(\pi_0)}"', from=1-2, to=1-1]
	\arrow[""{name=1, anchor=center, inner sep=0}, "{C_G}", from=1-2, to=2-2]
	\arrow["{\Sh(f)}", from=2-2, to=2-1]
	\arrow["{\widetilde{\lambda_{G,F}} \atop \simeq}"{description}, draw=none, from=0, to=1]
\end{tikzcd}\]

Now one may ask for an improved version of this result involving topologies on the categories. We first show below that the comma category from a comorphism of sites to a morphism of sites can be itself equipped with a canonical topology such that the projections inherit the appropriate structure:

\begin{lemma}
    Let $ f : (\mathcal{C},J) \rightarrow (\mathcal{D},K)$ a morphism of sites and $ G : (\mathcal{B},L) \rightarrow (\mathcal{D},K)$ a comorphism of sites; then there is a topology $J_{G,f}$ on $ G \downarrow f$ that makes $ \pi_0 : G \downarrow f \rightarrow \mathcal{B}$ a morphism of sites and $\pi_1 : G \downarrow f \rightarrow \mathcal{C}$ a comorphism of sites so we have a double cell of $\Site^\natural$ as below:
\[\begin{tikzcd}
	{({G \downarrow f}, J_{G,f})} & {({\mathcal{B}},L)} \\
	{({\mathcal{C}},J)} & {({\mathcal{D}},K)}
	\arrow["{\pi_0}", from=1-1, to=1-2]
	\arrow[""{name=0, anchor=center, inner sep=0}, "{\pi_1}"', "\shortmid"{marking}, from=1-1, to=2-1]
	\arrow[""{name=1, anchor=center, inner sep=0}, "{G}", "\shortmid"{marking}, from=1-2, to=2-2]
	\arrow["f"', from=2-1, to=2-2]
	\arrow["{{\lambda_{G,f}}}"{description}, draw=none, from=1, to=0]
\end{tikzcd}\]
\end{lemma}

\begin{proof}
 Take an object of the comma $ u : G(b) \rightarrow f(c)$: since $ f$ is a morphism of sites, for any sieve $S$ generated by a cover $ (\gamma_i : c_j \rightarrow c)_{j \in I} $ in $J(c)$, the (sieve $ f[S]$ generated by the) family $ (f(\gamma_j) : f(c_j) \rightarrow f(c))_{j \in J}$ is in $K(f(c))$, and by stability of coverages, so is the pullback sieve $ u^*f[S]$ containing all arrows $v : d \rightarrow G(b)$ such that $uv$ factorizes through some $f(\gamma_j)$. But as $G$ is a comorphism of sites, there exists some sieve $ R$ in $L(b)$, generated by some cover $ (\beta_i : b_i \rightarrow b)_{i \in I}$ such that the image $ G[R]$ is in $u^*f[S]$ (that is, $G^{-1}(u^*f[S])$ is in $ L(b)$).
 
 Hence we can define the topology $J_{G,f}$ whose generating sieves on an object $ u : G(b) \rightarrow f(c)$ of the comma are those generated from the families as below 
 \[  G \downarrow S = \Bigg{\{} 
\begin{tikzcd}
	{G(b')} & {f(c')} \\
	{G(b)} & {f(c)}
	\arrow[from=1-1, to=1-2]
	\arrow["{G(\beta)}"', from=1-1, to=2-1]
	\arrow["{f(\gamma)}", from=1-2, to=2-2]
	\arrow["u"', from=2-1, to=2-2]
\end{tikzcd} \mid \beta \in G^{-1}(u^*f[S]), \gamma \in S \Bigg{\}} \]
consisting of all the factorizations of precomposites $ u G(\beta)$ through some $f(\gamma)$ for $ \beta$ in $G^{-1}(u^*f[S])$. Then for any such $G\downarrow S$, the sieve $ {\pi_0}[G\downarrow S]$ is $ G^{-1}(u^*f[S])$ which is $L$-covering by the observation above so that $ \pi_0$ preserves covers, while its projection $ \pi_1(G \downarrow S)$ is contained in $S$ as it contains all arrows in $S$ through which some composite $ uv$ with $v \in G^{-1}(u^*f[S])$ factorizes, and it can be used to lift any cover $S$ of $f(c) = \pi_1(u)$, which makes $ \pi_1$ a comorphism. 

To finish, we have to prove that $ \pi_0$ is covering-flat. Let be $ (v_i : b \rightarrow b_i)_{i \in I}$ a cone in $\mathcal{B}$ with $I$ finite, together with a specified diagram $ (u_i : G(b_i) \rightarrow f(c_i))_{i \in I}$ in $ G \downarrow f$; we must show that the set of all arrows $ w : b' \rightarrow b$ such that there exists a cone 
\[\begin{tikzcd}
	{G(b'')} & {f(c'')} \\
	{G(b_i)} & {f(c_i)}
	\arrow["{u''}", from=1-1, to=1-2]
	\arrow["{G(t_i)}"', from=1-1, to=2-1]
	\arrow["{f(a_i)}", from=1-2, to=2-2]
	\arrow["{u_i}"', from=2-1, to=2-2]
\end{tikzcd}\]
together with a factorization 
\[\begin{tikzcd}
	{b'} & b \\
	{b''} & {b_i}
	\arrow["w", from=1-1, to=1-2]
	\arrow["{\exists s}"', from=1-1, to=2-1]
	\arrow["{v_i}", from=1-2, to=2-2]
	\arrow["{t_i}"', from=2-1, to=2-2]
\end{tikzcd}\]
is $L$-covering. For $f$ is covering flat (being a morphism of sites), the set of all arrows $ w: d \rightarrow G(b)$ such that there is a cone $(a_i :c \rightarrow c_i)_{i\in I}$ together with a factorization
\[\begin{tikzcd}
	d & {f(c)} \\
	{G(b)} \\
	{G(b_i)} & {f(c_i)}
	\arrow["{\exists s}", from=1-1, to=1-2]
	\arrow["w"', from=1-1, to=2-1]
	\arrow["{f(a_i)}", from=1-2, to=3-2]
	\arrow["{G(v_i)}"', from=2-1, to=3-1]
	\arrow["{u_i}"', from=3-1, to=3-2]
\end{tikzcd}\]
is $K$-covering over $G(b)$; but $G$ is a comorphism of sites, so the restriction of this sieve along $G$ is $L$-covering: but if $w : b' \rightarrow b$ is in this restriction, then one has a cone $(a_i :c \rightarrow c_i)_{i \in I}$ as above together with a factorization
\[\begin{tikzcd}
	{G(b')} & {f(c)} \\
	{G(b)} \\
	{G(b_i)} & {f(c_i)}
	\arrow["{\exists s}", from=1-1, to=1-2]
	\arrow["{G(w)}"', from=1-1, to=2-1]
	\arrow["{f(a_i)}", from=1-2, to=3-2]
	\arrow["{G(v_i)}"', from=2-1, to=3-1]
	\arrow["{u_i}"', from=3-1, to=3-2]
\end{tikzcd}\]
so that $w$ is in the sieve we must show to be covering.
\end{proof}

Since the Beck-Chevalley 2-cell $\overline{\lambda_{G,f}}^\flat$ induced by $ \lambda_{G,f}$ is already invertible in the presheaf topos $ \widehat{\mathcal{B}}$ by exactness of comma squares, it remains invertible in the sheaf topos and we have:

\begin{corollary}\label{comma are locally exact}
    Let $f : (\mathcal{C},J) \rightarrow (\mathcal{D},K)$ be a morphism and $G :( \mathcal{B}, L) \rightarrow (\mathcal{D},K) $ a comorphism, we have an invertible 2-cell between geometric morphisms
\[\begin{tikzcd}
	{\widehat{{G \downarrow f}}_{J_{G,f}}} & {\widehat{{\mathcal{B}}}_L} \\
	{\widehat{{\mathcal{C}}}_J} & {\widehat{{\mathcal{D}}}_K}
	\arrow[""{name=0, anchor=center, inner sep=0}, "{C_{\pi_1}}"', from=1-1, to=2-1]
	\arrow["{\Sh(\pi_0)}"', from=1-2, to=1-1]
	\arrow[""{name=1, anchor=center, inner sep=0}, "{C_G}", from=1-2, to=2-2]
	\arrow["{\Sh(f)}", from=2-2, to=2-1]
	\arrow["{{\lambda_{G,F}} \atop \simeq}"{description}, draw=none, from=1, to=0]
\end{tikzcd}\]
\end{corollary}

\begin{division}
    A dual class of exact squares is provided by \emph{cocomma squares}. For two functors $ f : \mathcal{C} \rightarrow \mathcal{B}$, $G : \mathcal{C} \rightarrow \mathcal{D}$, the cocomma $ f \uparrow G$ is the category defined as follows:\begin{itemize}
        \item it has as objects the disjoint union of the objects of $ \mathcal{B}$ and $\mathcal{D}$, denoted $ (0,b)$ and $(1,d)$
        \item it has as morphisms those $(0,v) :(0,b) \rightarrow (0,d')$ with $ v : b \rightarrow b'$ in $\mathcal{B}$, those $(1,w) : (1, d) \rightarrow (1,d')$ for $ w: d \rightarrow d'$ in $\mathcal{D}$, together with the classes of formal composites
\[\begin{tikzcd}
	{(0,b)} & {(0,f(c))} & {(1,g(c))} & {(1,d)}
	\arrow["{(0,v)}", from=1-1, to=1-2]
	\arrow["{\lambda_c}", from=1-2, to=1-3]
	\arrow["{(1,w)}", from=1-3, to=1-4]
\end{tikzcd}\]
    modulo the zigzag relation:
\[\begin{tikzcd}
	& {(0,f(c))} & {(1,g(c))} \\
	{(0,b)} & \cdots & \cdots & {(1,d)} \\
	& {(0,f(c'))} & {(1,g(c'))}
	\arrow["{\lambda_c}", from=1-2, to=1-3]
	\arrow["{(1,w)}", from=1-3, to=2-4]
	\arrow["{(0,v)}", from=2-1, to=1-2]
	\arrow[from=2-1, to=2-2]
	\arrow["{(0,v')}"', from=2-1, to=3-2]
	\arrow["{(0,f(u))}"{description}, from=2-2, to=1-2]
	\arrow["{(0,f(u'))}"{description}, from=2-2, to=3-2]
	\arrow["{(1,g(u))}"{description}, from=2-3, to=1-3]
	\arrow[from=2-3, to=2-4]
	\arrow["{(1,g(u'))}"{description}, from=2-3, to=3-3]
	\arrow["{\lambda_{c'}}"', from=3-2, to=3-3]
	\arrow["{(1,w')}"', from=3-3, to=2-4]
\end{tikzcd}\]
    \end{itemize}
\end{division}

\begin{lemma}[{\cite{guitart1980carrésexact}[example 1.14 (3)] }]
   For any $f$, $G$, the cocomma square below is exact
\[\begin{tikzcd}
	{\mathcal{C}} & {\mathcal{B}} \\
	{\mathcal{D}} & {f \uparrow G}
	\arrow["f", from=1-1, to=1-2]
	\arrow[""{name=0, anchor=center, inner sep=0}, "G"', from=1-1, to=2-1]
	\arrow[""{name=1, anchor=center, inner sep=0}, "{\iota_0}", from=1-2, to=2-2]
	\arrow["{\iota_1}"', from=2-1, to=2-2]
	\arrow["{\lambda_{f,G}}"', shorten <=7pt, shorten >=7pt, Rightarrow, from=1, to=0]
\end{tikzcd}\]
\end{lemma}

Then, seeing again $f$, $\iota_1$ and $ G$, $\iota_0$ as morphisms and comorphisms for the trivial topologies, the square as below is transported to an invertible 2-cell in $\Top$
\[\begin{tikzcd}
	{\widehat{\mathcal{C}}} & {\widehat{\mathcal{B}}} \\
	{\widehat{\mathcal{D}}} & {\widehat{f \uparrow G}}
	\arrow[""{name=0, anchor=center, inner sep=0}, "{C_G}"', from=1-1, to=2-1]
	\arrow["{\Sh(f)}"', from=1-2, to=1-1]
	\arrow[""{name=1, anchor=center, inner sep=0}, "{C_{\iota_0}}", from=1-2, to=2-2]
	\arrow["{\Sh(\iota_1)}", from=2-2, to=2-1]
	\arrow["{\overline{\lambda_{f,G}}}"', shorten <=7pt, shorten >=7pt, Rightarrow, from=1, to=0]
\end{tikzcd}\]

\begin{lemma}
For a morphism of sites $f : (\mathcal{C},J) \rightarrow (\mathcal{B},L)$ and a comorphism of sites $G : (\mathcal{C},J) \rightarrow (\mathcal{D},K)$, there is a topology on $ f \uparrow G$ that makes $ \iota_0 : \mathcal{B} \rightarrow f \uparrow G$ a comorphism of sites and $ \iota_1 : \mathcal{D} \rightarrow f \uparrow G$ a morphism of sites so we have a double cell of $\Site^\natural$ as below
\[\begin{tikzcd}
	{(\mathcal{C},J)} & {(\mathcal{B},L)} \\
	{(\mathcal{D},K)} & {(f \uparrow G, J_{f,G})}
	\arrow["f", from=1-1, to=1-2]
	\arrow[""{name=0, anchor=center, inner sep=0}, "G"', "\shortmid"{marking}, from=1-1, to=2-1]
	\arrow[""{name=1, anchor=center, inner sep=0}, "{\iota_0}", "\shortmid"{marking}, from=1-2, to=2-2]
	\arrow["{\iota_1}"', from=2-1, to=2-2]
	\arrow["{\lambda_{f,G}}"{description}, draw=none, from=1, to=0]
\end{tikzcd}\]
\end{lemma}

\begin{proof}
    The topology $ J_{f,G}$ is defined by cases as generated as follows:\begin{itemize}
        \item on $(0,b)$ consider as generating sieves all ${\iota_0}[S]$ for $ S \in L(b)$; 
        \item on $(1,d)$ consider as generating all $ {\iota_1}[R]$ for $ R \in K(d)$.
    \end{itemize}
First let us show that $ \iota_0$ is a comorphism. For a cover form $ {\iota_0}[S]$, $S$ provides trivially a lift. Observe that sieves of the form $(\lambda_c a)^*R$ also are in the topology $ J_{f,G}$: if $R \in K(G(c))$, since $G$ is a comorphism of sites, there is some $T \in J(c)$ such that $ G[T] \leq R$, and as $f$ is a morphism of sites, $ f[T] \in L(f(c))$, so $a^*f[T] \in L(b)$, and then we have a lift $ \iota_0[a^*f[T]] \leq (\lambda_c a)^*R $. On the other hand $\iota_1$ clearly sends $K$-covering sieves to generating sieves.

To finish we must show $\iota_1$ to be covering flat. Let be a finite diagram $ (d_i)_{i \in I}$ in $\mathcal{D}$ together with a cone in $ f \uparrow G$ given by the classes under zigzag of arrows of the form 
\[\begin{tikzcd}
	{(0,b)} & {(0,f(c_i))} & {(1,G(c_i))} & {(1, d_i)}
	\arrow["{(0,u_i)}", from=1-1, to=1-2]
	\arrow["{\lambda_{c_i}}", from=1-2, to=1-3]
	\arrow["{(1, v_i)}", from=1-3, to=1-4]
\end{tikzcd}\]
    As $f$ is covering flat, the set $S$ of all $w: b' \rightarrow b$ such that there is a cone $ (a_i : c \rightarrow c_i)$ together with a factorization
\[\begin{tikzcd}
	{b'} & {f(c)} \\
	b & {f(c_i)}
	\arrow["{\exists s}", from=1-1, to=1-2]
	\arrow["w"', from=1-1, to=2-1]
	\arrow["{f(a_i)}", from=1-2, to=2-2]
	\arrow["{u_i}"', from=2-1, to=2-2]
\end{tikzcd}\]
    is $ L$-covering; but for any $w$ in this sieve $S$, one can consider the composite cone $ (v_iG(a_i) : G(c) \rightarrow d_i)_{i \in I}$ in $ \mathcal{D}$ and we have a factorization 
\[\begin{tikzcd}
	{(0,b')} & {(0,f(c))} & {(1,G(c))} \\
	{(0,b)} & {(0,f(c_i))} & {(1,G(c_i))} & {(1,d_i)}
	\arrow["{(0,s)}", from=1-1, to=1-2]
	\arrow["{(0,w)}"', from=1-1, to=2-1]
	\arrow["{\lambda_{c}}", from=1-2, to=1-3]
	\arrow["{(0,f(a_i))}"{description}, from=1-2, to=2-2]
	\arrow["{(1, G(a_i))}"{description}, from=1-3, to=2-3]
	\arrow["{(1, v_iG(a_i))}", from=1-3, to=2-4]
	\arrow["{(0,u_i)}"', from=2-1, to=2-2]
	\arrow["{\lambda_{c_i}}"', from=2-2, to=2-3]
	\arrow["{(1,v_i)}"', from=2-3, to=2-4]
\end{tikzcd}\]
    exhibiting $(0,w)$ as being in the sieve of arrows over $ (0,b)$ whose composite with the cone $ ((1,v_i)\lambda_{c_i}(0,u_i))_{i \in I}$ factorizes through a cone in the range of $ \iota_1$: hence this sieve contains $ {\iota_0}[S]$ which is covering by definition of $ J_{f,G}$.
\end{proof}

\subsection{Conjoints to companions, companions to conjoints}

  Many categorical notions are encapsulated by the exactness of some square, as listed in \cite{guitart1980carrésexact}[Examples 1.14]. Combined with our observation regarding sheafifications, such exactness criteria can be used to rephrase known results of topos theory as site-flavoured statements of formal category theory. This matter relates to a curious consequence of the horizontal contravariance of the sheafification functor: it interchanges two classes of double-cells, namely \emph{conjoints} and \emph{companions} squares between $ \Site^\natural$ and $ \Top^\square_\lax$. Let us first recall those two notions. 

\begin{definition}
  A horizontal cell $f$ and a vertical cell $g$ are said to be \emph{conjoint} (with $f$ being the \emph{right} and $g$ being the \emph{left} conjoints) if one has a pair of double cells as below
\[\begin{tikzcd}
	A & B \\
	A & A
	\arrow["f", from=1-1, to=1-2]
	\arrow[""{name=0, anchor=center, inner sep=0}, equals, from=1-1, to=2-1]
	\arrow[""{name=1, anchor=center, inner sep=0}, "G", "\shortmid"{marking}, from=1-2, to=2-2]
	\arrow[equals, from=2-1, to=2-2]
	\arrow["\epsilon"{description}, draw=none, from=1, to=0]
\end{tikzcd} \hskip1cm 
\begin{tikzcd}
	B & B \\
	A & B
	\arrow[equals, from=1-1, to=1-2]
	\arrow[""{name=0, anchor=center, inner sep=0}, "G"', "\shortmid"{marking}, from=1-1, to=2-1]
	\arrow[""{name=1, anchor=center, inner sep=0}, equals, from=1-2, to=2-2]
	\arrow["f"', from=2-1, to=2-2]
	\arrow["\eta"{description}, draw=none, from=1, to=0]
\end{tikzcd}\]
whose horizontal and vertical pasting satisfy
\[   \epsilon \bullet \eta = \id_G  \hskip1cm \eta \circ \epsilon  = 1_f \]
where $ 1_G$ and $\id_f$ are respectively the horizontal and vertical identities of $G$ and $f$.
\end{definition}

\begin{definition}
    A horizontal cell $f$ and a vertical cell $G$ are said to be \emph{companions} if one has a pair of double cells as below
\[\begin{tikzcd}
	A & B \\
	B & B
	\arrow["f", from=1-1, to=1-2]
	\arrow[""{name=0, anchor=center, inner sep=0}, "G"', "\shortmid"{marking}, from=1-1, to=2-1]
	\arrow[""{name=1, anchor=center, inner sep=0}, equals, from=1-2, to=2-2]
	\arrow[equals, from=2-1, to=2-2]
	\arrow["\phi"{description}, draw=none, from=1, to=0]
\end{tikzcd} \hskip1cm 
\begin{tikzcd}
	A & A \\
	A & B
	\arrow[equals, from=1-1, to=1-2]
	\arrow[""{name=0, anchor=center, inner sep=0}, equals, from=1-1, to=2-1]
	\arrow[""{name=1, anchor=center, inner sep=0}, "G", "\shortmid"{marking}, from=1-2, to=2-2]
	\arrow["f"', from=2-1, to=2-2]
	\arrow["\psi"{description}, draw=none, from=1, to=0]
\end{tikzcd}\]
whose horizontal and vertical pasting satisfy
\[ \phi \bullet \psi = \id_G \hskip1cm \psi \circ \phi = 1_f  \]
\end{definition}

In the double category $ \Site^\natural$, conjoints correspond to adjunctions between a morphisms and a comorphisms; but this presentation trough squares connects it with the following formulation in terms of exact squares:

\begin{division}
  A functor $ f$ is right adjoint to $G$ in $\Cat$ if and only if the following square below is exact
\[\begin{tikzcd}
	{\mathcal{C}} & {\mathcal{D}} \\
	{\mathcal{C}} & {\mathcal{C}}
	\arrow["f", from=1-1, to=1-2]
	\arrow[""{name=0, anchor=center, inner sep=0}, equals, from=1-1, to=2-1]
	\arrow[""{name=1, anchor=center, inner sep=0}, "G", from=1-2, to=2-2]
	\arrow[equals, from=2-1, to=2-2]
	\arrow["\epsilon"', shorten <=6pt, shorten >=6pt, Rightarrow, from=1, to=0]
\end{tikzcd}\]
as this expresses that one has a natural isomorphism
\[ \lext_f \simeq \rest_G \]
which in particular, applied to the identity functor of $ \mathcal{C}^{\op}$, exhibits $G^{\op}$ as the left Kan extension of $f^{\op}$, hence $ G$ as the left adjoint of $f$; the unit is induced as follows: by uniqueness of adjoints, one has also $ \rest_f \simeq \rext_G$, which, applied to the identity of $ \mathcal{D}^{\op}$, exhibits $ f^\op$ as a right Kan extension of $G^{\op}$ and the universal 2-cell provides the unit.
\[\begin{tikzcd}
	{\mathcal{C}} & {\mathcal{C}} \\
	{\mathcal{D}} & {\mathcal{C}}
	\arrow[""{name=0, anchor=center, inner sep=0}, "G"', from=1-1, to=2-1]
	\arrow[equals, from=1-2, to=1-1]
	\arrow["f"', from=2-1, to=2-2]
	\arrow[""{name=1, anchor=center, inner sep=0}, equals, from=2-2, to=1-2]
	\arrow["\eta"', shorten <=6pt, shorten >=6pt, Rightarrow, from=1, to=0]
\end{tikzcd}\]
This exhibits $ f,G$ as a conjoint pair in the double category $ \Cat^{\square}$. If now $ f$ is a flat functor and we put the trivial topology on $\mathcal{C}$, $\mathcal{D}$, so that $f$ is a morphism of sites and $G$ a comorphism of sites, then exactness of the double cell above says that the sheafification double functor sends this cell to an invertible companion pair (where the invertible cells are respectively the natural isomorphism $ \lext_f \simeq \rest_G$ and its inverse):
\[\begin{tikzcd}
	{\widehat{\mathcal{C}}} & {\widehat{\mathcal{D}}} \\
	{\widehat{\mathcal{C}}} & {\widehat{\mathcal{C}}}
	\arrow[""{name=0, anchor=center, inner sep=0}, equals, from=1-1, to=2-1]
	\arrow["{\Sh(f)}"', from=1-2, to=1-1]
	\arrow[""{name=1, anchor=center, inner sep=0}, "{C_G}", from=1-2, to=2-2]
	\arrow[equals, from=2-2, to=2-1]
	\arrow["{\widetilde{\epsilon} \atop\simeq}"{description}, draw=none, from=1, to=0]
\end{tikzcd} \hskip1cm 
\begin{tikzcd}
	{\widehat{\mathcal{D}}} & {\widehat{\mathcal{D}}} \\
	{\widehat{\mathcal{C}}} & {\widehat{\mathcal{D}}}
	\arrow[equals, from=1-1, to=1-2]
	\arrow[""{name=0, anchor=center, inner sep=0}, "{C_G}"', from=1-1, to=2-1]
	\arrow[""{name=1, anchor=center, inner sep=0}, equals, from=2-2, to=1-2]
	\arrow["{\Sh(f)}", from=2-2, to=2-1]
	\arrow["{\widetilde{\eta} \atop \simeq}"{description}, draw=none, from=0, to=1]
\end{tikzcd}\]
    In fact this observation raises no surprises. Indeed, companions and conjoints are preserved by any double pseudofunctor (see \cite{shulman2007comparing}[Proposition 6.7 and 6.8]). On the other hands, conjoints and companions are dual up to inverting horizontal 1-cells; hence the sheafification double functor $ \Sh$ sends conjoints to companions and conversely. But the point here is that it shed new light on the following sequence of results. 
\end{division}

\begin{corollary}
    A flat functor $f : \mathcal{C} \rightarrow \mathcal{D}$ has a left adjoint $G$ if and only if they both induce the same geometric morphism $ C_G \simeq \Sh(f) : \widehat{\mathcal{D}} \rightarrow \widehat{C}$.
\end{corollary}

\begin{division}
    In particular we can recover a classical result of topos theory stated for instance at \cite{elephant}[Scholium C2.5.2]: if we consider a morphism of sites $ f : (\mathcal{C},J) \rightarrow (\mathcal{D},K)$, it is known that a left adjoint $G$ to $f$ must be a comorphism of sites, and moreover, they induce a same geometric morphism $ \Sh(f) \simeq C_G$: this becomes an immediate consequence of our observation that $ \Sh$ sends conjoints to companions.
\end{division}

But now $f$ and $G$ need not be already adjoint in $\Cat$ to induce the same geometric morphism between sheaf topoi; this leads to the following notion:

\begin{definition}
    A morphism of sites $ f : (\mathcal{C}, J) \rightarrow (\mathcal{D},K)$ will be said to be \emph{weakly $K$-right adjoint} to a comorphism $G$ if there exists a $K$-exact square as below:
\[\begin{tikzcd}
	{(\mathcal{C},J)} & {(\mathcal{D},K)} \\
	{(\mathcal{C},J)} & {(\mathcal{C},J)}
	\arrow["f", from=1-1, to=1-2]
	\arrow[""{name=0, anchor=center, inner sep=0}, equals, from=1-1, to=2-1]
	\arrow[""{name=1, anchor=center, inner sep=0}, "G", from=1-2, to=2-2]
	\arrow[equals, from=2-1, to=2-2]
	\arrow["\epsilon"', shorten <=7pt, shorten >=7pt, Rightarrow, from=1, to=0]
\end{tikzcd}\]
\end{definition}

\begin{proposition}
    Let $f :(\mathcal{C},J) \rightarrow (\mathcal{D},K) $ be a morphism of sites and $G : (\mathcal{D},K) \rightarrow (\mathcal{C},J)$ a comorphism of sites; if $f$ is weakly $K$-right adjoint to $G$ then $ f$ and $G$ induce the same geometric morphism $\Sh(f) \simeq C_G : \Sh(\mathcal{D},K) \rightarrow \Sh(\mathcal{C},J)$ 
\end{proposition}

\begin{division}

In the following we examine exactness of another pair of dual squares. For a functor $ f : \mathcal{C} \rightarrow \mathcal{D}$, one has two possible identity squares:
\[\begin{tikzcd}
	{\mathcal{C}} & {\mathcal{C}} \\
	{\mathcal{C}} & {\mathcal{D}}
	\arrow[equals, from=1-1, to=1-2]
	\arrow[equals, from=1-1, to=2-1]
	\arrow["f", from=1-2, to=2-2]
	\arrow["f"', from=2-1, to=2-2]
\end{tikzcd} \hskip1cm  \begin{tikzcd}
	{\mathcal{C}} & {\mathcal{D}} \\
	{\mathcal{D}} & {\mathcal{D}}
	\arrow["f", from=1-1, to=1-2]
	\arrow["f"', from=1-1, to=2-1]
	\arrow[equals, from=1-2, to=2-2]
	\arrow[equals, from=2-1, to=2-2]
\end{tikzcd} \] 
and their corresponding mates are respectively the unit and counit of the $\lext_f \dashv \rest_f$ adjunction
\[\begin{tikzcd}
	{\widehat{C}} & {\widehat{C}} \\
	{\widehat{C}} & {\widehat{D}}
	\arrow[equals, from=1-1, to=1-2]
	\arrow[""{name=0, anchor=center, inner sep=0}, equals, from=2-1, to=1-1]
	\arrow["{\lext_f}"', from=2-1, to=2-2]
	\arrow[""{name=1, anchor=center, inner sep=0}, "{\rest_f}"', from=2-2, to=1-2]
	\arrow["\eta", shorten <=6pt, shorten >=6pt, Rightarrow, from=0, to=1]
\end{tikzcd} \hskip1cm 
\begin{tikzcd}
	{\widehat{C}} & {\widehat{D}} \\
	{\widehat{D}} & {\widehat{D}}
	\arrow["{\lext_f}", from=1-1, to=1-2]
	\arrow[""{name=0, anchor=center, inner sep=0}, "{\rest_f}", from=2-1, to=1-1]
	\arrow[equals, from=2-1, to=2-2]
	\arrow[""{name=1, anchor=center, inner sep=0}, equals, from=2-2, to=1-2]
	\arrow["\epsilon", shorten <=6pt, shorten >=6pt, Rightarrow, from=0, to=1]
\end{tikzcd}\]

If now we consider two sites $ (\mathcal{C},J)$ and $ (\mathcal{D},K)$ such that $ f$ is both a morphism and a comorphism of sites, then it induces two geometric morphisms $ \Sh(f) $ and $C_f $, with $ C_f$ left adjoint to $\Sh(f)$ in the 2-category $ \Top$. This allows to recover from simple exactness conditions two dual classes of geometric morphisms.
\end{division}

\begin{division}[Fully faithful functors]
   It is known that a functor $ f : \mathcal{C} \rightarrow \mathcal{D}$ is fully faithful if and only if its identity square below is exact, which simply expresses that the universal 2-cell of the left Kan extension is invertible
\[\begin{tikzcd}
	{\mathcal{C}} & {\mathcal{C}} \\
	{\mathcal{C}} & {\mathcal{D}}
	\arrow[equals, from=1-1, to=1-2]
	\arrow[equals, from=1-1, to=2-1]
	\arrow["f", from=1-2, to=2-2]
	\arrow["f"', from=2-1, to=2-2]
\end{tikzcd} \hskip1cm  1 \simeq \rest_f \lext_f \]
If $f$ is a flat functor and we equip $\mathcal{C}, \mathcal{D}$ with the trivial topologies, then $f$ is also a comorphism since covers are trivial, which exhibits the identity cell above as a companion double cell in $\Site^\natural$, and after sheafification it induces a pair conjoint squares in $\Top^\square$, where the left one corresponding to the unit $ 1 \simeq \rest_f\lext_f$ is invertible and the right one is given by the counit $ \epsilon: \lext_f \rest_f \Rightarrow 1$):
\[\begin{tikzcd}
	{\widehat{\mathcal{C}}} & {\widehat{\mathcal{C}}} \\
	{\widehat{\mathcal{C}}} & {\widehat{\mathcal{D}}}
	\arrow[equals, from=1-1, to=1-2]
	\arrow[""{name=0, anchor=center, inner sep=0}, equals, from=1-1, to=2-1]
	\arrow[""{name=1, anchor=center, inner sep=0}, "{C_f}", from=1-2, to=2-2]
	\arrow["{\Sh(f)}", from=2-2, to=2-1]
	\arrow["\simeq"{description}, draw=none, from=1, to=0]
\end{tikzcd}
\hspace{1cm}
\begin{tikzcd}
	{\widehat{\mathcal{C}}} & {\widehat{\mathcal{D}}} \\
	{\widehat{\mathcal{D}}} & {\widehat{\mathcal{D}}}
	\arrow[""{name=0, anchor=center, inner sep=0}, "{C_f}"', from=1-1, to=2-1]
	\arrow["{\Sh(f)}"', from=1-2, to=1-1]
	\arrow[""{name=1, anchor=center, inner sep=0}, equals, from=1-2, to=2-2]
	\arrow[equals, from=2-2, to=2-1]
	\arrow["\epsilon", between={0.2}{0.8}, Rightarrow, from=0, to=1]
\end{tikzcd}\]
Hence the geometric morphism $C_f$ is exhibited as an internally fully faithful left adjoint to $\Sh(f)$, and as $ \lext_f = \Sh(f)^*$ itself is fully faithful, this makes $\Sh(f)$ a local geometric morphism with $C_f$ its universally initial section. Then we recover the following result, which relates to \cite{caramello2020denseness}[Theorem 7.20] stating that a continuous comorphism induces the left adjoint of a local geometric morphism: 
\end{division}

\begin{proposition}
    Suppose that $f : (\mathcal{C}, J) \rightarrow (\mathcal{D},K)$ is both a morphism and a comorphism of sites. Then $\Sh(f)$ is a local geometric morphism if and only if the identity square of $f$ below is $J$-locally exact:
\[\begin{tikzcd}
	{(\mathcal{C},J)} & {(\mathcal{C},J)} \\
	{(\mathcal{C},J)} & {(\mathcal{D},K)}
	\arrow[equals, from=1-2, to=1-1]
	\arrow["f", from=1-2, to=2-2]
	\arrow[equals, from=2-1, to=1-1]
	\arrow["f"', from=2-1, to=2-2]
\end{tikzcd}\]
\end{proposition}

\begin{proof}
    If $f$ is both morphism and comorphism, the pairs of adjoints $ \lext_f \dashv \rest_f$ and $ \rest_f \dashv \rext_f$ induce adjoint geometric morphisms $  C_f \dashv \Sh(f)$; moreover, $J$-exactness of the square amounts to $C_f$ being a section of $ \Sh(f)$, making $\Sh(f)^*$ fully faithful. Hence this condition is equivalent to asking $\Sh(f)$ to be a local geometric morphism with section $ C_f$. 
\end{proof}

\begin{remark}
    The condition (iii) of \cite{caramello2020denseness}[7.20] establishes that $C_f$ is local if and only if $f$ is a $J$-fully faithful continuous comorphism of sites: following \cite{caramello2020denseness}[definition 5.14] this means that \begin{itemize}
        \item for any pair of maps $u,v : c \rightrightarrows c'$ in $\mathcal{C}$, if $f(u)=f(v)$, then there is a cover $ (s_i : c_i \rightarrow c)_{i \in I}$ in $J(c)$ such that for each $i\in I$ one has $ us_i= vs_i$;
        \item for any arrow $u : f(c) \rightarrow f(c')$ there is a cover $(s_i : c_i \rightarrow c)_{i \in I}$ in $J(c)$ together with a family of maps $(v_i : c_i \rightarrow c)_{i \in I}$ such that for each $i \in I$ one has $ u f(s_i) = f(vi)$.
    \end{itemize}
    On the other hand asking for $J$-exactness of the square above amounts to asking for the following functor to be $J$-cofinal at each $c$ in $\mathcal{C}$:
\[\begin{tikzcd}
	{\mathcal{C}/c} && {f\downarrow f(c)} \\
	& {\mathcal{C}}
	\arrow["{f/c}", from=1-1, to=1-3]
	\arrow["{\pi_0}"', from=1-1, to=2-2]
	\arrow["{\pi_0^f}", from=1-3, to=2-2]
\end{tikzcd}\]
In this case, applying \cref{Olivia characterization of relative cofinality}, we have that\begin{itemize}
    \item for any $s : c'' \rightarrow c'$ together with $u: f(c') \rightarrow f(c)$, there is a $J$-cover $ (s_i : c''_i \rightarrow c'')_{i \in I}$ and a family $(t_i : c_i'' \rightarrow c'_i)_{i\in I}$ with $ (u_i : c'_i \rightarrow c)_{i \in I}$ such that the following diagram is in the same connected component of $ c''_i \downarrow \pi_0^f$
\[\begin{tikzcd}[sep=small]
	& {c_i''} \\
	{c_i'} && {c''} \\
	&& c
	\arrow["{t_i}"', from=1-2, to=2-1]
	\arrow["{s_i}", from=1-2, to=2-3]
	\arrow[from=2-3, to=3-3]
\end{tikzcd}\]
    \item if one has a span 
\[\begin{tikzcd}[sep=small]
	& {c''} \\
	{c_0'} && {c_1'}
	\arrow["{s_0}"', from=1-2, to=2-1]
	\arrow["{s_1}", from=1-2, to=2-3]
\end{tikzcd}\] equiped with tow arrows $ u_0 : c_0' \rightarrow c$ and $ u_1 : c_1' \rightarrow c$ such that the following image commutes
\[\begin{tikzcd}[sep=small]
	& {f(c'')} \\
	{f(c_0')} && {f(c_1')} \\
	& {f(c)}
	\arrow["{f(s_0)}"', from=1-2, to=2-1]
	\arrow["{f(s_1)}", from=1-2, to=2-3]
	\arrow["{f(u_0)}"', from=2-1, to=3-2]
	\arrow["{f(u_1)}", from=2-3, to=3-2]
\end{tikzcd}\]
then there exists a $J$-cover $ (u_i : c_i'' \rightarrow c'')_{i \in I}$ such that one already has a commutation in $\mathcal{C}$
\[\begin{tikzcd}[sep=small]
	& {c''_i} \\
	{c_0'} && {c_1'} \\
	& c
	\arrow["{s_0u_i}"', from=1-2, to=2-1]
	\arrow["{s_1u_i}", from=1-2, to=2-3]
	\arrow["{u_0}"', from=2-1, to=3-2]
	\arrow["{u_1}", from=2-3, to=3-2]
\end{tikzcd}\]
\end{itemize}
Then, supposing $J$-exactness, applying the first condition of the $J$-cofinality criterion at the pair $ u :f(c') \rightarrow f(c)$ together with the identity of $c'$ provide a $J$-cover $ (s_i :c''_i \rightarrow c')_{i \in I}$ together with a family $ (t_i,u_i)_{i \in I}$ such that the maps $ s_i$ and $u_it_i$ are in the same connected component of $c''_i \downarrow \pi_0^f$, which means that they are connected by some zigzag 
\[\begin{tikzcd}[sep=small]
	& {c_i''} \\
	{c_i'} & \cdots & c'
	\arrow["{t_i}"', from=1-2, to=2-1]
	\arrow[from=1-2, to=2-2]
	\arrow["{s_i}", from=1-2, to=2-3]
	\arrow["{w_i}", from=2-2, to=2-1]
	\arrow["{v_i}"', from=2-2, to=2-3]
\end{tikzcd}\]
made of arrows in the range of $\pi^f_0$, so that the following pasting in $\mathcal{D}$ commutes and ensures an equality of the composites
\[\begin{tikzcd}[sep=small]
	& {f(c_i'')} \\
	{f(c_i')} & \cdots & {f(c')} \\
	& {f(c')}
	\arrow["{f(t_i)}"', from=1-2, to=2-1]
	\arrow[from=1-2, to=2-2]
	\arrow["{f(s_i)}", from=1-2, to=2-3]
	\arrow["{f(u_i)}"', from=2-1, to=3-2]
	\arrow["{f(w_i)}", from=2-2, to=2-1]
	\arrow["{f(v_i)}"', from=2-2, to=2-3]
	\arrow[from=2-2, to=3-2]
	\arrow["u", from=2-3, to=3-2]
\end{tikzcd}\]
Similarly, applying the second condition to a pair of maps $ u_0,u_1 : c' \rightarrow c$ such that $f(u_0)=f(u_1)$ ensures the relative faithfulness condition. For the converse implication, suppose that $f$ is $J$-fully faithful; take $ s: c'' \rightarrow c'$ and $u: f(c') \rightarrow f(c)$; by $J$-fullness, there is a $J$-cover of $(s_i : c'_i \rightarrow c')_{i \in I}$ together with $ v_i : c'_i \rightarrow c$ with $uf(s_i)=f(v_i)$; so one has to take the pullback sieve induced from $(s_i)_{i \in I}$ along $s$ to provide the desired data. Similar argument for the faithfulness condition.
\end{remark}


\begin{remark}
    In particular our characterization meet the characterization of \emph{local sites} of \cite{elephant}[Example C3.6.3(b) and (d)] as those sites $ (\mathcal{C},J)$ such that $ 1$ is $J$-indecomposable. Indeed, asking for a square 
\[\begin{tikzcd}
	1 & 1 \\
	1 & {\mathcal{C}}
	\arrow[equals, from=1-2, to=1-1]
	\arrow["x", from=1-2, to=2-2]
	\arrow[equals, from=2-1, to=1-1]
	\arrow["x"', from=2-1, to=2-2]
\end{tikzcd}\]
to be in $ \Site^\natural$ amounts to asking $ x $ to be a morphism of sites, hence to be flat, which forces $x$ to be the name of the terminal object of $\mathcal{C}$ (it is uniquely determined in the fact that in will induce the unique geometric morphism $ \widehat{C}_J \rightarrow \Set$); but now asking it to be a comorphism forces $1$ to be $J$-irreducible, hence $ (\mathcal{C},J)$ to be a local site, for which $ \widehat{C}_J$ is a local topos. Observe that any functor $ 1 \rightarrow \mathcal{C}$ is fully faithful, so the square above is trivially exact. 
\end{remark}

\begin{division}[Absolutely dense functors]
    A functor $f : \mathcal{C} \rightarrow \mathcal{D}$ is \emph{co-fully faithful} if the square below is exact, which corresponds to the equation on the right
\[\begin{tikzcd}
	{\mathcal{C}} & {\mathcal{D}} \\
	{\mathcal{D}} & {\mathcal{D}}
	\arrow["f", from=1-1, to=1-2]
	\arrow["f"', from=1-1, to=2-1]
	\arrow[equals, from=1-2, to=2-2]
	\arrow[equals, from=2-1, to=2-2]
\end{tikzcd} \hskip1cm \lext_f\rest_f=1 \]
so that in particular $ \rest_f$ is fully-faithful. This means that $ f$ is generating in a very strong way, a reason for which such functor are also be said to be \emph{absolutely dense}, and it is known by \cite{adamek2001laxepi} that they are the lax epimorphisms in $\Cat$. \\

Then if $f$ is also flat, seeing it both as a morphism and comorphism of sites for the trivial topologies, this means that the induced functor $C_f$ has an essential image part given by $\lext_f$, which is also lex whenever $f$ is flat, and moreover $ \rest_f$ is exhibited as fully faithful, which makes $ C_f$ a connected geometric morphism. Moreover the dual square of the identity of $f$, through not necessarily exact, provides for a unit cell exhibiting $ \Sh(f)$ as a right adjoint of $ C_f$ in $\Top$, so $C_f$ is a \emph{totally connected} geometric morphism, with $\Sh(f)$ its terminal section.
\[\begin{tikzcd}
	{\widehat{\mathcal{C}}} & {\widehat{\mathcal{D}}} \\
	{\widehat{\mathcal{D}}} & {\widehat{\mathcal{D}}}
	\arrow["{C_f}"', from=1-1, to=2-1]
	\arrow["{\Sh(f)}"', from=1-2, to=1-1]
	\arrow[equals, from=1-2, to=2-2]
	\arrow[equals, from=2-1, to=2-2]
\end{tikzcd}\]
\end{division}

\begin{proposition}
    Suppose that $f : (\mathcal{C}, J) \rightarrow (\mathcal{D},K)$ is both a morphism and a comorphism of sites. Then $C_f$ is a totally connected geometric morphism if and only if the identity square of $f$ below is $K$-locally exact: 
\[\begin{tikzcd}
	{(\mathcal{C},J)} & {(\mathcal{D},K)} \\
	{(\mathcal{D},K)} & {(\mathcal{D},K)}
	\arrow["f", from=1-1, to=1-2]
	\arrow["f"', from=1-1, to=2-1]
	\arrow[equals, from=1-2, to=2-2]
	\arrow[equals, from=2-1, to=2-2]
\end{tikzcd}\]
\end{proposition}

\begin{proof}
    This is because again we have $ C_f$ left adjoint to $\Sh(f)$ in $\Top$, and now the dual condition of exactness says that $\Sh(f)$ is section of $C_f$, and $ C_f$ is connected for $\rest_f$ is fully faithful. 
\end{proof}

\begin{remark}

Suppose that, for a functor $ f$  the following square is $K$-exact: then the following functor is relatively $K$-cofinal
\[\begin{tikzcd}
	{f\downarrow d} && {\mathcal{D}\downarrow d} \\
	& {\mathcal{D}}
	\arrow["{\iota_d}", from=1-1, to=1-3]
	\arrow["{f\pi_0^f}"', from=1-1, to=2-2]
	\arrow["{\pi_0}", from=1-3, to=2-2]
\end{tikzcd}\]

Unfold, this conditions asserts that for any arrow $ v : d' \rightarrow d$, there is a $K$-cover $ (s_i : d'_i \rightarrow d)$ for which the the category of factorization $ vs_i//f$ at each $i$ is non empty and connected: this means in some sense that this is \emph{locally} a lax epimorphism. 

\end{remark}

\begin{remark}
    We also recover the following results from \cite{elephant}[C3.6.17(b) and (c)]: a site $(\mathcal{C},J)$ induces a totally connected topos $ \widehat{\mathcal{C}}_J$ if and only if $ \mathcal{C}$ is cofiltered and every $J$-cover is connected. Consider the square 
\[\begin{tikzcd}
	{\mathcal{C}} & 1 \\
	1 & 1
	\arrow["{!}", from=1-1, to=1-2]
	\arrow["{!}"', from=1-1, to=2-1]
	\arrow[equals, from=1-2, to=2-2]
	\arrow[equals, from=2-1, to=2-2]
\end{tikzcd}\]
    Asking this square to be in $\Site^\natural$ amounts to asking that $!_\mathcal{C}$ is a morphism of sites, and in particular is flat (which forces $ \mathcal{C}$ to be cofiltered), while it is trivially a comorphism of sites. But now local exactness of the square above ensures that $C_{!_\mathcal{C}}$ is connected and left adjoint to $ \Sh(!_\mathcal{C})$, which ensures total connectedness.  
\end{remark}


\section{$\Site$ as a double category of coalgebras}

One of the main instances of double categories outside the family of equipments is the double category of (strict) algebras, lax morphisms and colax morphisms of algebra for a 2-monad on a 2-category, as introduced in \cite{paregrandismultiple}[Section 5.4]. By formal duality, we also have for any 2-comonad a double category of coalgebras, lax and colax morphisms of coalgebras. In fact this even makes sense for an arbitrary copointed endo-2-functors for a 2-category, for which the notion of coalgebra is still meaningful. In this section, we describe a 2-comonad on $\Cat$ such that sites are coalgebras at least for the underlying copointed endofunctor, while cover-preserving functors are lax morphisms of coalgebras and cover-lifting functors colax morphisms of coalgebras. This gives an explanation of the dichotomy between morphisms and comorphisms. 

\subsection{Double category of coalgebras}

\begin{definition}
    Let $ \mathcal{K}$ be a 2-category and $ T : \mathcal{K} \rightarrow \mathcal{K}$ a copointed endo-2-functor, that is, equipped with a pseudonatural transformation $ \varepsilon: T \Rightarrow 1_\mathcal{K} $. A \emph{coalgebra} for this copointed endofunctor is the data of a pair $ (C,\gamma) $ with  $C$ in $\mathcal{K}$ and $ \alpha$ a section of the copointer 
\[\begin{tikzcd}
	C & TC \\
	& C
	\arrow["\gamma", from=1-1, to=1-2]
	\arrow[Rightarrow, no head, from=1-1, to=2-2]
	\arrow["{\varepsilon_C}", from=1-2, to=2-2]
\end{tikzcd}\]
A \emph{lax} (resp. \emph{colax}) \emph{morphism} of coalgebras $ (C, \gamma) \rightarrow (D,\delta)$ is a pair $(f, \phi) $ with $f : C \rightarrow D$ in $\mathcal{K}$ and $ \phi$ a 2-cell as on the left (resp. on the right)
\[
\begin{tikzcd}
	C & D \\
	TC & TD
	\arrow["f", from=1-1, to=1-2]
	\arrow[""{name=0, anchor=center, inner sep=0}, "\gamma"', from=1-1, to=2-1]
	\arrow[""{name=1, anchor=center, inner sep=0}, "\delta", from=1-2, to=2-2]
	\arrow["Tf"', from=2-1, to=2-2]
	\arrow["\phi"', shorten <=6pt, shorten >=6pt, Rightarrow, from=1, to=0]
\end{tikzcd} \hskip1cm \begin{tikzcd}
	C & D \\
	TC & TD
	\arrow["f", from=1-1, to=1-2]
	\arrow[""{name=0, anchor=center, inner sep=0}, "\gamma"', from=1-1, to=2-1]
	\arrow[""{name=1, anchor=center, inner sep=0}, "\delta", from=1-2, to=2-2]
	\arrow["Tf"', from=2-1, to=2-2]
	\arrow["\phi", shorten <=6pt, shorten >=6pt, Rightarrow, from=0, to=1]
\end{tikzcd}\]
such that moreover one has the following coherence condition (resp. its dual one)
\[\begin{tikzcd}
	C & D \\
	TC & TD \\
	C & D
	\arrow["f", from=1-1, to=1-2]
	\arrow[""{name=0, anchor=center, inner sep=0}, "\gamma"', from=1-1, to=2-1]
	\arrow[""{name=1, anchor=center, inner sep=0}, "\delta", from=1-2, to=2-2]
	\arrow["Tf"{description}, from=2-1, to=2-2]
	\arrow["{\varepsilon_C}"', from=2-1, to=3-1]
	\arrow["{\varepsilon_D}", from=2-2, to=3-2]
	\arrow["f"', from=3-1, to=3-2]
	\arrow["\phi"', shorten <=6pt, shorten >=6pt, Rightarrow, from=1, to=0]
\end{tikzcd} = 1_f \hskip1cm \begin{tikzcd}
	C & D \\
	TC & TD \\
	C & D
	\arrow["f", from=1-1, to=1-2]
	\arrow[""{name=0, anchor=center, inner sep=0}, "\gamma"', from=1-1, to=2-1]
	\arrow[""{name=1, anchor=center, inner sep=0}, "\delta", from=1-2, to=2-2]
	\arrow["Tf"{description}, from=2-1, to=2-2]
	\arrow["{\varepsilon_C}"', from=2-1, to=3-1]
	\arrow["{\varepsilon_D}", from=2-2, to=3-2]
	\arrow["f"', from=3-1, to=3-2]
	\arrow["\phi", shorten <=6pt, shorten >=6pt, Rightarrow, from=0, to=1]
\end{tikzcd}= 1_f 
\]
\end{definition}

\begin{proposition}
    For any copointed endo-2-functor $T$, one can form a double category $ T\hy\coAlg$ of strict coalgebras, lax morphisms as horizontal cells, colax morphisms as vertical cells, and as 2-cell, the lax squares of the form 
\[\begin{tikzcd}
	{(A,\alpha)} & {(C,\gamma)} \\
	{(B, \beta)} & {(D,\delta)}
	\arrow["{(h,\eta)}", from=1-1, to=1-2]
	\arrow[""{name=0, anchor=center, inner sep=0}, "{(f,\phi)}"', from=1-1, to=2-1]
	\arrow[""{name=1, anchor=center, inner sep=0}, "{(g,\psi)}", from=1-2, to=2-2]
	\arrow["{(k, \chi)}"', from=2-1, to=2-2]
	\arrow["\xi"', between={0.2}{0.8}, Rightarrow, from=1, to=0]
\end{tikzcd}\]
    consisting in 2-cells $ \xi : gh \Rightarrow kf$ intertwinning the carriers of the lax and colax morphism structures such that the equality of 2-cell holds
\[\begin{tikzcd}
	& B & TB \\
	A & TA && TD \\
	& C & TC
	\arrow["\beta", from=1-2, to=1-3]
	\arrow[""{name=0, anchor=center, inner sep=0}, "Tk", from=1-3, to=2-4]
	\arrow[""{name=1, anchor=center, inner sep=0}, "f", from=2-1, to=1-2]
	\arrow["\alpha"{description}, from=2-1, to=2-2]
	\arrow[""{name=2, anchor=center, inner sep=0}, "h"', from=2-1, to=3-2]
	\arrow[""{name=3, anchor=center, inner sep=0}, "Tf"{description}, from=2-2, to=1-3]
	\arrow[""{name=4, anchor=center, inner sep=0}, "Th"{description}, from=2-2, to=3-3]
	\arrow["\gamma"', from=3-2, to=3-3]
	\arrow["Tg"', from=3-3, to=2-4]
	\arrow["\eta"{description}, shorten <=6pt, shorten >=6pt, Rightarrow, from=2, to=4]
	\arrow["\phi"{description}, shorten <=6pt, shorten >=6pt, Rightarrow, from=3, to=1]
	\arrow["{T\xi}"{description}, shorten <=7pt, shorten >=7pt, Rightarrow, from=4, to=0]
\end{tikzcd} = 
\begin{tikzcd}
	& B & TB \\
	A && D & TD \\
	& C & TC
	\arrow["\beta", from=1-2, to=1-3]
	\arrow[""{name=0, anchor=center, inner sep=0}, "k"{description}, from=1-2, to=2-3]
	\arrow[""{name=1, anchor=center, inner sep=0}, "Tk", from=1-3, to=2-4]
	\arrow["f", from=2-1, to=1-2]
	\arrow[""{name=2, anchor=center, inner sep=0}, "h"', from=2-1, to=3-2]
	\arrow["\delta"{description}, from=2-3, to=2-4]
	\arrow[""{name=3, anchor=center, inner sep=0}, "g"{description}, from=3-2, to=2-3]
	\arrow["\gamma"', from=3-2, to=3-3]
	\arrow[""{name=4, anchor=center, inner sep=0}, "Tg"', from=3-3, to=2-4]
	\arrow["\chi", shorten <=6pt, shorten >=6pt, Rightarrow, from=0, to=1]
	\arrow["\xi"{description}, shorten <=7pt, shorten >=7pt, Rightarrow, from=2, to=0]
	\arrow["\psi"', shorten <=6pt, shorten >=6pt, Rightarrow, from=4, to=3]
\end{tikzcd}\]
\end{proposition}


Copointed endofunctors are often considered together with an additional structure of \emph{comonad}; here we chose to dissociate those notions as we will have ton consider objects bearing a coalgebra structure only for the copointed endo-2-functor of some 2-comonad. However, it is worth discussing also the comonadic aspects of the question, for which we recall here the axioms of comonad and the additional structure required on their coalgebra.

\begin{definition}
    A \emph{2-comonad} on a 2-category $ \mathcal{K}$ is a copointed endo-2-functor $ (T,\epsilon) $ of $\mathcal{K}$ which is additionally endowed with a \emph{comultiplication} $ \delta : T \Rightarrow TT$ satisfying the following strict natural commutations
\[\begin{tikzcd}
	& T \\
	T & TT & T
	\arrow[equals, from=1-2, to=2-1]
	\arrow["\delta"', Rightarrow, from=1-2, to=2-2]
	\arrow[equals, from=1-2, to=2-3]
	\arrow["{\epsilon_T}", Rightarrow, from=2-2, to=2-1]
	\arrow["{T\epsilon}"', Rightarrow, from=2-2, to=2-3]
\end{tikzcd} \hskip1cm
\begin{tikzcd}
	T & TT \\
	TT & TTT
	\arrow["\delta", Rightarrow, from=1-1, to=1-2]
	\arrow["\delta"', Rightarrow, from=1-1, to=2-1]
	\arrow["{\delta_T}", Rightarrow, from=1-2, to=2-2]
	\arrow["{T\delta}"', Rightarrow, from=2-1, to=2-2]
\end{tikzcd}\]
\end{definition}

\begin{definition}
    Let $T$ be a 2-comonad on $\mathcal{K}$; a \emph{normal lax} (resp. \emph{normal oplax}) coalgebra \emph{for the comonad $(T,\epsilon, \delta)$} is a strict coalgebra $(C, \gamma)$ for the underlying copointed endo-2-functor $(T,\epsilon)$ equipped with an additional 2-cell
\[\begin{tikzcd}
	C & TC \\
	TC & TTC
	\arrow["\gamma", from=1-1, to=1-2]
	\arrow[""{name=0, anchor=center, inner sep=0}, "\gamma"', from=1-1, to=2-1]
	\arrow[""{name=1, anchor=center, inner sep=0}, "{\delta_C}", from=1-2, to=2-2]
	\arrow["{T\gamma}"', from=2-1, to=2-2]
	\arrow["\lambda"', shorten <=6pt, shorten >=6pt, Rightarrow, from=1, to=0]
\end{tikzcd} \hskip1cm \textup{resp.} \hskip1cm
\begin{tikzcd}
	C & TC \\
	TC & TTC
	\arrow["\gamma", from=1-1, to=1-2]
	\arrow[""{name=0, anchor=center, inner sep=0}, "\gamma"', from=1-1, to=2-1]
	\arrow[""{name=1, anchor=center, inner sep=0}, "{\delta_C}", from=1-2, to=2-2]
	\arrow["{T\gamma}"', from=2-1, to=2-2]
	\arrow["\rho", shorten <=6pt, shorten >=6pt, Rightarrow, from=0, to=1]
\end{tikzcd}\]
satisfying some coherences described for instance in \cite{marmolejo1999distributive}. A \emph{strict} coalgebra \emph{for $(T, \epsilon, \delta)$} will be a coalgebra for $(T,\epsilon)$ such that one has a strict commutation
\[\begin{tikzcd}
	C & TC \\
	TC & TTC
	\arrow["\gamma", from=1-1, to=1-2]
	\arrow[""{name=0, anchor=center, inner sep=0}, "\gamma"', from=1-1, to=2-1]
	\arrow[""{name=1, anchor=center, inner sep=0}, "{\delta_C}", from=1-2, to=2-2]
	\arrow["{T\gamma}"', from=2-1, to=2-2]
	\arrow["{=}"{description}, draw=none, from=0, to=1]
\end{tikzcd}\]
\end{definition}

\begin{definition}
    A \emph{lax} (resp. \emph{colax}) morphism of coalgebras for the comonad $ (T,\epsilon, \delta)$ is a lax (resp. colax) morphism of coalgebras $(f,\phi)$ for $(T,\epsilon)$ such that the following coherence is satisfied
\[\begin{tikzcd}
	C & B \\
	TC & TC & TB \\
	& TTC & TTB
	\arrow["f"{description}, from=1-1, to=1-2]
	\arrow["\gamma"', from=1-1, to=2-1]
	\arrow[""{name=0, anchor=center, inner sep=0}, "\gamma"{description}, from=1-1, to=2-2]
	\arrow[""{name=1, anchor=center, inner sep=0}, "\beta", from=1-2, to=2-3]
	\arrow["{T\gamma}"', from=2-1, to=3-2]
	\arrow["{T_f}"{description}, from=2-2, to=2-3]
	\arrow["{\delta_C}"{description}, from=2-2, to=3-2]
	\arrow["{\delta_B}", from=2-3, to=3-3]
	\arrow["TTf"', from=3-2, to=3-3]
	\arrow["\phi"', shorten <=6pt, shorten >=6pt, Rightarrow, from=1, to=0]
\end{tikzcd} = 
\begin{tikzcd}
	C & B \\
	TC & TB & TB \\
	& TTC & TTB
	\arrow["f"{description}, from=1-1, to=1-2]
	\arrow[""{name=0, anchor=center, inner sep=0}, "\gamma"', from=1-1, to=2-1]
	\arrow[""{name=1, anchor=center, inner sep=0}, "\beta"{description}, from=1-2, to=2-2]
	\arrow["\beta", from=1-2, to=2-3]
	\arrow["Tf"{description}, from=2-1, to=2-2]
	\arrow[""{name=2, anchor=center, inner sep=0}, "{T\gamma}"', from=2-1, to=3-2]
	\arrow[""{name=3, anchor=center, inner sep=0}, "{T\beta}"{description}, from=2-2, to=3-3]
	\arrow["{\delta_B}", from=2-3, to=3-3]
	\arrow["TTf"', from=3-2, to=3-3]
	\arrow["\phi"', shorten <=6pt, shorten >=6pt, Rightarrow, from=1, to=0]
	\arrow["{T\phi}"', shorten <=6pt, shorten >=6pt, Rightarrow, from=3, to=2]
\end{tikzcd}\]
\end{definition}

The notion of 2-comonad admits a relaxation into the notion of \emph{lax-comonad}, where one alows both the functors and the transformations to be lax. Lax monads have been considered in several situation, as in \cite{bunge1974coherent} or \cite{Mellies2006LaxAlgebras}, and are related to lax-monoidal structures. 

\begin{definition}{[\cite{bunge1974coherent}, Definition 2.3]}\label{almost}
    A \emph{lax-comonad} is the data of a lax-functor $ T $ endowed with lax-natural transformations $ \pi, \delta $ as in a comonad, together with families of 2-cells 
    \[ \lambda_{C} : T\pi_\mathcal{C} \delta \Rightarrow 1_\mathcal{C} \quad \rho_{C} : 1_\mathcal{C} \Rightarrow \pi_{T\mathcal{C}} \delta \]
    \[ \alpha : T\delta_{\mathcal{C}} \delta_\mathcal{C} \Rightarrow \delta_{T\mathcal{C}} \delta_\mathcal{C} \]
    satisfying a lengthy set of coherences we do not need to reproduce here. We will say a lax-comonad to be \emph{almost a comonad} if \begin{itemize}
        \item its underlying lax functor is actually a 2-functor
        \item its is \emph{normal}, that is, its unit is strictly natural
        \item and the 2-cells $ \lambda, \rho, \alpha$ are all actual strict pointwise equalities. 
    \end{itemize} 
\end{definition}

\begin{remark}
    In other words we only allow the comultiplication to be lax, while all other data are strictly 2-comonadic. Here we are going to consider an almost-comonad whose counit is fibered in posets: as a consequence, the eventual lax or oplax coalgebra structure will consists only in the statement of some inequality rather than the specification of some 2-cell; moreover the distinction between strict and \emph{pseudo}, either for coalgebras or morphisms, will vanish since invertible 2-cells in coherence data will be forcibly equalities. Similarly, we will see that the coherences of lax and colax morphisms relative to the comultiplication will become inequality statement that will be trivially satisfied.
\end{remark}

\subsection{The covers comonad}

We introduce here a canonical copointed endo-2-functor on $\Cat$, which one may see as a \emph{cofree site construction}, sending a category to a category of all possible choices of sets (maybe closed under some conditions) of sieves on objects: then sites will be revealed as coalgebras for this endo-2-functor, and morphisms and comorphisms of sites respectively as lax and colax morphisms. We will see moreover that this endo-2-functor actually bears a comonad structure, though sites will only bear a structure of normal lax coalgebra. We will also consider in the last section a relative version of this construction for which this hindrance vanishes and which exhibits some notion of relative sites as coalgebras for the full comonad structure.

\begin{division}[A suited notion of coverage]\label{our notion of coverage}
Recall that for a Grothendieck coverage on a category $ \mathcal{C}$, the set $J(c)$ of $J$-covering sieves $ S \rightarrowtail \hirayo_c$ on an object $c$ form a poset, which is a subposet $ J(c)$ of the poset of subobjects $\Sub_{\widehat{\mathcal{C}}}\hirayo_c$. However the properties of this poset depends on the axioms of \cref{axiom for topo} we enforce on $J$. Though any assortment should result in a construction satisfying similar results, we will require in this section the following three conditions on $J$ for practicality:\begin{itemize}
    \item it must satisfy maximality so that $\max(c)$ is in $J(c)$
    \item it must satisfy the stability condition, so that for any $ u: d \rightarrow c$ and $S$ in $J(c)$ the pullback sieve must be $ J(d)$
    \item it must be up-closed: if $ S \leq R$ with $S$ in $J(c)$, then $R$ also is in $J(c)$.
\end{itemize}
In the following, we will just call such $J$ simply \emph{coverage}. Then the corresponding $J(c)$ form non empty up-sets containing the maximal element of $\Sub_{\widehat{\mathcal{C}}}(\hirayo_c)$. 
\end{division}

\begin{remark}
    
    Any weaker notion of coverage still can be closed under either of the axioms we require; on the other hand a coverage as defined above can be closed under any of the other axioms listed in \cref{axiom for topo}, and ultimately into a Grothendieck topology. In the case where $ J$ satisfies the locality axiom, it is even a filter, that is, is closed under intersections. We will not require this condition however. In \cref{more general version} we also discuss a relaxed version of the construction in absence of the pullback sieve. Notions of sheaves relative to those notions of coverage are always equivalent at the end. 
\end{remark}

\begin{definition}
    For a poset $P$, we will call a \emph{prefilter} of $P$ a subset $F$ of $P$ that is up-closed and non-empty. For a category $\mathcal{C}$ and $c$ in $\mathcal{C}$ define $ \mathbb{F}_\mathcal{C}(c)$ as the poset of prefilters of $\Sub_{\widehat{\mathcal{C}}} \hirayo_c$.
\end{definition}

\begin{division}
  For a morphism $ u : c \rightarrow c'$ in $\mathcal{C}$, each $ S \rightarrowtail \hirayo_{c'}$ defines a pullback sieve $ u^*S \rightarrow \hirayo_c$, and this defines a morphism of posets 
\[\begin{tikzcd}
	{\Sub_{\widehat{\mathcal{C}}}(c')} & {\Sub_{\widehat{\mathcal{C}}}(c)}
	\arrow["{ u^*}", from=1-1, to=1-2]
\end{tikzcd}\]

If now $ F$ is a prefilter of $\Sub_{\widehat{\mathcal{C}}}(c)  $, the inverse image $ (u^*)^{-1}(F) = \{ R \rightarrowtail \hirayo_{c'} \mid u^*R \in F \} $ is a prefilter of $ \Sub_{\widehat{\mathcal{C}}}(c') $. Hence we have a morphism of posets 
\[\begin{tikzcd}
	{\mathbb{F}_\mathcal{C}(c) } & {\mathbb{F}_\mathcal{C}(c')}
	\arrow["{(u^*)^{-1}}", from=1-1, to=1-2]
\end{tikzcd}\]

Observe that at the level of subposets (rather than prefilters) this morphism as a left adjoint $ u^*[-]$ sending a prefilter (or more generally a subposet) $ F \hookrightarrow \Sub_{\widehat{\mathcal{C}}}\hirayo_{c'}$ to the subposet $ \{ u^*R \mid R \rightarrowtail \hirayo_{c'} \in F' \}$ of $ \Sub_{\widehat{\mathcal{C}}} \hirayo_c$. Indeed one has $ u^*[F'] \leq F$ as subposets if and only if $ F' \leq (u^*)^{-1}(F)$ as prefilters.
\end{division}

\begin{division}[The category $ \mathbb{S}_\mathcal{C}$]

In the following we will denote as $\mathbb{S}(\mathcal{C})$ the category whose:\begin{itemize}
    \item objects are pairs $(c,F)$ with $ c$ an object of $\mathcal{C}$ and $ F$ a prefilter of $ \Sub_{\widehat{\mathcal{C}}}\hirayo_c$,
    \item morphisms $ (c,F) \rightarrow (c',F')$ are morphisms $ u :c \rightarrow c'$ satisfying the property that 
    \[ F' \leq (u^*)^{-1}(F)  \textup{, or equivalently }  u^*[F'] \leq F\]
    which amount to asking that for any $ R \rightarrowtail \hirayo_{c'}$ one has $ u^*R \in F$, as visualized by the condition that $u^*$ restricts between $F'$ and $F$ seen as subposets
\[\begin{tikzcd}
	{\Sub_{\widehat{\mathcal{C}}}\hirayo_{c'}} & {\Sub_{\widehat{\mathcal{C}}}\hirayo_c} \\
	{F'} & F
	\arrow["{u^*}", from=1-1, to=1-2]
	\arrow[hook, from=2-1, to=1-1]
	\arrow[dashed, from=2-1, to=2-2]
	\arrow[hook, from=2-2, to=1-2]
\end{tikzcd}\] 
\end{itemize} 
Remark that this defines correctly a category as those arrows composes:
if one has composable pair $ u : (c, F) \rightarrow (c',F')$ and $(u' : (c', F') \rightarrow (c'', F'')$, then one has $ u : c \rightarrow c'$ and $ u': c' \rightarrow c''$ in $\mathcal{C}$ such that $ F' \leq (u^*)^{-1}(F)$ and $ F'' \leq (u'^*)^{-1}(F')$ so one has $ F'' \leq ((u'u)^{*})^{-1}(F)$. 
\end{division}

\begin{definition}
    We will call the category $ \mathbb{S}\mathcal{C}$ defined above the \emph{cofree site} over $ \mathcal{C}$. 
\end{definition}

\begin{remark}
    This construction is very close to one of the Grothendieck construction, albeit modulo inverting arrows. The prefilter construction defines a poset-valued functor $ \mathbb{F}_\mathcal{C} : \mathcal{C} \rightarrow \Pos$ sending $ c$ to $ \mathbb{F}_\mathcal{C}(c)$ and a morphism $u :c \rightarrow c'$ in $\mathcal{C}$ to the posets morphism $ (u^*)^{-1}: \mathbb{F}_\mathcal{C}(c) \rightarrow \mathbb{F}_\mathcal{C}(c')$. Then one could consider the \emph{lax colimit} of $\mathbb{F}$, which would be opfibered over $ \mathbb{C}$, but its morphisms are those $ u : (c,F) \rightarrow (c',F')$ with $ (u^*)^{-1}(F) \leq F'$, as opposed to what we require in our $\mathbb{S}\mathcal{C}$ in order to express pullback stability. On the other hands, we can consider the \emph{oplax colimit} of $\mathbb{F}$, whose objects are pairs $(c,F)$ with $ F $ in $\mathbb{F}(c)$ and morphisms are $u :(c', F') \rightarrow (c,F)$ consisting of a pasting of the following form
\[\begin{tikzcd}
	& {\mathbb{F}_\mathcal{C}(c')} \\
	1 && {\underset{c \in \mathcal{C}}\oplaxcolim \; \mathbb{F}_\mathcal{C}(c)} \\
	& {\mathbb{F}_\mathcal{C}(c) }
	\arrow[""{name=0, anchor=center, inner sep=0}, "{q_{c'}}", end anchor =170, from=1-2, to=2-3]
	\arrow[""{name=1, anchor=center, inner sep=0}, "{F'}", from=2-1, to=1-2]
	\arrow[""{name=2, anchor=center, inner sep=0}, "F"', from=2-1, to=3-2]
	\arrow["{(u^*)^{-1}}"{description}, from=3-2, to=1-2]
	\arrow[""{name=3, anchor=center, inner sep=0}, "{q_c}"', end anchor =-160, from=3-2, to=2-3]
	\arrow["{q_u}"', shift right=2, shorten <=5pt, shorten >=5pt, Rightarrow, from=0, to=3]
	\arrow["{{\leq}}"', shorten <=4pt, shorten >=4pt, Rightarrow, from=1, to=2]
\end{tikzcd}\]
    where $ q_u : (c', (u^*)^{-1}(F)) \rightarrow (c,F)$ represents the generic cartesian morphism at $u$ while the left cell represent an inequality $ F' \leq (u^*)^{-1}(F)$ as desired; but this oplax colimit is fibered over $ \mathcal{C}^{\op}$ while we need a construction which admits a covariant canonical projection to $ \mathcal{C}$ itself. Hence the actual expression of our construction:
        \[\mathbb{S}\mathcal{C} = (\oplaxcolim_{c \in \mathcal{C}} \mathbb{F}_\mathcal{C}(c))^{\op}\]
    where such data are formally inverted to be represented by a morphism $u : (c,F) \rightarrow (c',F')$, though the left cell still codes for the same inequality.
\end{remark}

\begin{division}[The endofunctor $ \mathbb{S}$]
Moreover this construction is functorial on $\Cat$: for a functor $ f : \mathcal{C} \rightarrow \mathcal{D}$, recall that one can define for each sieve $ S \rightarrowtail \hirayo_c$ in $\mathcal{C}$ the sieve $ f[S] = \{ v:d \rightarrow f(c) \mid  \exists u : c' \rightarrow c \in S \textup{ such that } v \leq f(u) \}$. For a given prefilter $F$ of $ \Sub_{\widehat{\mathcal{C}}}\hirayo_c$, one can consider the prefilter $ \lext_f[F]$ generated from the set of sieves of the form $ f[S]$ for $S$ in $F$. This prefilter contains every sieve $ R \rightarrow \hirayo_{f(c)}$ that contains a sieve of the form $ f[S]$ for $S$ in $F$. 

Then for a functor $f : \mathcal{C} \rightarrow \mathcal{D}$, define $ \mathbb{S}f$ as sending $ (c,F)$ to $(f(c), \lext_f[F])$: then for $ u : (c',F') \rightarrow (c,F)$ and $R \in \lext_f[F]$, there is $ S \in F'$ such that $ u\circ S \leq R$, hence $ f(u) \circ f[S] \leq f[R]$ and $ f[S] \in \lext_f[F']$, so that $ f(u)$ defines a map $ (f(c'), \lext_f[F']) \rightarrow (f(c), \lext_f[F])$. Hence by up-closure, for $ S \leq u^*R$, one has $u^*R \in F'$, and hence $ f(u)^*[f[R]] \in \lext_f[F']$: this ensures that $ f(u)^*[\lext_f[F]] \leq \lext_f[F']$, so that $f(u)$ defines a morphism $ (f(c'), F') \rightarrow (f(c), F)$ in $\mathbb{S}\mathcal{D}$, so that $ \mathbb{S}f$ is a functor.

 For a natural transformation $ \phi : f \Rightarrow g$, we want a natural transformation $ \mathbb{S}\phi : \mathbb{S}f \Rightarrow \mathbb{S}g$, which requires $ \mathbb{S}\phi_{(c,F)} : (f(c), \lext_f[F]) \rightarrow (g(c), \lext_{g}[F])$ to define a morphism in $\mathbb{S}\mathcal{D}$: let be $ S \in \lext_g[F]$; then there is $R$ in $ F$ such that $S$ is generated from the $ g(v)$ with $v \in S$; but then the naturality square of $\phi$ at $v : c' \rightarrow c$ ensures that for any $v \in S$ one has $ \phi_c f(v) = g(v) \phi_{c'} $ so that $\phi_c \circ f[S] \leq R$.  

\end{division}

\begin{division}[In term of restrictions]
Seeing a prefilter $F $ of $\Sub_{\mathcal{C}}(\hirayo_c)$ as a poset map $ F : \Sub_{\mathcal{C}}(\hirayo_c) \rightarrow 2$, the direct and inverse images correspond respectively to the following extension and precomposition functors
\[\begin{tikzcd}
	{\Sub_{\mathcal{C}}(\hirayo_c)} & 2 \\
	{\Sub_{\mathcal{D}}(\hirayo_{f(c)})}
	\arrow["F", from=1-1, to=1-2]
	\arrow["{\lext_f}"', from=1-1, to=2-1]
	\arrow[""{name=0, anchor=center, inner sep=0}, "{\lext_f[F] }"', from=2-1, to=1-2]
	\arrow["\leq"{marking, allow upside down, pos=0.4}, draw=none, from=1-1, to=0]
\end{tikzcd} \hskip1cm 
\begin{tikzcd}
	{\Sub_{\mathcal{C}}(\hirayo_c)} & 2 \\
	{\Sub_{\mathcal{D}}(\hirayo_{f(c)})}
	\arrow["F", from=1-1, to=1-2]
	\arrow["{\rest_f}", from=2-1, to=1-1]
	\arrow["{\rest_f^{-1}(F)}"', from=2-1, to=1-2]
\end{tikzcd}\]
and by generality on extensions along adjoint we know that $ \lan_{\lext_f} \simeq \rest_f^{-1}$, that is
\[  \lext_f[F] \simeq (\rest_f)^{-1}(F) = \{ S \in \Sub_{\mathcal{D}}(\hirayo_{f(c)}) \mid f^{-1}(S) \in F \}  \]
\end{division}

\begin{remark}[More general version]\label{more general version}
    We could have defined alternatively the endo-2-functor $\mathbb{S}$ in a slightly more general manner that will fit for the weakest notion of coverages. For $ \mathcal{C}$ a category, we could define $ \mathbb{S}_0\mathcal{C}$ as having\begin{itemize}
        \item as objects pairs $ (c,F)$ with $c \in \mathcal{C}$ and $ F $ a \emph{set} of sieves on $c$ 
        \item as morphisms $ u : (b,G) \rightarrow (c,F)$ morphisms $u : b \rightarrow c$ in $\mathcal{C}$ such that one has
        \[  \forall R \in F , \exists S \in G, \forall v \in S, uv \in R \]
    \end{itemize}

    In this formulation we only require that for any sieve $R$ in $F$ there is some $S \in G$ contained in $u^*R$, though we do not explicitly stipulate that $u^*R$ is in $G$, which is not automatic as $G$ is not required to be upclosed. However in our context, it will be more practical to enforce upclosure of the sets of sieves and use explicitely pullback sieves. 
\end{remark}

 \begin{division}[Counit and comultiplication]
     The endo-2-functor $\mathbb{S}$ is copointed through the natural projections which will play the role of counits: 
\[\begin{tikzcd}
	{\mathbb{S}\mathcal{C}} & {\mathbb{S}\mathcal{D}} \\
	{\mathcal{C}} & {\mathcal{D}}
	\arrow["{\mathbb{S}f}", from=1-1, to=1-2]
	\arrow["{\pi_\mathcal{C}}"', from=1-1, to=2-1]
	\arrow["{\pi_\mathcal{D}}", from=1-2, to=2-2]
	\arrow["f"', from=2-1, to=2-2]
\end{tikzcd}\]
 
We now want to endow $\mathbb{S}$ with a comultiplication structure to make it a comonad. Beforehand we must give a minimal amount of details about the iterated cofree site $ \mathbb{S}\mathbb{S}\mathcal{C}$: its objects are triples $ (c,F, \mathcal{F})$ with $ (c,F)$ in $\mathbb{S}\mathcal{C}$ and $ \mathcal{F}$ a prefilter of sieves over $ (c,F)$ in $\mathbb{S}\mathcal{C}$; here a sieve will be simply a downard closed collection of arrows $ v : (d,G) \rightarrow (c,F)$, which have to satisfy the condition that $ F \leq (v^*)^{-1}(G)$, that is, that for any sieve $S$ over $c$ in $F$, the pullback sieve $ v^*S$ is in the prefilter $G$. 

Then there is a canonical way to embed $\mathbb{S} \mathcal{C}$ into $ \mathbb{S}\mathbb{S}\mathcal{C}$: for an object $ (c,F)$ in $\mathbb{S}\mathcal{C}$ and $S \in F$, we can define the \emph{stabilizing sieve} in $\mathbb{S} \mathcal{C}$ as the sieve generated in $\mathbb{S}\mathcal{C}$ from preimages of $F$ along morphisms in $S$:
\[  \mathcal{R}_{(S,F)} = \Bigg{\{} v : (d,G) \rightarrow (c,F) \mid \exists u \in S \textup{ such that } 
\begin{tikzcd}[sep=small]
	{(d,G)} && {(c,F)} \\
	& {(b, u^*[F])}
	\arrow["v", from=1-1, to=1-3]
	\arrow["{\exists w}"', from=1-1, to=2-2]
	\arrow["u"', from=2-2, to=1-3]
\end{tikzcd} \Bigg{\}} \]

Then we define, for $ (c,F)$ of $\mathbb{S}\mathcal{C}$, the prefilter generated from stabilizations of sieves in $F$:
\[ \mathcal{F}_{(c,F)} = \uparrow \{ \mathcal{R}_{(S,F)} \mid S \in F  \}\]
in such a way that a sieve on $(c,F)$ is in $ \mathcal{F}_{(c,F)}$ if and only if it contains a stabilizing sieve of the form $ \mathcal{R}_{(S,F)}$ with $S$ in $F$.

\end{division}

\begin{lemma}
    For any category $ \mathcal{C}$, the construction above defines a functor 
\[\begin{tikzcd}
	{\mathbb{S}\mathcal{C}} & {\mathbb{S}\mathbb{S}\mathcal{C}}
	\arrow["{\delta_{\mathcal{C}}}", from=1-1, to=1-2]
\end{tikzcd}\]
sending $ (c,F)$ to $(c,F, \mathcal{F}_{(c,F)})$, that is to its prefilter of stabilizations.
\end{lemma}

\begin{proof}
  We have to check that a map $ v: (d,G) \rightarrow (c,F)$ with $ v^*[F] \leq G$ will induce a map
\[\begin{tikzcd}
	{(d,G, \mathcal{F}_{(d,G)})} & {(c,F, \mathcal{F}_{(c,F)})}
	\arrow["v", from=1-1, to=1-2]
\end{tikzcd}\]
which is to say that $v^*[\mathcal{F}_{(c,F)}] \leq \mathcal{F}_{(d,G)} $. It suffices to check that if $S \in F$, then the image of its stabilizing sieve $ v^*[\mathcal{R}_{(S,F)}]$ contains a stabilizing sieve of the form $\mathcal{R}_{(R,G)} $ for $ R \in G$. The sieve $ v^*[\mathcal{R}_{(S,F)}]$ is the sieve on $(d,G)$ defined as
\[ v^*[\mathcal{R}_{(S,F)}] = \Bigg{\{} w: (a,H) \rightarrow (d,G)  \mid \exists u \in S \textup{ such that } 
\begin{tikzcd}[sep=small]
	{(d,G)} & {(c,F)} \\
	{(a,H)} & {(b,u^*[F])}
	\arrow["v", from=1-1, to=1-2]
	\arrow["w", from=2-1, to=1-1]
	\arrow["{\exists x}"', dashed, from=2-1, to=2-2]
	\arrow["{u \in S}"', from=2-2, to=1-2]
\end{tikzcd} \Bigg{\}} \]
so that in particular, if $w \in v^*[\mathcal{R}_{(S,F)}]$, then in particular the underlying map $w $ is in $ v^*[S]$ for it factorizes through some $ u \in S$, and moreover one has $x^*[u^*[F]] \leq H$. But from the very condition that $ v$ was a map in $ \mathbb{S}\mathcal{C}$ it is stipulated that $ v^*[F] \leq G$, so in particular $ v^*S \in G$, and we can then consider its stabilization $ \mathcal{R}_{(v^*[S],G)}$, which is then in $ \mathcal{F}_{(d,G)}$: but we can see that $ \mathcal{R}_{(v^*[S], G)} \in v^*[\mathcal{R}_{(S,F)}]$: indeed, if $w \in v^*[S]$, then there is some factorization as below with $u \in S$
\[\begin{tikzcd}
	d & c \\
	a & b
	\arrow["v", from=1-1, to=1-2]
	\arrow["w", from=2-1, to=1-1]
	\arrow["x"', from=2-1, to=2-2]
	\arrow["{u \in S}"', from=2-2, to=1-2]
\end{tikzcd}\]
and from $v^*F \leq G$, we have $ x^*[u^*[F]] = w^*[v^*[F]] \leq w^*[G]$ so that the square above lifts to a square in $ \mathbb{S}\mathcal{C}$
\[\begin{tikzcd}
	{(d,G)} & {(c,F)} \\
	{(a,w^*G)} & {(b,u^*[F])}
	\arrow["v", from=1-1, to=1-2]
	\arrow["w", from=2-1, to=1-1]
	\arrow["x"', from=2-1, to=2-2]
	\arrow["{u \in S}"', from=2-2, to=1-2]
\end{tikzcd}\]
which ensures that $w$ is in $v^*[\mathcal{R}_{(S,F)}]$ by its characterization above. This ensures functoriality of the construction $\delta_\mathcal{C}$.
\end{proof}


\begin{lemma}
    The functors $ \delta_\mathcal{C}$ define altogether a lax-natural transformation $\delta : \mathbb{S} \Rightarrow \mathbb{S}\mathbb{S}$.
\end{lemma}

\begin{proof}
    To construct a lax-naturality 2-cell 
\[\begin{tikzcd}
	{\mathbb{S}\mathcal{C}} & {\mathbb{S}\mathcal{D}} \\
	{\mathbb{S}\mathbb{S}\mathcal{C}} & {\mathbb{S}\mathbb{S}\mathcal{D}}
	\arrow["{\mathbb{S}f}", from=1-1, to=1-2]
	\arrow[""{name=0, anchor=center, inner sep=0}, "{\delta_\mathcal{C}}"', from=1-1, to=2-1]
	\arrow[""{name=1, anchor=center, inner sep=0}, "{\delta_\mathcal{D}}", from=1-2, to=2-2]
	\arrow["{\mathbb{S}\mathbb{S}f}"', from=2-1, to=2-2]
	\arrow["{\delta_f}", between={0.2}{0.8}, Rightarrow, nfold, from=0, to=1]
\end{tikzcd}\]
    we must prove the inequality 
    \[ \mathcal{F}_{\mathbb{S}f(c,F)} \leq \lext_{\mathbb{S}f}[\mathcal{F}_{(c,F)}] \]
    which amounts to show that each generating $ \mathcal{R}_{T, \lext_{f}[F]}$ for $T \in \lext_f[F]$ lies in $\lext_{\mathbb{S}f(c,F)]}$. To do so, it suffices to show that there is some $ S \in F$ such that $\lext_{\mathbb{S}f}[ \mathcal{R}_{(S,F)} ] \leq \mathcal{R}_{T, \lext_{f}[F]} $. The sieve $S = \rest_f(T)$ is convenient: indeed one has $ \rest_f(T) \in F$ because one $T \in \lext_f[F]$ iff there is $S \in F$ such that $ \lext_f[S] \leq T$ so that $ S \leq \rest_f[\lext_f[S]] \leq \rest_f[T] $ and $F$ is upclosed. 
    Moreover, for every \(v\in\mathcal R_{(S,F)}\), we have \(\pi_{\mathcal C}(v)\in S=\rest_f(T)\), whence
\[
\pi_{\mathcal D}(\mathbb S f(v))=f(\pi_{\mathcal C}(v))\in T.
\]
Thus \(\mathbb S f(v)\in\mathcal R_{T,\lext_f[F]}\). Since the latter is a sieve, it contains the sieve generated by all such \(\mathbb S f(v)\), giving
\[
\lext_{\mathbb S f}[\mathcal R_{(S,F)}]
\;\leq\;
\mathcal R_{T,\lext_f[F]}.
\]
\end{proof}

\begin{remark}
    The lax-naturality square above cannot be made invertible in general. It is possible to construct indeed a concrete example where the inclusion is strict. The reader may check indeed the strictness of the inequality constitutive of the lax-naturality of $ \delta$ at the inclusion of the walking cospan into the walking square below
\[\begin{tikzcd}
	& a \\
	b & c
	\arrow["u", from=1-2, to=2-2]
	\arrow["v"', from=2-1, to=2-2]
\end{tikzcd} \hookrightarrow
\begin{tikzcd}
	e & a \\
	b & c
	\arrow["r", from=1-1, to=1-2]
	\arrow["s"', from=1-1, to=2-1]
	\arrow["u", from=1-2, to=2-2]
	\arrow["v"', from=2-1, to=2-2]
\end{tikzcd}\]
\end{remark}

\begin{division}[Iteration of counit]
    In the following we are also going to consider the two possible iterations of the free construction at this projection:\begin{itemize}
    \item the functor $ \pi_{\mathbb{S}\mathcal{C}} : \mathbb{S}\mathbb{S}\mathcal{C} \rightarrow \mathbb{S}\mathcal{C}$ sends an object $ (c,F,\mathcal{F})$ to the underlying $ (c,F)$;
    \item the functor $ \mathbb{S}\pi_\mathcal{C} : \mathbb{S}\mathbb{S}\mathcal{C} \rightarrow \mathbb{S}\mathcal{C}$ sends $ (c,F, \mathcal{F})$ to the pair $ (c,\lext_{\pi_\mathcal{C}}(\mathcal{F}))$, where $ \lext_{\pi_\mathcal{C}}(\mathcal{F})$ is the prefilter generated from sieve of the form $ {\pi_\mathcal{C}}[\mathcal{R}]$ for $ \mathcal{R}$ a sieve in $\mathcal{F}$: those image sieve simply are the underlying sieves in $\mathcal{C}$ generated from the underlying sieve of $\mathcal{R}$ over $c$. 
\end{itemize}
\end{division}

\begin{proposition}\label{S is a comonad}
    The data $(\mathbb{S}, \pi, \delta)$ make $ \mathbb{S}$ a normal lax comonad on $\Cat$, and this lax-comonad is almost a comonad (see \cref{almost}).
\end{proposition}

\begin{proof}
    We saw that the comultiplication is only lax-natural. However beside this defect we can see that $\mathbb{S}$ validates all the axioms of comonads. We fix some category $ \mathcal{C}$.\begin{itemize}
        \item For the first axioms: the composite $ \pi_{\mathbb{S}\mathcal{C}}\delta_\mathcal{C} $ is trivially the identity for the underling object of $ \delta_\mathcal{C}(c,F)$ is $(c,F)$ itself: so we do have the first retraction condition
\[\begin{tikzcd}
	{\mathbb{S}\mathbb{S}\mathcal{C}} & {\mathbb{S}\mathcal{C}} \\
	{\mathbb{S}\mathcal{C}}
	\arrow["{\pi_{\mathbb{S}\mathcal{C}}}", from=1-1, to=1-2]
	\arrow["{\delta_\mathcal{C}}", from=2-1, to=1-1]
	\arrow[equals, from=2-1, to=1-2]
\end{tikzcd}\] 
        \item For the second axiom, the composite $ \mathbb{S}\pi_{\mathcal{C}} \delta_\mathcal{C}$ sends $ (c,F)$ to the pair $ (c,\lext_{\pi_{\mathcal{C}}}[\mathcal{F}_{(c,F)}])$; but $ \lext_{\pi_{\mathcal{C}}}[\mathcal{F}_{(c,F)}]$ is the prefilter generated from underlying sieves $ {\pi_\mathcal{C}}[\mathcal{R}_{(S,F)}]$, for $ S \in F$, but $ \mathcal{R}_{(S,F)}$ is precisely defined as the sieve $S$ itself equiped with stabilization of $F$, so its underlying sieve is $S$ itself. So we do have the second retraction condition
    \[\begin{tikzcd}
	{\mathbb{S}\mathbb{S}\mathcal{C}} & {\mathbb{S}\mathcal{C}} \\
	{\mathbb{S}\mathcal{C}}
	\arrow["{\mathbb{S}\pi_{\mathcal{C}}}", from=1-1, to=1-2]
	\arrow["{\delta_\mathcal{C}}", from=2-1, to=1-1]
	\arrow[equals, from=2-1, to=1-2]
    \end{tikzcd}\]
        \item For the last axioms, we compute first $ \delta_{\mathbb{S}\mathcal{C}} (c,F,\mathcal{F}) = ((c,F,\mathcal{F}), \mathcal{F}_{(c,F,\mathcal{F})})$ where  \(\mathcal{F}_{(c,F,\mathcal{F})} \) is the prefilter generated from sieves of the form $ \langle \{ (d,G, v^*[\mathcal{F}]) \rightarrow (c,F,\mathcal{F}) \}, v \in \mathcal{R} \rangle $ with $\mathcal{R} \in \mathcal{F}$. 
        On the other hand, $ \mathbb{S}\delta_{\mathcal{C}}(c,F,\mathcal{F}) = (c,F, \mathcal{F}, \lext_{\delta_\mathcal{C}}(\mathcal{F}))$, where $\lext_{\delta_\mathcal{C}}(\mathcal{F})$ is the prefilter generated from sieves of the form $ {\delta_\mathcal{C}}[\mathcal{R}]$, which are precisely sieves generated from families of the form $\{ (d,G, v^*[\mathcal{F}]) \rightarrow (c,F,\mathcal{F}) , v \in \mathcal{R}\}$. Hence applying those composites to $ \delta_\mathcal{C}(c,F)$ clearly begets the same value for $ \lext_{\delta_\mathcal{C}}(\mathcal{F}_{(c,F)}) = \mathcal{F}_{((c,F), \mathcal{F}_{(c,F)})}$. This ensures the last condition 
\[\begin{tikzcd}
	{\mathbb{S}\mathcal{C}} & {\mathbb{S}\mathbb{S}\mathcal{C}} \\
	{\mathbb{S}\mathbb{S}\mathcal{C}} & {\mathbb{S}\mathbb{S}\mathbb{S}\mathcal{C}}
	\arrow["{\delta_\mathcal{C}}", from=1-1, to=1-2]
	\arrow["{\delta_\mathcal{C}}"', from=1-1, to=2-1]
	\arrow["{\delta_{\mathbb{S}\mathcal{C}}}", from=1-2, to=2-2]
	\arrow["{\mathbb{S}\delta_{\mathcal{C}}}"', from=2-1, to=2-2]
\end{tikzcd}\]
    \end{itemize}    
\end{proof}

\subsection{Sites morphisms and comorphisms as (co)lax morphisms}

We established that $ \mathbb{S}$ is an almost-comonad on $\Cat$. This lax-comonad sends a category to the category of all possible choices of prefilter of subobjects over its objects: as coverages and Grothendieck topologies pick exactly such a prefilter at each object, we can foreseen that they have something to do with coalgebra structures. This is what we discuss in this subsection; as a first step, it is quite easy to see that coverages are the same as sections of the projection maps (forgetting the comultiplication):

\begin{proposition}\label{sites as coalg for the endofunctor}
    A Grothendieck coverage on $\mathcal{C}$ is the same as a coalgebra structure on $\mathcal{C}$ for the underlying copointed endo-2-functor $(\mathbb{S}, \pi)$. 
\end{proposition}

\begin{proof}
    A coalgebra defines a section $ J : \mathcal{C} \rightarrow \mathbb{S}\mathcal{C}$ of $\pi_\mathcal{C}$. This means that for any $c$ the object $J(c)$ is of the form $(c, J(c))$ with $J(c)$ a prefilter of $\Sub_{\widehat{\mathcal{C}}}\hirayo_c$, that is, consist of a collection of sieves on $c$ such that $\hirayo_c$ is in $J(c)$ (that is, the maximal sieve $ \max(c)$ is in $J(c)$), and $J(c)$ is up-closed for inclusion. Moreover functoriality says that for any $ u : c \rightarrow c'$, one has a restriction
\[\begin{tikzcd}
	{J(c)} & {\Sub_{\widehat{\mathcal{C}}}\hirayo_c} \\
	{J(c')} & {\Sub_{\widehat{\mathcal{C}}}\hirayo_{c'}}
	\arrow[hook, from=1-1, to=1-2]
	\arrow[dashed, from=2-1, to=1-1]
	\arrow[hook, from=2-1, to=2-2]
	\arrow["{u^*}"', from=2-2, to=1-2]
\end{tikzcd}\]
expressing that for any $R$ in $J(c')$ the pullback sieve $ u^*R$ is in $J(c)$. This is exactly how we defined coverages in \cref{our notion of coverage}. 
\end{proof}

\begin{remark}
    For the same reason, one can easily see why a general coverage (not necessarily closed under) is a coalgebra structure for the modified version of the cofree site comonad $ \mathbb{S}_0$ described at \cref{more general version}.
\end{remark}

\begin{proposition}\label{inequality}
    For any coverage $ (\mathcal{C},J)$, we have at each $c$ of $\mathcal{C}$ an inequality
    \[ \mathcal{F}_{(c,J(c))} \leq \lext_J[J(c)] \]
    In other word, any coverage defines a structure of normal lax coalgebra for the comonad $(\mathcal{S}, \pi, \delta)$.
\end{proposition}

\begin{proof}
    If $J$ is a coverage, then for any $ u : b \rightarrow c$ in $\mathcal{C}$, one has $ u^*[J(c)] \leq J(b)$ as for any $ S \in J(c)$, the pullback sieve $ u^*S \in J(b)$, and this exactly says that for any $u \in S$, one has a factorization in $\mathbb{S}\mathcal{C}$
\[\begin{tikzcd}
	{(b, J(b))} & {(c,J(c))} \\
	{(b, u^*[J(c)])}
	\arrow["{u }", from=1-1, to=1-2]
	\arrow[dashed, from=1-1, to=2-1]
	\arrow["{u \in S}"', from=2-1, to=1-2]
\end{tikzcd}\]
    This means that the sieve $ \lext_J(S)$ is contained in the stabilizing sieve $ \mathcal{R}_{(S,J(c))}$, which hence must be in the generated prefilter $ \lext_J[J(c)]$. Hence for any $S \in J(c)$, one has $ \mathcal{R}_{(S,J(c))} \in \lext_J[J(c)]$, and as $\mathcal{F}_{(c,J(c))}$ is generated from those stabilizing prefilters, one has the desired inequality.\end{proof} 

\begin{division}
    The missing coherence condition in \cref{sites as coalg for the endofunctor} is the comultiplication condition that the following square commutes strictly
\[\begin{tikzcd}
	{\mathcal{C}} & {\mathbb{S}\mathcal{C}} \\
	{\mathbb{S}\mathcal{C}} & {\mathbb{S}\mathbb{S}\mathcal{C}}
	\arrow["J", from=1-1, to=1-2]
	\arrow["J"', from=1-1, to=2-1]
	\arrow["{\delta_\mathcal{C}}", from=1-2, to=2-2]
	\arrow["{\mathbb{S}J}"', from=2-1, to=2-2]
\end{tikzcd}\]
which amounts to having at any $c$ in $\mathcal{C}$ an equality 
\[   \mathcal{F}_{(c,J(c))} =  \lext_J[J(c)] \]
We saw at \cref{inequality} that the left to right inequality is always true for a coverage. On the other hand, asking that $ \lext_J[J(c)] \leq \mathcal{F}_{(c,J(c))}$ amounts to asking that for any $R \in J(c)$, the sieve $ \lext_J(R)$ contains a generating sieve of $ \mathcal{F}_{(c,J(c))}$, that is, there is a sieve $S \in J(c)$ such that $ \mathcal{R}_{(S,J(c))  }\leq \lext_J(R) $; hence for such a sieve $S$ one will have that for any $ v : d \rightarrow c $ in $S$, there will be a $u : b \rightarrow c$ in $R$ and a factorization $ v = uw$ such that one has a factorization in $ \mathbb{S}\mathcal{C}$
\[\begin{tikzcd}
	{(d,v^*[J(c)])} & {(c,J(c))} \\
	{(b,J(b))}
	\arrow["{v \in S}", from=1-1, to=1-2]
	\arrow["w"', from=1-1, to=2-1]
	\arrow["{u \in R}"', from=2-1, to=1-2]
\end{tikzcd}\]
expressing that $ w^*[J(b)] \leq v^*[J(c)]$ on $d$. This later condition says that for any $T \in J(b)$, $w^*T \in v^*[J(c)]$, which is to say that there is some $Q \in J(c)$ such that $ v^*Q \leq w^*T$, that is, that for any $f $ on $d $ such that $vf$ factorizes through some $f' \in Q$, there is some $g $ such that $wy$ factorizes through some $g' \in T$ and a factorization as below
\[\begin{tikzcd}[sep=large]
	& {d'} & {c'} \\
	{d''} & d & c \\
	{b'} & b
	\arrow["x", from=1-2, to=1-3]
	\arrow["{\exists z}"', dashed, from=1-2, to=2-1]
	\arrow["f"{description}, from=1-2, to=2-2]
	\arrow["{f' \in Q}", from=1-3, to=2-3]
	\arrow["g"{description}, from=2-1, to=2-2]
	\arrow["y"', from=2-1, to=3-1]
	\arrow["{v \in S}"{description}, from=2-2, to=2-3]
	\arrow["w"{description}, from=2-2, to=3-2]
	\arrow["{g' \in T}"', from=3-1, to=3-2]
	\arrow["{u \in R}"', from=3-2, to=2-3]
\end{tikzcd}\]

  The condition we obtain by unfolding the coherence condition can itself be seen as a local form of transitivity: it says that for any $J$-coviering sieve $R$ there is another test $J$-covering sieve $S$ such that for any choice of $J$-covering sieves on domains of maps in $R$, the composite sieve will ``look'' covering locally in $S$. This condition is rather strong and does not seems to be satisfied by Grothendieck coverages nor topologies, as it forces somewhat the prefilter to be related in a very tight manner through pullback. 
\end{division}

At the level of morphisms and comorphisms however there is a good correspondence with morphisms of coalgebras, without complications arising from the comultiplication:

\begin{proposition}
   Let $(\mathcal{C},J)$, $(\mathcal{D}, K)$ be coalgebras for the copointed endofunctor $(\mathbb{S},\pi)$ and $ f: \mathcal{C} \rightarrow \mathcal{D}$ a functor: \begin{itemize}
        \item $f$ supports a lax morphism of coalgebras if and only if it is cover-preserving;
        \item $f$ supports a colax morphism of coalgebras if and only if it is cover-lifting.
    \end{itemize}
\end{proposition}

\begin{proof}
For the first item, suppose that one has a 2-cell 
\[\begin{tikzcd}
	{\mathcal{C}} & {\mathcal{D}} \\
	{\mathbb{S}\mathcal{C}} & {\mathbb{S}\mathcal{D}}
	\arrow["f", from=1-1, to=1-2]
	\arrow[""{name=0, anchor=center, inner sep=0}, "J"', from=1-1, to=2-1]
	\arrow[""{name=1, anchor=center, inner sep=0}, "K", from=1-2, to=2-2]
	\arrow["{\mathbb{S}f}"', from=2-1, to=2-2]
	\arrow["\phi"', shorten <=6pt, shorten >=6pt, Rightarrow, from=1, to=0]
\end{tikzcd}\]
This consists for each $c$ in $\mathcal{C}$ of a morphism $\phi_c : (f(c), K(f(c))) \rightarrow (f(c), \lext_f[J(c)])$; moreover, from the axiom of retraction
\[\begin{tikzcd}
	{\mathcal{C}} & {\mathcal{D}} \\
	{\mathbb{S}\mathcal{C}} & {\mathbb{S}\mathcal{D}} \\
	{\mathcal{C}} & {\mathcal{D}}
	\arrow["f", from=1-1, to=1-2]
	\arrow[""{name=0, anchor=center, inner sep=0}, "J"', from=1-1, to=2-1]
	\arrow[""{name=1, anchor=center, inner sep=0}, "K", from=1-2, to=2-2]
	\arrow["{\mathbb{S}f}"{description}, from=2-1, to=2-2]
	\arrow["{\pi_\mathcal{C}}"', from=2-1, to=3-1]
	\arrow["{\pi_\mathcal{D}}", from=2-2, to=3-2]
	\arrow["f"', from=3-1, to=3-2]
	\arrow["\phi"', shorten <=6pt, shorten >=6pt, Rightarrow, from=1, to=0]
\end{tikzcd} = 
\begin{tikzcd}
	{\mathcal{C}} & {\mathcal{D}} \\
	{\mathcal{C}} & {\mathcal{D}}
	\arrow["f", from=1-1, to=1-2]
	\arrow[""{name=0, anchor=center, inner sep=0}, equals, from=1-1, to=2-1]
	\arrow[""{name=1, anchor=center, inner sep=0}, equals, from=1-2, to=2-2]
	\arrow["f"', from=2-1, to=2-2]
	\arrow["\begin{array}{c} 1_f \\ = \end{array}"{description}, draw=none, from=0, to=1]
\end{tikzcd}\]
we know that the underlying morphism $\pi_{\mathcal{D}}(\phi_c)$ in $ \mathcal{D}$ is the identity of $ f(c)$, so it has to live in the fiber $\mathbb{F}_{\mathcal{D}}(f(c))$; but as $ \mathbb{S}\mathcal{D}$ has posetal fibers, the existence of such a 2-cell amounts to an inequality between subobjects
\[ \lext_f[J(c)] \leq K(f(c))  \]
which means that if a sieve $ R$ contains the image of a sieve of $J(c)$, then it is in $K(f(c))$. Hence the sieve generated from the arrows of the form $ f(u): f(c') \rightarrow f(c)$ for $u \in S$ with $S$ in $J(c)$ is in particular in $K(f(c))$, which makes $ f$ a cover preserving functor $ (\mathcal{C},J) \rightarrow (\mathcal{D},K)$.  

For the second item, suppose that one has a 2-cell 
\[\begin{tikzcd}
	{\mathcal{C}} & {\mathcal{D}} \\
	{\mathbb{S}\mathcal{C}} & {\mathbb{S}\mathcal{D}}
	\arrow["f", from=1-1, to=1-2]
	\arrow[""{name=0, anchor=center, inner sep=0}, "J"', from=1-1, to=2-1]
	\arrow[""{name=1, anchor=center, inner sep=0}, "K", from=1-2, to=2-2]
	\arrow["{\mathbb{S}f}"', from=2-1, to=2-2]
	\arrow["\phi", shorten <=6pt, shorten >=6pt, Rightarrow, from=0, to=1]
\end{tikzcd}\]
This consists for each $c$ in $\mathcal{C}$ of a morphism $\phi_c :(f(c), \lext_f[J(c)]) \Rightarrow  (f(c), K(f(c))) $, whose underlying map in $\mathcal{D}$ is the identity of $f(c)$ by the same argument as above: but as $ \mathbb{S}\mathcal{D}$ has posetal fibers, the existence of such a 2-cell amounts to an inequality between subobjects
\[ K(f(c)) \leq \lext_f[J(c)] \]
which means any sieve $ R$ in $K(f(c))$ contains a sieve of the form $ \lext_fS$ for some $ S $ in $J(c)$, so that there is a sieve $ S$ on $c$ such that the sieve generated from the $f(u)$ with $u$ in $S$ is contained in $R$, which exactly says that $f$ is a cover-lifting functor.
\end{proof}

\begin{remark}
   {We only have a double category whose horizontal morphisms are \emph{cover-preserving} functors rather than morphisms of sites, for the structure involved there will have no special relation with flatness. Though this does not change anything for vertical morphisms as comorphisms are not required to be flat, the resulting double category of coalgebra has more horizontal 1-cells than $ \Site^\natural$ and the sheafification double functor does not properly extends to it as flatness is crucial to restrict between sheaf topoi. }
\end{remark}

\subsection{Indexed version}

The two defects surrounding the comultiplication of $ \mathbb{S}$ can be fixed by considering an indexed version that properly saturate filters regarding to their projections in a basis.

\begin{division}
    Consider the strict slice $ \Cat/\mathcal{C}$ of functors over $\mathcal{C}$ together with strict triangles between them. For a functor $ p : \mathcal{D} \rightarrow \mathcal{C}$, we say that a prefilter $F$ of sieves on $d$ in $\mathcal{D}$ is $p$-saturated if for any two sieves $R,S$ such that $ R \in F$ and $p[R] = p[S]$ as sieves on $p(d)$ in $\mathcal{C}$ then $S \in F$: in other words $F$ does not distinguishes between sieves that have the same projection in $\mathcal{C}$. Equivalently those are prefilters $F$ such that $ F = \rest_{p}\lext_{p}[F]$ in the sense that if $R$ has $ p[R] = \langle \{ p(u) : p(d') \rightarrow p(d) \mid u \in R \} \rangle $ is in $ \lext_p(F)$ then $ R \in F$. In particular $p$-saturated prefilters are entirely determined by their projection: if $ F,F'$ are two $p$-saturated prefilters with $ \lext_p[F] = \lext_p[F']$ then $F = F'$.
    \end{division}

    \begin{division}
    For any set of sieves $F$ on $d$ one can take its $p$-saturation $ \overline{F}$ containing all $S$ such that $ p(S) = p(R)$ for some $R \in F$. Then define $ \mathbb{S}_\mathcal{C}(\mathcal{D},p)$ as having\begin{itemize}
        \item as objects pairs $(d,F)$ with $d $ an object of $\mathcal{D}$ and $ F$ a $p$-saturated filter of sieves on $d$;
        \item as morphisms $(d',F') \rightarrow (d,F)$ any $u : d' \rightarrow d$ in $\mathcal{D}$ such that $p(u)^*[\lext_p[F]] \leq \lext_p[F']$. 
    \end{itemize}

    Equivalently, having in mind that $ p$-saturated prefilters are determined by their projection, one can define $ \mathbb{S}_\mathcal{C}(\mathcal{D},p)$ as having $ (d, K)$ with $K$ a prefilter on $p(d)$ in $\mathcal{C}$, and as morphisms $ u : (d', K') \rightarrow (d, K)$ those $ u : d' \rightarrow d$ such that $ u^*[K] \leq K'$. 
    
    Then we have a projection $ \pi_{(\mathcal{D},p)} : \mathbb{S}_\mathcal{C}(\mathcal{D},p) \rightarrow \mathcal{D}$ sending $(d,F)$ to $d$: this functor is a fibration, as for any $ u:d' \rightarrow d$ with $F$ on $d$, the morphism $u : (d,\overline{u^*F}) \rightarrow F$ is cartesian. 
    We can then define the following composite together with its composition 2-cell to exhibit $\pi_{(\mathcal{D},p)}$ as a 1-cell of $\Cat/\mathcal{C}$
\[\begin{tikzcd}
	{\mathbb{S}_\mathcal{C}{(\mathcal{D},p)}} && {\mathcal{D}} \\
	& {\mathcal{C}}
	\arrow["{\pi_{(\mathcal{D},p)}}", from=1-1, to=1-3]
	\arrow["{p\pi_{(\mathcal{D},p)}}"', from=1-1, to=2-2]
	\arrow["p", from=1-3, to=2-2]
\end{tikzcd}\]
\end{division}

\begin{division}
    Then for $ m : (\mathcal{D},p ) \rightarrow (\mathcal{B},q)$ a strict triangle over $\mathcal{C}$, then define $ \mathbb{S}_\mathcal{C}(m) : \mathbb{S}_\mathcal{C}(\mathcal{D},p ) \rightarrow \mathbb{S}_\mathcal{C}(\mathcal{B},q)$ sending $ (d,F)$ to the pair $ (m(d), \overline{\lext_m[F]})$. In term of filters on the basis, we have an even more explicit description: as $ qm =p$, one has $ \lext_q[\lext_m[F]] = \lext_p[F]$, so $\lext_m$ does not change the projection of prefilters, and one can describe $ \mathbb{S}_\mathcal{C}(m)(d, K) = (m(d), K)$. Hence we have a copointed 2-endofunctor $ \mathbb{S}_\mathcal{C}: \Cat/\mathcal{C} \rightarrow \Cat/\mathcal{C}$ over $ \mathcal{C}$ with $ \pi$ as the copointer.

    Then for each $(\mathcal{D},p )$ the category $ \mathbb{S}_\mathcal{C}\mathbb{S}_\mathcal{C}(\mathcal{D},p )$ contains all triples $(d,F,\mathcal{F})$ where $ F$ is a $p$-saturated prefilter of sieves on $d$, and $ \mathcal{F}$ is a $p\pi_{(\mathcal{D},p )}$-saturated prefilter of sieves in $\mathbb{S}_\mathcal{C}(\mathcal{D},p )$ over $ (d,F)$. But again those large saturated prefilters can now be more easily be seen as filters over $ p \pi_{(\mathcal{D},p)}(d, K)$ in the basis $ \mathcal{C}$. 
    
    Then define a comultiplication as follows: for any $(\mathcal{D},p )$ define $ \mathcal{F}_{(c,F)}$ as the saturated prefilter generated from sieves of the form $ \mathcal{R}_{S,F}$. Then observe that equivalently, by saturation, $ \mathcal{F}_{(c,F)}$ contains all sieves generated from families of the form $\{ u : (d',F') \rightarrow (d,F) \mid u^*F \leq F', u \in R \}$ for $R \in F$, that are, sieves $ \mathcal{R}$ such that $ {\pi_{(\mathcal{D},p )}}[ \mathcal{R}] \in F$. Indeed, for such a $ \mathcal{R}$, one has then $ p \pi_{(\mathcal{D},p )}(\mathcal{R}) \in \lext_p[F] = \lext_{p}\lext_{\pi_{(\mathcal{D},p)}}(\mathcal{F}_{(c,F)})$, so $ \mathcal{R}$ is in $ \mathcal{F}_{(c,F)}$ by saturation. But observe that though more complicated in appearance, $ \mathcal{F}_{(d,F)}$ has the same projection than $F$ itself in $\mathcal{C}$, so it can be represented by the same prefilter. That is, in terms of filters of the basis, $ \delta_{(\mathcal{D}, p)} $ sends simply $ (d, K)$ to the duplication $ ((d, K), K)$. 

    Then we have a functor $ \delta_{(\mathcal{D},p)} : \mathbb{S}_\mathcal{C}(\mathcal{D},p) \rightarrow \mathbb{S}_\mathcal{C}\mathbb{S}_\mathcal{C}(\mathcal{D},p)$ sending $ (c,F)$ to $ (c,F, \mathcal{F}_{(c,F)})$. This defines a strict triangle 
\[\begin{tikzcd}
	{\mathbb{S}_\mathcal{C}(\mathcal{D},p)} && {\mathbb{S}_\mathcal{C}\mathbb{S}_\mathcal{C}(\mathcal{D},p)} \\
	& {\mathcal{C}}
	\arrow["{ \delta_{(\mathcal{D},p)}}", from=1-1, to=1-3]
	\arrow["{p\pi_{(\mathcal{D},p)}}"', from=1-1, to=2-2]
	\arrow["{p\pi_{(\mathcal{D},p)}\pi_{\mathbb{S}(\mathcal{D},p)}}", from=1-3, to=2-2]
\end{tikzcd}\]
\end{division}

\begin{proposition}
    For any category $ \mathcal{C}$, the data of $(\mathbb{S}_\mathcal{C}, \pi, \delta)$ define a 2-comonad on $ \Cat/\mathcal{C}$.
\end{proposition}

\begin{proof}
    By its very construction mirroring $\mathbb{S}$, the arguments of \cref{S is a comonad} that made $ \mathbb{S}$ an almost-comonad all apply here. The only new thing to prove is that as defined the comultiplication $\delta_{\mathcal{C}}$ of $ \mathbb{S}_{\mathcal{C}}$ is strictly natural rather than only lax-natural; this can be seen in terms of filters on the basis. For $ m: (\mathcal{D},p) \rightarrow (\mathcal{B},q)$ recall that $\mathbb{S}_\mathcal{C}(m)(d, K) = (m(d), K)$; hence at $ (d, K)$ in $\mathbb{S}_\mathcal{C}(\mathcal{D},p)$ one has 
    \begin{align*}
        \mathbb{S}_\mathcal{C}\mathbb{S}_\mathcal{C}(m) \delta_{(\mathcal{D},p)}(d, K) &= \mathbb{S}_\mathcal{C}\mathbb{S}_\mathcal{C}(m)((d,K), K) \\
        &=\mathbb{S}_\mathcal{C}(\mathbb{S}_\mathcal{C}(m)(d,K), K) \\
        &= ((m(d), K), K) \\
        &= \delta_{(\mathcal{B}, q)}(m(d), K) \\
        &= \delta_{(\mathcal{B}, q)} \mathbb{S}_\mathcal{C}(m)(d,K)
    \end{align*}
\end{proof}

\begin{definition}
    For $ p : \mathcal{D} \rightarrow \mathcal{C}$ a fibration, a \emph{$p$-coverage} on $\mathcal{D}$ will be a coverage such that for each $d$ in $\mathcal{D}$, the prefilter $J(d)$ is $p$-saturated.
    
\end{definition}

\begin{proposition}
    For $p : \mathcal{D} \rightarrow\mathcal{C}$, a coalgebra structure for the comonad $  (\mathbb{S}_\mathcal{C}, \pi, \delta)$ on $(\mathcal{D},p)$ is the same as a $p$-coverage on $\mathcal{D}$.
\end{proposition}

\begin{proof}
     For $p : \mathcal{D} \rightarrow\mathcal{C}$, a $ \mathbb{S}_\mathcal{C}$-coalgebra structure on $(\mathcal{D},p)$ is a strict triangle $ J : (\mathcal{D}, p) \rightarrow \mathbb{S}_\mathcal{C}(\mathcal{D},p)$ satisfying the equations
\[\begin{tikzcd}
	{\mathcal{D}} & {\mathbb{S}_\mathcal{C}(\mathcal{D},p)} & {\mathcal{D}} \\
	& {\mathcal{C}}
	\arrow["J", from=1-1, to=1-2]
	\arrow["p"', from=1-1, to=2-2]
	\arrow["{\pi_{(\mathcal{D},p)}}", from=1-2, to=1-3]
	\arrow["{p\pi_{(\mathcal{D},p)}}"{description}, from=1-2, to=2-2]
	\arrow["p", from=1-3, to=2-2]
\end{tikzcd} = 
\begin{tikzcd}
	{\mathcal{D}} && {\mathcal{D}} \\
	& {\mathcal{C}}
	\arrow[equals, from=1-1, to=1-3]
	\arrow["p"', from=1-1, to=2-2]
	\arrow["p", from=1-3, to=2-2]
\end{tikzcd}  \hskip1cm 
\begin{tikzcd}
	{(\mathcal{D},p)} & {\mathbb{S}_\mathcal{C}(\mathcal{D},p)} \\
	{\mathbb{S}_\mathcal{C}(\mathcal{D},p)} & {\mathbb{S}_\mathcal{C}\mathbb{S}_\mathcal{C}(\mathcal{D},p)}
	\arrow["J", from=1-1, to=1-2]
	\arrow[""{name=0, anchor=center, inner sep=0}, "J"', from=1-1, to=2-1]
	\arrow[""{name=1, anchor=center, inner sep=0}, "{\delta_{(\mathcal{D},p)}}", from=1-2, to=2-2]
	\arrow["{\mathbb{S}_\mathcal{C}J}"', from=2-1, to=2-2]
	\arrow["{=}"{description}, draw=none, from=0, to=1]
\end{tikzcd}\]
Being a section of $ \pi_{(\mathcal{D},p)}$ implies to be also a $p$-local section of $p\pi_{(\mathcal{D},p)}$, but we also know that it amounts for $J$ to be a coverage such that $ J(d)$ is a $p$-saturated prefilter for each $d$. Conversely we see that any coverage on $\mathcal{D}$ that consists in $p$-saturated prefilters statisfies the first condition. Now let us show that the comultiplication condition entails from the saturation: for any $d$ in $\mathcal{D}$, we know already even before saturation that $\mathcal{F}_{(c,J(c)} \leq \lext_J[J(c)]$; if now one has $\mathcal{R} \in \lext_J[J(d)]$, there is $R \in J(d)$ such that $${p\pi_{(\mathcal{D},p)}}[\mathcal{R}] = {p\pi_{(\mathcal{D},p)}}[{J}[R]] $$
But we always have that $${p\pi_{(\mathcal{D},p)}}[{J}[R]] = {p\pi_{(\mathcal{D},p)}}[\mathcal{R}_{R,J(c)}] = p[R] $$
so $ J[R] \in \mathcal{F}_{(c,J(c))}$ and so does $\mathcal{R}$. Hence one has $ \lext_{J}[J(c)] = \mathcal{F}_{(c,J(c))}$.
\end{proof}

\begin{lemma}
    Any prefilter $ F$ on an object $c$ in $\mathcal{C}$ is $ 1_\mathcal{C}$-saturated. 
\end{lemma}

\begin{theorem}
    A coverage $ J$ on $\mathcal{C} $ is the same as a coalgebra structure for the comonad $ (\mathbb{S}_\mathcal{C}, \pi, \delta)$ on $ 1_\mathcal{C}$. 
\end{theorem}

\begin{proof}
    Let $ J$ be a coverage on $\mathcal{C}$; then it is automatically $1_\mathcal{C}$-saturated. Moreover, from our previous consideration on the comonad $ \mathbb{S}$, $J$ defines a section of $\pi_{(\mathcal{C}, 1_\mathcal{C})}$. Moreover we know that for any $c$, one has $ \mathcal{F}_{(c,J(c))} \leq \lext_J[J(c)]$ even before saturation, hence after saturation; it suffices to show they have the same saturation: but both have the same reduce
    \[ \lext_{\pi_{(\mathcal{C},1_\mathcal{C})}}[\lext_J[J(c)]] = J(c) = \lext_{\pi_{(\mathcal{C},1_\mathcal{C})}}[\mathcal{F}_{(c,J(c))}] \]
   whence a commutation
\[\begin{tikzcd}
	{(\mathcal{C},1_{\mathcal{C}})} & {\mathbb{S}(\mathcal{C},1_{\mathcal{C}})} \\
	{\mathbb{S}_\mathcal{C}(\mathcal{C},1_{\mathcal{C}})} & {\mathbb{S}_\mathcal{C}\mathbb{S}(\mathcal{C},1_{\mathcal{C}})}
	\arrow["J", from=1-1, to=1-2]
	\arrow[""{name=0, anchor=center, inner sep=0}, "J"', from=1-1, to=2-1]
	\arrow[""{name=1, anchor=center, inner sep=0}, "{ \delta_{(\mathcal{C},1_{\mathcal{C}})}}", from=1-2, to=2-2]
	\arrow["{\lext_J}"', from=2-1, to=2-2]
	\arrow["{=}"{description}, draw=none, from=0, to=1]
\end{tikzcd}\]
\end{proof}

The virtue of this indexed construction is to resolve the dissymmetry of our \cref{inequality}, as now for any $ p : \mathcal{D} \rightarrow \mathcal{C}$, $ p$-coverages are exactly strict coalgebras structures on $\mathcal{D}$, rather than lax normal coalgebras as in the unindexed case.

\section*{Acknowledgement}

The second author is grateful to Nathanael Arkor for the reference \cite{paregrandismultiple} concerning double categories of algebras, as well to Ryuya Hora for some technical remarks and discussions about the site aspects in the characterization of local and totally connected geometric morphisms, and also to Igor Bakovi\'c for communicating us insightful remarks, in particular the observation that the comultiplication of the cofree-site construction is only lax-normal. Both authors are also grateful to the ACS referee for their careful report.

\section*{Funding}

This work was realized without any funding.

\printbibliography
\newpage

\textsc{Olivia Caramello} 

\vspace{0.2cm}
{\small \textsc{Dipartimento di Scienza e Alta Tecnologia, Universit\`a degli Studi dell'Insubria, via Valleggio 11, 22100 Como, Italy.}\\
	\emph{E-mail address:} \texttt{olivia.caramello@uninsubria.it}}

\vspace{0.2cm}

{\small \textsc{Istituto Grothendieck,
		Corso Statuto 24, 12084 Mondovì, Italy.}\\
	\emph{E-mail address:} \texttt{olivia.caramello@igrothendieck.org}}

\vspace{0.6cm}

\textsc{Axel Osmond} 

\vspace{0.2cm}
{\small \textsc{Istituto Grothendieck,
		Corso Statuto 24, 12084 Mondovì, Italy, and Paris, France.}\\
	\emph{E-mail address:} \texttt{axelosmond@orange.fr}}

\end{document}